\documentclass[11pt]{amsart}
\usepackage{amsmath}
\usepackage{amssymb}
\usepackage{amscd}
\usepackage{amsthm}
\usepackage[mathscr]{eucal}
\usepackage{graphicx}
\usepackage{fullpage}
\usepackage[all]{xy}
\usepackage{xr-hyper}

\usepackage{hyperref}
\newtheorem{theorem}{Theorem}[section]
\newtheorem{proposition}[theorem]{Proposition}
\newtheorem{lemma}[theorem]{Lemma}

\newtheorem{corollary}[theorem]{Corollary}
\newtheorem{conjecture}[theorem]{Conjecture}
\newtheorem{assumption}[theorem]{Assumption}
\newtheorem{D}[theorem]{Definition}
\newenvironment{definition}{\begin{D} \rm }{\end{D}}
\newtheorem{R}[theorem]{Remark}
\newenvironment{remark}{\begin{R}\rm }{\end{R}}
\newtheorem{E}[theorem]{Example}

\def\Zee{\mathbb{Z}}
\def\Q{\mathbb{Q}}
\def\Ar{\mathbb{R}}
\def\Cee{\mathbb{C}}
\def\Pee{\mathbb{P}}
\def\Id{\operatorname{Id}}
\def\Ker{\operatorname{Ker}}
\def\Hom{\operatorname{Hom}}
\def\Ext{\operatorname{Ext}}

\def\Aut{\operatorname{Aut}}
\def\Pic{\operatorname{Pic}}
\def\im{\operatorname{Im}}

\def\Gr{\operatorname{Gr}}

\def\scrO{\mathcal{O}}
\def\spcheck{^{\vee}}

\def\hY{\widetilde{Y}}

\title{Limiting  mixed Hodge structures associated to I-surfaces with simple elliptic singularities}

\begin{document}
\author[R. Friedman]{Robert Friedman}
\address{Columbia University, Department of Mathematics, New York, NY 10027}
\email{rf@math.columbia.edu}
\author[P. Griffiths]{Phillip   Griffiths}
\address{Institute for Advanced Study, Princeton, NJ  08540 and IMSA, The University of Miami, Coral Gables, FL 33146}
\email{pg@ias.edu}

\begin{abstract}  An  I-surface  $X$ is a surface of general type with $K_X^2 =1$ and $p_g(X) =2$. This paper studies the asymptotic behavior of the period map for I-surfaces acquiring simple elliptic singularities. First we describe the relationship between the deformation theory of such surfaces and their $d$-semistable models. Next we analyze the mixed Hodge structures on the $d$-semistable models, the corresponding limiting mixed Hodge structures, and the monodromy.  There are $6$ possible boundary strata for which the relevant limiting mixed Hodge structures satisfy:  $\dim W_1 = 4$, and hence $W_2/W_1$ is of pure type $(1,1)$. We show that, in each case, the nilpotent orbit of limiting mixed Hodge structures determines the boundary stratum and prove a global Torelli theorem for one such  stratum. 
\end{abstract}

\bibliographystyle{amsalpha}
\maketitle

\section*{Introduction}

An \textsl{I-surface} $X$ is a surface of general type with $K_X^2 =1$ and $p_g(X) =2$. Recent research has centered on  classifying  the  various strata of the KSBA compactification of the moduli space of such surfaces. In particular, in \cite{FPR1}, \cite{FPR2},  Franciosi, Pardini, and Rollenske have given  a description of the boundary strata with Gorenstein singularities, and Coughlan, Franciosi, Pardini, and Rollenske \cite{CFPR} have given a qualitative  description of the corresponding limiting mixed Hodge structures (including  some non-Gorenstein cases). Beyond the intrinsic interest in these results, it is natural to ask if the study of singular I-surfaces has applications to the period map and in particular if it can be used to prove a Torelli theorem for I-surfaces. There is a  well established strategy for approaching such questions (see for example \cite{FriedmanTorelli}), which very roughly goes as follows: 

\begin{enumerate}
\item Given a moduli space $\mathcal{M}$ of algebraic varieties satisfying some weak version of local Torelli and a corresponding period map $\Phi\colon \mathcal{M} \to \Gamma\backslash D$ with image $Z$, partially compactify the map $\Phi$. Explicitly, this means: Find
\begin{enumerate}
\item A partial compactification  $\overline{\mathcal{M}}$, typically the KSBA compactification;
\item  A partial compactification  $\overline{Z}$ of the image of the period map $\Phi$, typically such that  (up to   finite group actions) $\overline{Z}$ is smooth and $\overline{Z}-Z$  is   a divisor with normal crossings; 
\item An explicit  blowup $\widetilde{\mathcal{M}} \to \overline{\mathcal{M}}$ which is an isomorphism over $\mathcal{M}$ such that the period map $\Phi$ extends to a  holomorphic map $\widetilde{\Phi} \colon \widetilde{\mathcal{M}}\to \overline{Z}$. Typically,  $\widetilde{\mathcal{M}}$ is an orbifold and   the complement of $\mathcal{M}$ in $\widetilde{\mathcal{M}}$ is an orbifold  divisor with normal crossings. In this case, the hope is that, for $x\in \widetilde{\mathcal{M}}$, $\widetilde{\Phi} (x)$   records both the nilpotent orbit of the  limiting mixed Hodge structure of the corresponding degeneration as well as  other data related to monodromy.
\end{enumerate} 
\item Identify a convenient stratum $\mathcal{S}$ of $\widetilde{\mathcal{M}} -\mathcal{M}$ and check that $\widetilde{\Phi} ^{-1}(\widetilde{\Phi} (\mathcal{S})) =\mathcal{S}$.  In practice, if $\widetilde{\mathcal{M}} -\mathcal{M}$ is a divisor with normal crossings, then $\mathcal{S}$ will be a connected component of the natural stratification of $\widetilde{\mathcal{M}} -\mathcal{M}$. 
\item Prove a Torelli theorem for the variation of mixed Hodge structure  corresponding to the morphism $\widetilde{\Phi} |\mathcal{S}$.
\item Compute the differential of $\widetilde{\Phi} $ at a point $x\in \mathcal{S}$ in the directions normal to $\mathcal{S}$.
\item Prove that the map $\widetilde{\Phi} \colon \widetilde{\mathcal{M}}  \to \overline{Z}$ is proper.
\end{enumerate}

We now describe this program in more detail for I-surfaces. For Step (1a), it is natural first to allow I-surfaces with either rational double point (RDP) or simple elliptic singularities or some combination of these. We will ignore the issue of RDP singularities as this is not a major problem. As for simple elliptic singularities, allowing these types of singularities as well as RDP singularities leads to a partial compactification $\overline{\mathcal{M}}$ of $\mathcal{M}$ which is an open subset of the KSBA compactification.  Unfortunately,  the  complement  $\overline{\mathcal{M}}-  \mathcal{M}$ has fairly high codimension. However, there is a standard procedure for replacing this compactification by one such that the boundary has normal crossings, by replacing a normal surface with simple elliptic singularities by a surface with simple normal crossings of a special type ($d$-semistable in the terminology of \cite{FriedmanSmoothings}). These models are well-suited to understanding the corresponding limiting mixed Hodge structures. As for Steps (1b) and (1c), we will not attempt here to find a partial  Hodge-theoretic compactification of the image of the period map. While various such compactifications have been proposed, the overall picture has not yet been fully clarified. In our situation, there is a partial compactification due to Deng-Robles \cite{DengRobles}, based on work of Kato-Nakayama-Usui \cite{KU}, \cite{KNU}.
Recent work in progress of Deng-Robles is likely to establish the existence of a partial compactification which has some of the necessary properties  (and in fact in much greater generality). See Conjectures~\ref{extendedperiodmapthm} and \ref{extendedperiodmapthm2} for more details. 
%\footnote{It is possible that in some circumstances, for example if $\overline{\mathcal{M}}-\mathcal{M}$ is a smooth divisor,  this issue can be finessed, so that one can prove a Torelli theorem without actually constructing $\overline{Z}$.}

For Step (2), the simplest strata would be I-surfaces $Y$ with one simple elliptic singularity corresponding to limiting mixed Hodge structures of type $\lozenge_{0,1}$ in the notation of \cite[Example 4.9]{Robles} (see Definition~\ref{definelozenge}). A $d$-semistable model  for such a surface is of the form $X_0= \hY \amalg_DZ$, where $\hY$ is the minimal resolution of $Y$, $D$ is the exceptional divisor in $\hY$ and hence is an elliptic curve, $\hY \amalg_DZ$ denoted the normal crossing surface obtained by gluing $\hY$ and $Z$ along an isomorphism from $D$ to an isomorphic curve in $Z$, and $\hY$, $Z$ satisfy: Either 
\begin{enumerate}
\item[\rm(i)] $\hY$ is a minimal elliptic surface with $\kappa =1$,  $p_g(\hY) = 1$,   a  multiple fiber of multiplicity $2$, and a smooth bisection $D$  with $D^2 =-1$, and $Z$ is a del Pezzo surface of degree one containing $D$ as an anticanonical divisor, or
\item[\rm(ii)] $\hY$ is the blowup of a $K3$ surface $Y_0$ at a point $p$, $D$ is the proper transform of a  curve $\Gamma$  on $Y_0$ of arithmetic genus $2$ with a node  or cusp at $p$, and $Z$ is a del Pezzo surface of degree two containing $D$ as an anticanonical divisor.
\end{enumerate}
If one chooses to work with such surfaces, the information in the limiting mixed Hodge structure is captured by an extension of pure Hodge structures of the form
$$0 \to W_1 \to W_2 \to W_2/W_1 \to 0.$$
Here $W_1$ is the pure weight one Hodge structure corresponding to $D$, and $W_2/W_1$ is essentially the Hodge structure on $H^2(\hY)$. 
For example, in Case (ii) above, the limiting mixed Hodge structure determines the polarized Hodge structure of $Y_0$. Hence, by the global Torelli theorem for $K3$ surfaces, it determines $Y_0$ together with the linear system  $|\Gamma|$, whose general member is a smooth  curve of genus $2$. If $Y_0$ is otherwise generic, there are only finitely many singular curves in $|\Gamma|$ whose normalizations are an elliptic curve with a given $j$-invariant, so the period map is generically finite-to-one. However, it seems hard to go further and use the full information of the limiting mixed Hodge structure to determine the pair $(\hY, D)$, or equivalently to determine exactly which singular element of $|\Gamma|$ corresponds to the limiting mixed Hodge structure. Concretely, since $H^{2,0}(\hY)$ has dimension one, mixed Hodge structures which are given as extensions of the above type have not yet been understood geometrically.  The problem is roughly equivalent to finding a way to exploit the information contained in certain points  of the intermediate Jacobian of $\hY \times D$. 

Thus it is natural to  further degenerate to I-surfaces $Y$ with two or three simple elliptic singularities, corresponding to limiting mixed Hodge structures of type $\lozenge_{0,2}$.  For instance, in the case of two  simple  elliptic singularities, there are $4$ posssible boundary strata and the   $d$-semistable models for such surfaces are of the form $X_0= \hY \amalg_{D_1}Z_1\amalg_{D_2}Z_2$. In this case, $\hY$, the minimal resolution of $Y$, is either a rational surface or an Enriques surface blown up at a point, $D_1$ and $D_2$ are the exceptional divisors of the morphism $\hY \to Y$,   $Z_1$, $Z_2$ are del Pezzo surfaces of degree $-D_i^2$ containing $D_i$ as an anticanonical divisor, and $Z_i$ is glued to $\hY$ along an isomorphism from $D_i\subseteq \hY$ to $D_i \subseteq Z_i$. Thus $H^{2,0}(\hY) = H^{2,0}(Z_i) = 0$, and the corresponding extension of pure Hodge structures is of ``classical" type: the weight two part  $W_2/W_1$ is pure of type $(1,1)$, $J^0W_1$ is essentially $JD_1 \oplus JD_2$, where $JD_i =\Pic^0D_i$ is the Jacobian of the elliptic curve $D_i$, and the extension can be understood  geometrically via Carlson's theory of extensions of mixed Hodge structures \cite{Carlson0}, \cite{Carlson2}, \cite{Carlson3}.

For I-surfaces of type $\lozenge_{0,2}$, denoting by $(W_2/W_1)_0$ the corresponding graded piece of the polarized limiting mixed Hodge structure, the lattice $(W_2/W_1)_0$ is an even negative definite unimodular lattice $\Lambda$ of rank $24$, and such lattices have been classified by Niemeyer \cite{Niemeier}. Up to isomorphism, there are $24$ possibilities and they are classified by the root system $R(\Lambda)$ formed by the vectors of square $-2$ in $\Lambda$. (Here and throughout the paper we use the convention in algebraic geometry that root systems are \textbf{negative definite}.) While \emph{a priori} all of these could appear in limiting mixed Hodge structures of I-surfaces, it turns out that, in the cases studied in this paper, only two of them are relevant: the lattice $\Lambda$ for  which $R(\Lambda)=E_8+ E_8 + E_8$ and the lattice $\Lambda$ for  which $R(\Lambda)=E_7+ E_7 + D_{10}$.  (See   Remark~\ref{otherrootsystems} for a more detailed discussion.) More precisely, for three of the four strata of I-surfaces with two  simple  elliptic singularities, $R(\Lambda) = E_8+ E_8 + E_8$, and for the remaining  case,    $R(\Lambda) =E_7+ E_7 + D_{10}$. Although it is not essential for the overall strategy, we show that the Hodge theory in the three  cases where $R(\Lambda) = E_8+ E_8 + E_8$ look different from each other. For  the case $R(\Lambda) =E_7+ E_7 + D_{10}$, we look at the limiting mixed Hodge structure in detail. However, we are not quite able to establish a Torelli theorem for this stratum. The situation is analogous to  that coming from the $\lozenge_{0,1}$ case when $\hY$ is a blown up $K3$ surface: it is not too hard to show that, for $\hY$ generic,  there are only finitely many possibilities for the the pair $(\hY, D_1 + D_2)$ or the corresponding $d$-semistable model $X_0$, but it is not clear how to fully determine $(\hY, D_1 + D_2)$ or equivalently $X_0$.  The difficulty is similar to that encountered  by Engel-Greer-Ward in their study of the global Torelli problem for elliptic surfaces with $p_g=1$ over an elliptic base \cite{EGW}. 

Instead, we consider I-surfaces with three  simple  elliptic singularities, where there are $2$ possible boundary strata. For such surfaces, the geometry and relevant mixed Hodge structures are much simpler. The price that must be paid for this simplification is that the corresponding image in the compactification of period space is more complicated. The picture is formally analogous to the period map for a stable curve $C = C_1\cup C_2$, where $C_1$ and $C_2$ are two smooth curves meeting at three points. The limiting mixed Hodge structures of such curves look  like those of an irreducible stable curve with two nodes, but the monodromy and the image in a ``reasonable" compactification of period space such as the second Voronoi compactification are quite different. The moral here for weight two Hodge structures is that, unlike the $K3$ case or the case considered in \cite{EGW}, the fact that the period domain is not Hermitian symmetric offers an advantage by giving us extra room to maneuver. We show that, for both strata for which the singular I-surface has three simple  elliptic singularities, the limiting mixed Hodge structure  is of type $\lozenge_{0,2}$, the lattice $\Lambda$ satisfies $R(\Lambda) = E_8+ E_8 + E_8$, and the limiting mixed Hodge structures together with associated monodromy data distinguish the two strata for these surfaces both from the strata for I-surfaces  with two  simple  elliptic singularities and from each other. Finally, we prove a global Torelli theorem for one of the two strata. 

As for Step (4), this   amounts to understanding  the arithmetic of the relevant Picard-Lefschetz transformations. For two simple elliptic singularities, this can be done  by analogy with \cite[3.8]{FriedmanTorelli}. For three simple elliptic singularities, the monodromy picture is slightly more complicated  and is described in detail in \S\ref{PLfmlas}.

This leaves Step (5), the question of properness. Some other I-surfaces with   limiting mixed Hodge structures of type $\lozenge_{0,2}$ are described in \cite{CFPR}, and it would be interesting to work out the corresponding semistable reductions,  lattice theory, weight one Hodge structure on $W_1$,  and monodromy. Typically, however, these examples seem to have either a different weight one piece $W_1$ and/or a different lattice $\Lambda$. At the moment, however, a complete classification of all possible such I-surfaces seems out of reach, since it seems very difficult to enumerate the various strata of non-Gorenstein I-surfaces. Even for Gorenstein I-surfaces with minimally elliptic (i.e.\ elliptic Gorenstein) singularities which are worse than simple elliptic or cusp singularities,  the problem of understanding the possible limiting mixed Hodge structures and the relevant monodromy seems very daunting. Of course, such surfaces will not be semi log canonical and thus will not appear in the KSBA compactification. 

The contents of this paper are as follows. Section 1 collects facts about I-surfaces, both smooth and with simple elliptic singularities, as well as standard results about lattices and anticanonical pairs. Section 2 analyzes the $d$-semistable versions of I-surfaces with simple elliptic singularities and compares their deformation theory to that of the original surfaces. Section 3 studies the mixed Hodge structures on the $d$-semistable models. Section 4 deals with the corresponding limiting mixed Hodge structures of smoothings and gives a detailed description of the monodromy for the case of $3$ simple elliptic singularities. Finally, in Section 5, we show that the limiting mixed Hodge structure and the monodromy distinguish the various boundary strata, and prove a global Torelli theorem for one such stratum,   namely the one consisting of I-surfaces with $3$ simple elliptic singularities, all of   multiplicity one. In an appendix, we outline the theory of simultaneous log resolutions for simple elliptic singularities. While there is nothing here which is not  well-known to specialists, it is hard to find explicit statements in the literature for the results that are used in this paper. 

\subsection*{Acknowledgements} It is a pleasure to thank Colleen Robles and Haohua Deng for many discussions during which they patiently explained to us the subtleties involved in constructing partial compactifications of the images of period maps. We would also like to thank Johan de Jong, Mark Green, Radu Laza, John Morgan, and Nick Shepherd-Barron, for conversations and correspondence on matters related to this paper and spanning several decades.

\section{Preliminaries} 

\subsection{The smooth case}\label{smoothcase}    If $X$ is an I-surface, from $c_1^2(X) + c_2(X) = 12\chi(\scrO_X)$, $c_2(X) = 35$, $b_2(X) = 33$ and $H^2_0(X;\Zee)= [K_X]^\perp \subseteq H^2(X;\Zee)$ is an even unimodular lattice of signature $(4, 28)$. Thus as a lattice 
$$H^2_0(X;\Zee) \cong U^4\oplus (\Lambda_{E_8})^3,$$ where $U$ is the hyperbolic (rank two even unimodular) lattice and $\Lambda_{E_8}$ is the (negative definite) $E_8$ lattice, i.e.\ the unique negative definite even unimodular lattice of rank $8$.

\subsection{Geometry of certain I-surfaces with elliptic singularities}
As outlined in the introduction, we are concerned with minimal resolutions of Gorenstein I-surfaces with two or three simple elliptic singularities. The following gives the rough classification of such surfaces, due to   Franciosi-Pardini-Rollenske \cite[Theorem 4.1]{FPR1}, \cite[Proposition 4.3]{FPR2}:

\begin{theorem} Suppose that $Y$ is a normal Gorenstein I-surface  with two or three simple elliptic singularities, with minimal resolution $\pi\colon \hY\to Y$, and that $D_i$ are the exceptional fibers of $\pi$, $1\le i \le k$, with $m_i = -D_i^2$. Then $\hY$ satisfies exactly one of the following:
\begin{enumerate} 
\item[\rm(i)] $\hY$ is the blowup of an Enriques surface $Y_0$ at one point, $k=2$, and the exceptional fibers $D_i$ are the proper transforms of two smooth elliptic curves on $Y_0$ meeting transversally at a point. In this case, $m_1 = m_2 = 1$. 
\item[\rm(ii)] $\hY$ is a rational surface, $k=2$,  and the possibilities for the pair $(m_1, m_2)$ up to order are $(2,2)$, $(2,1)$, and $(1,1)$. 
\item[\rm(iii)] $\hY$ is the blowup of an  elliptic ruled surface over a base elliptic curve $B$, $k=3$, and the possibilities for the pair $(m_1, m_2, m_3)$ up to order are $(2,1,1)$, and $(1,1,1)$.  \qed
\end{enumerate}
\end{theorem} 

We will refer to these cases as the \textsl{Enriques}, \textsl{rational}, or \textsl{elliptic ruled} cases respectively (and will call $\hY$ an Enriques surface even though it is not minimal). In the first two cases, we will call the unordered pair $(m_1, m_2)$ the \textsl{multiplicities} of $\hY$, and similarly in the elliptic ruled case for the unordered triple $(m_1, m_2, m_3)$. By convention, however, we order the $m_i$ so that $m_1 \ge m_2 \ge \cdots$. 

For the rest of this section, we assume that $\hY$ is the minimal resolution of a normal Gorenstein I-surface $Y$ with two or three simple elliptic singularities. We will need more precise information in  the rational and elliptic ruled cases. For the rational case, we have the following:

\begin{theorem}\label{22case} \cite[Example 5.2, Proposition 5.3]{FG24}  Suppose that $\hY$ is a rational surface  with multiplicities  $(m_1, m_2)=(2,2)$. Then there exists an exceptional curve $C$ on $\hY$ such that $C\cdot D_1 =2$ and  $C\cdot D_2 =0$. If $\rho_1\colon \hY \to Y_0^{(1)}$ is the contraction of $C$, then $(Y_0^{(1)}, D_2)$ is an anticanonical pair (i.e.\ $D_2 = -K_{Y_0^{(1)}}$) and the image $\overline{D}_1$ is a curve of arithmetic genus $2$ with a node or a cusp at the image of $C$. The linear system $|\overline{D}_1|$ defines a morphism $\nu\colon Y_0^{(1)} \to \mathbb{F}_1$. 

Assume that $Y$ has no RDP singularities. If $\sigma_0$ is the negative section of $\mathbb{F}_1$ and $f$ is a fiber, the morphism $\nu$   is a double cover branched along a smooth element in $|2\sigma_0 + 6f|$, $\overline{D}_1 \equiv \nu^*(\sigma_0 + f)$, and $D_1 = \nu^*\sigma_0$.  \qed
\end{theorem}  

\begin{theorem}\label{21case} \cite[Proposition 5.10]{FG24}  Suppose that $\hY$ is a rational surface with multiplicities  $(m_1, m_2)=(2,1)$.  Then there exists a rational elliptic surface $X$ with a multiple fiber $F$ of multiplicity $2$, a smooth  nonmultiple fiber $G$, and a smooth bisection $\Gamma$, such that $D_1$ and $D_2$ are the proper transforms of $G$, $\Gamma$ respectively, $\Gamma$ meets $G$ transversally at two  points $p_1$ and $p_2$, and   $\hY$ is the blowup of $X$ at $p_1$ and $p_2$. Finally, in the generic case, i.e.\ if $X$ does not contain a smooth rational curve of self-intersection $-2$, then  $\Gamma \in |F+E|$ for some exceptional curve $E$. \qed
\end{theorem}  

Before dealing with the case of three simple elliptic singularities, we make the following definition:

\begin{definition}\label{defY0}  Let $B$ be an elliptic curve and let $Y_0$ be the unique elliptic ruled surface over $B$ with invariant $e=-1$. Thus $Y_0= \Pee(W)$, where $W$ is a rank $2$ stable bundle over $B$ with $\deg \det W=1$.  In particular, there exists a one parameter family of sections $\sigma$ of $Y_0$ with $\sigma^2 =1$, and there exist exactly three disjoint bisections $\Gamma$ of $Y_0$ with $\Gamma^2=0$ and $\sigma \cdot \Gamma =1$, indexed by the $2$-torsion points of $B$.
\end{definition} 

\begin{theorem}\label{211case} \cite[Proposition 7.1]{FG24}   Suppose that $\hY$ has three elliptic singularities with multiplicities $(m_2, m_2, m_3) = (2,1,1)$. Then $\hY$ is the blowup of $Y_0$ at three points. More precisely, there exist:  
\begin{enumerate} 
\item[\rm(i)] A smooth bisection $\Gamma$  of $Y_0 \to B$  with $\Gamma^2 =0$; 
\item[\rm(ii)] Sections $\sigma_2$ and $\sigma_3$   with  $\sigma_i^2 = \sigma_i\cdot \sigma_j = \sigma_i\cdot \Gamma = 1$, with   $\Gamma \cap \sigma_i = p_i$, $i=2,3$, and $\sigma_2\cap \sigma_3 = p_1$, where the points $p_1, p_2, p_3$ are all distinct,
\end{enumerate}   such that $\hY$ is the blowup of $Y_0$ at $p_1, p_2, p_3$,   $D_1$ is the proper transform of $\Gamma$ and  $D_2$  and $D_3$ are the proper transforms of $\sigma_2$, $\sigma_3$ respectively.  \qed
\end{theorem}  

\begin{theorem}\label{111case} \cite[Proposition 7.2]{FG24}  Suppose that $\hY$ has three elliptic singularities with multiplicities $(m_2, m_2, m_3) = (1,1,1)$.   Then $\hY$ is the blowup of $Y_0$ at two points. More precisely, there exist: 
\begin{enumerate} 
\item[\rm(i)] Two smooth bisections $\Gamma_1$,    $\Gamma_2$ of $Y_0 \to B$  with $\Gamma_i^2 =\Gamma_1 \cdot \Gamma_2 = 0$; 
\item[\rm(ii)] A smooth section  $\sigma_3$   with  $\sigma_3^2 =   \sigma_3\cdot \Gamma = 1$, with   $\Gamma_i \cap \sigma_3 = p_i$, $i=1,2$;
\end{enumerate}   
such that $\hY$ is the blowup of $Y_0$ at $p_1, p_2$,   $D_1$ and  $D_2$ are the proper transforms of $\Gamma_1$, $\Gamma_2$ respectively,  and $D_3$ is the proper transform  of   $\sigma_3$. \qed
\end{theorem} 

\begin{remark} In the elliptic ruled case, $Y$ has no RDPs, and its only singular points are the three simple elliptic singularities. Hence $\omega_Y$ is ample. 
\end{remark}  

\subsection{Definition of the Hodge diamond}

We recall the following notation from \cite[Example 4.9]{Robles}:

\begin{definition}\label{definelozenge} Let $H$ be an $N$-polarized mixed Hodge structure whose weight filtration has the form
$$W_0 \subseteq W_1 \subseteq W_2 \subseteq W_3 \subseteq W_4,$$
where $W_i/W_{i-1}$ is a pure (effective) Hodge structure of weight $i$. We say that $H$ is \textsl{of Hodge type $\lozenge_{r, s}$}  if $\dim\Gr_F^0\Gr^W_0H=h^{0,0} = r$ and $\dim\Gr_F^1\Gr^W_1H=h^{1,0} = s$, i.e.\  if $(W_0)^{0,0}$ has dimension $r$ and $(W_1/W_0)^{1,0}$ and  $(W_1/W_0)^{0,1}$ both have dimension $s$.  In particular, the statement that $H$ is  of Hodge type $\lozenge_{0, s}$ means that $W_0 =0$ and $\dim W_1 = 2s$. 
\end{definition}

\begin{remark} If $H$ is the $N$-polarized limiting mixed Hodge structure of Hodge type $\lozenge_{0, 2}$ corresponding to the primitive cohomology of a degeneration of I-surfaces, then $W_1$ is a primitive rank $4$ isotropic subspace of $H$.  Thus  $(W_2/W_1)_0$ is a negative definite even unimodular lattice of rank $24$.  Compare also Proposition~\ref{latticeLambda} and Corollary~\ref{elllatticeLambda}.
\end{remark} 

\subsection{Some relevant lattice theory}\label{latticessect}

Let $\Lambda$ be a negative definite even unimodular lattice of rank $24$ and let
$$R(\Lambda) =\{\alpha \in \Lambda: \alpha^2 = -2\}.$$
Let $\Lambda_R$ be the sublattice of $\Lambda$ spanned by $R(\Lambda)$.  More generally for an arbitrary simply laced  root system  or generalized root system $R$, we denote by $\Lambda_R$   the  lattice with a  $\Zee$-basis   given by  a set of simple roots of $R$, and  with intersection form specified by the Dynkin diagram of $R$:   each root $\alpha$ satisfies $\alpha^2 = -2$ and, for simple roots,  $\alpha_i \cdot \alpha_j =1$ if $\alpha_i $ and $\alpha_j$ are connected by an edge in the Dynkin diagram and $\alpha_i \cdot \alpha_j =0$ otherwise.

The following is then an easy consequence of  the  fundamental classification results  of Niemeier \cite{Niemeier}:

\begin{theorem}\label{Niem}  \begin{enumerate} \item[\rm(i)] Two negative definite even unimodular lattices $\Lambda_1$ and $\Lambda_2$  of rank $24$ are isomorphic $\iff$ the root systems $R(\Lambda_1)$ and $R(\Lambda_2)$ are isomorphic root systems.
\item[\rm(ii)]  Given a   negative definite even unimodular lattice $\Lambda$ of rank $24$, either $R(\Lambda)$ has rank $24$ or $\Lambda$ is the Leech lattice and $R(\Lambda) =\emptyset$. 
\item[\rm(iii)] If $R(\Lambda)$ contains a sub-root system of type $E_7 + E_7$, then $R(\Lambda)$ is either $E_8+ E_8 + E_8$ or $E_7+ E_7 + D_{10}$.  In the first case, $\Lambda_R = \Lambda$ is the orthogonal direct sum of three copies of the $E_8$ lattice $\Lambda_{E_8}$. In the second case, $\Lambda_R$ has index $4$ in $\Lambda$. \qed
\end{enumerate}
\end{theorem} 

\begin{remark} Clearly, the Weyl group $W(R(\Lambda))$ is a normal subgroup of the group $\Aut \Lambda$ of integral isometries of $\Lambda$. Conversely, if $A(R(\Lambda))$ is the group of automorphisms of the root system $R(\Lambda)$, then there is a homomorphism $\Aut \Lambda \to A(R(\Lambda))$ which is injective as long as $R(\Lambda)$ spans $\Lambda$ over $\Q$, i.e.\ as long as $R(\Lambda)$ has rank $24$ or equivalently $\Lambda$ is not the Leech lattice. If $R(\Lambda)= E_8+ E_8+ E_8$, then $\Aut \Lambda  \cong A(R(\Lambda))$, $W(R(\Lambda)) \cong W(E_8)^3$,  and  $A(R(\Lambda))/W(R(\Lambda)) \cong \mathfrak{S}_3$, the symmetric group on three letters. If $R(\Lambda)= E_7+ E_7 + D_{10}$, then $A(R(\Lambda))/W(R(\Lambda)) \cong (\Zee/2\Zee)\times (\Zee/2\Zee)$, where one generator switches the two $E_7$ factors and is the identity on $D_{10}$, and a second generator is the outer automorphism of the $D_{10}$ root system and is the identity on the two copies of $E_8$.
\end{remark} 

\begin{remark}\label{otherrootsystems}  For the cases considered in this paper ($Y$ a Gorenstein I-surface with simple elliptic singularities), only the root systems $E_8+ E_8 + E_8$ and $E_7+ E_7 + D_{10}$ arise. This is to be expected  by the discussion in \S3, because we only consider surfaces with two or three simple elliptic singularities of multiplicities one or two. There is no compelling reason why any other of the $24$ negative definite even unimodular lattices  of rank $24$ should necessarily arise, but it is certainly natural to speculate that  some of the other root systems appear in the examples of \cite{CFPR} or otherwise in degenerations of I-surfaces of Hodge type $\lozenge_{0,2}$.  This issue is related to the structure of the monodromy group of the universal family of I-surfaces (smooth or with rational double points). More precisely, given any negative definite even unimodular lattice $\Lambda$  of rank $24$, we have $H^2_0(S;\Zee) \cong U^4\oplus  \Lambda$ by the classification of indefinite even unimodular lattices, and it is easy to construct unipotent integral isometries $T$ of $H^2_0(S;\Zee)$, such that, if $N =\log T$, then the associated weight filtration $W_{\bullet}H^2_0(S;\Zee)$ of $H^2_0(S;\Zee)$  satisfies: $\Gr_W^2H^2_0(S;\Zee) \cong \Lambda$. Hence, if the monodromy group is of finite index in the automorphism group of the lattice $H^2_0(S;\Zee)$, then it will contain such elements. (In unpublished notes, the second author along with Green, Laza and Robles has outlined an argument for a similar result for the monodromy group of H-surfaces \cite{GGLR}, and it is likely that these methods will also handle the case of I-surfaces.)  The question is then whether such $T$ can arise as the monodromy of a one parameter degeneration of I-surfaces. In particular, it would interesting to construct a degeneration of I-surfaces of type $\lozenge_{0, 2}$ for which $\Lambda$ is the Leech lattice.
\end{remark}

\subsection{Torelli for anticanonical pairs}\label{anticanonTorellisub}  Recall the $E_n$ diagram for $n\ge 3$:

$$
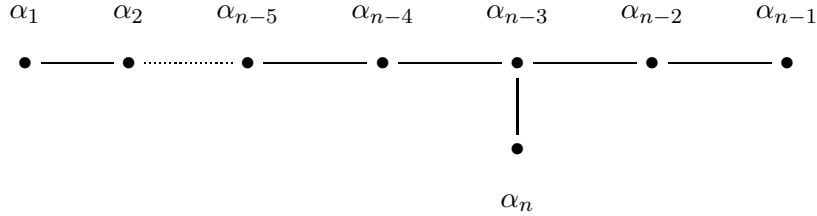
\xymatrix@R=.25cm@C=.25cm{
\alpha_1&&\alpha_2&&\alpha_{n-5}&&\alpha_{n-4}&&\alpha_{n-3}&&\alpha_{n-2}&&\alpha_{n-1}\\
 {\bullet}\ar@{-}[rr]&&{\bullet}\ar@{.}[rr]&&{\bullet}
\ar@{-}[rr]&&{\bullet}
\ar@{-}[rr]&&{\bullet}\ar@{-}[dd]\ar@{-}[rr]&& {\bullet}\ar@{-}[rr]
&&{\bullet}\\  &&&&&&&&&&&&&\\
 &&&&&&&&{\bullet}&&&&&&\\
 &&&&&&&&{\alpha_n}&&&&&&\\ 
}$$
 Consistent with the notation of \S\ref{latticessect}, we define the lattice $\Lambda_{E_n}$ as follows:  the $\alpha_i$ are a basis for  $\Lambda_{E_n}$, and the intersection form is specified by the root diagram for $E_n$. In other words,  $\alpha_i^2 =-2$,  $\alpha_i\cdot \alpha_j=0$ if $\alpha_i, \alpha_j$ are not connected by an edge in the diagram, and $\alpha_i\cdot \alpha_j=1$ if $\alpha_i, \alpha_j$ are   connected by an edge in the diagram. The lattice $\Lambda_{E_n}$ is negative definite if $n\le 8$, negative semi-definite if $n=9$,  and indefinite of signature $(1,n-1)$ if $n \ge 10$. The absolute value of the discriminant of $\Lambda_{E_n}$ is $|9-n|$. 
 
 The lattice $\Lambda_{E_n}$ arises very naturally in algebraic geometry as follows: Let $D\subseteq \Pee^2$ be a smooth cubic, let $h$ be the class of a line in $\Pee^2$ (in $\Pic \Pee^2$ or in $H^2(\Pee^2;\Zee)$),  and let $X$ be the blowup of $\Pee^2$ at $n\ge 3$ points $p_1, \dots, p_n$ (for simplicity assumed distinct). Let $\varepsilon_1, \dots, \varepsilon_n$ be the corresponding exceptional curves. We will also denote by $h$ the pullback of $h$ to $\Pic X$ or   $H^2(X;\Zee)$ and by $\varepsilon_i$ the corresponding element of $\Pic X$ or of $H^2(X;\Zee)$. Then $[K_X] = -3h +\sum_{i=1}^n\varepsilon_i$, and a basis for $[K_X]^\perp \subseteq H^2(X;\Zee)$ is given by
 $$\alpha_1 = \varepsilon_2 - \varepsilon_1, \alpha_2 = \varepsilon_3-\varepsilon_2, \dots, \alpha_{n-1} = \varepsilon_n - \varepsilon_{n-1}, \alpha_n = h - \varepsilon_{n-2}-\varepsilon_{n-1} - \varepsilon_n.$$
 In fact, given a  pair $(X,D)$, where $X$ is a smooth surface, $D\subseteq X$ is an elliptic curve, and $K_X \cong \scrO_X(-D)$, then either $X$ is a blowup of $\Pee^2$ at $n$ points or $X\cong \mathbb{F}_0$ or $\mathbb{F}_2$. 
 
 Following  \cite{GHK} and \cite[Definition 5.4]{Friedanticanon}, we define  $\overline{\mathcal{A}}_{\text{\rm{gen}}}(X)$,  the \textsl{generic ample cone of $X$} by
 $$\overline{\mathcal{A}}_{\text{\rm{gen}}}(X) =\{x\in H^2(X;\Ar): x \cdot \alpha \ge 0 \text{ for all effective numerical exceptional curves $\alpha$}\},$$
 where an effective numerical exceptional curve  $\alpha$ is an element $\alpha \in H^2(X;\Zee)$ such that $\alpha^2 = \alpha \cdot [K_X]  =-1$ and $\alpha$ is the class of an effective divisor. 
We also have   the \textsl{period map} $\varphi_X \colon [K_X]^\perp\cong \Lambda_{E_n} \to JD$ defined by: if $\xi\in [K_X]^\perp$, then $\xi$ is the class of a unique line bundle $\lambda \in \Pic X$, and $\deg(\lambda|D) =0$. Then define 
 $$\varphi_X(\xi) = \lambda|D \in \Pic^0D \cong JD.$$ 
 Of course, $\varphi_X$ is specified by its values on the $\alpha_i$. In particular, note that
 $$\varphi_X(\alpha_i) = \begin{cases} \scrO_D(p_{i+1}-p_i), &\text{if $1\le i \le n-1$;}\\
 \scrO_D(h - p_{n-2}-p_{n-1} - p_n), &\text{if $i=n$.}
 \end{cases}$$

 With this said,  the methods of  \cite{Carlson0}, \cite{Carlson3}, \cite[Theorem I.5.1]{Looi3}, \cite[Theorem 1.8]{GHK}, \cite[Theorem 8.5]{Friedanticanon} can easily be adapted to prove the following    theorem of Torelli type for the pair $(X,D)$:
 
 \begin{theorem}\label{anticanonTorelli}  Let $(X,D)$ and $(X', D')$ be two blowups of $\Pee^2$ at two smooth cubic curves $D$, $D'$. Suppose that $D\cong D'$ and that, fixing an isomorphism $JD\to JD'$ and identifying $JD$ and  $JD'$  via this isomorphism, there is an integral isometry $f \colon H^2(X;\Zee) \to H^2(X'; \Zee)$ such that
 \begin{enumerate}
 \item[\rm(i)] $f([D]) = [D']$.
 \item[\rm(ii)] $f(\overline{\mathcal{A}}_{\text{\rm{gen}}}(X) ) = \overline{\mathcal{A}}_{\text{\rm{gen}}}(X')$.
  \item[\rm(iii)] $f(\Delta_X) = \Delta_{X'}$, where $\Delta_X$ denotes the set of elements of $H^2(X;\Zee)$ of the form $\delta = [C]$, where $C$ is a smooth rational curve of self-intersection $-2$, and similarly for $\Delta_{X'}$.
 \item[\rm(iv)] $\varphi_{X'}\circ f = \varphi_X$. 
 \end{enumerate}
 Then there is a unique isomorphism $g\colon X'\to X$ with $g(D') = D$ and $g^* = f$. \qed
 \end{theorem}
 
We sketch the proof of Theorem~\ref{anticanonTorelli} under the simplifying assumption that $X$ contains no smooth rational curves of self-intersection $-2$, i.e.\  $\Delta_X =\emptyset$. It is easy to see that the same must be true for $X'$. As above, let $X$ be the blowup of $\Pee^2$ at $n\ge 3$ distinct points $p_1, \dots, p_n$ and let $\varepsilon_1, \dots, \varepsilon_n$ be the classes of the corresponding exceptional curves. The hypotheses imply that $f(\varepsilon_1)= \varepsilon_1', \dots, f(\varepsilon_n)= \varepsilon_n'$ are the classes of disjoint exceptional curves on $X'$ and that blowing them down gives a morphism $X'\to \Pee^2$ for which $D'$ is a smooth cubic. There is then an isomorphism of pairs $(\Pee^2, D') \to (\Pee^2, D)$ inducing the given isomorphism $D\cong D'$, and it is unique up to a projective automorphism of $\Pee^2$ which fixes $D$ and is given by translation by a $3$-torsion point of $\Pic ^0D$.  
Use this identification to identify $(\Pee^2, D')$ and $ (\Pee^2, D)$. 

Given the set $\{\alpha_1, \dots, \alpha_n\} \subseteq \Lambda_{E_n}$, the period point $\varphi_X$ determines the differences  $\varphi_X(\alpha_i) = p_{i+1}-p_i$, $i< n$,  and $\varphi_X(\alpha_n) = h - p_{n-2}-p_{n-1} - p_n$ in $JD$. 
By construction,  $X'$ is the blowup of $\Pee^2$ at $n$ distinct points $p_1', \dots, p_n'$ of $D$, and $f(\alpha_i) = \alpha_i'$, where $\alpha_i' = \varepsilon_{i+1}' - \varepsilon_i'$ for $i < n$ and  $\alpha_n' = h - \varepsilon_{n-2}'-\varepsilon_{n-1}' - \varepsilon_n'$. Since  $\varphi_{X'}(\alpha_i' )=    p_{i+1}'-p_i'$ for $i< n$ and   $\varphi_{X'}(\alpha_n' )= h - p_{n-2}'-p_{n-1}' - p_n'$ in $JD$,  Condition (iv) implies that  $p_{i+1}-p_i = p_{i+1}'-p_i'$ and $h - p_{n-2}-p_{n-1} - p_n = h - p_{n-2}'-p_{n-1}' - p_n'$.  The points $p_i$ are determined up to translation by a point $\xi\in \Pic ^0D$. Replacing $p_i$ by $p_i+ \xi$ leaves the differences $p_{i+1}-p_i$ unchanged  and replaces $h - p_{n-2}-p_{n-1} - p_n$ by $h - p_{n-2}-p_{n-1} - p_n-3\xi$, and hence $3\xi=0$. Thus $p_i' = p_i + \xi$, where $\xi$ is a point of order $3$. Then there is   an isomorphism of pairs $(\Pee^2, D) \to (\Pee^2, D)$ which fixes $h$ but which acts on $D$ as translation by $\xi$. Hence the configurations
  $(D, p_1, \dots, p_n)$ and  $(D, p_1', \dots, p_n')$ are projectively equivalent, and therefore the corresponding blowups $X$ and $X'$ of $\Pee^2$ are isomorphic, via an isomorphism $g\colon X'\to X$ as in the statement of the theorem. It is then easy to check that  $g$ is unique.
  
  \medskip
 
Condition (ii) is automatic if $n \le 9$ and in particular if $X$ and $X'$ are (generalized) del Pezzo surfaces. More generally, we have the following:

\begin{lemma}\label{genample}  Suppose that $(X,D)$ and $(X', D')$ are two blowups of $\Pee^2$ at two smooth cubic curves $D$, $D'$ and that $f \colon H^2(X;\Zee) \to H^2(X'; \Zee)$  is an integral isometry  such that  $f([D]) = [D']$. Finally suppose that one of the following holds:
\begin{enumerate}
 \item[\rm(i)] $n \le 9$. 
 \item[\rm(ii)] $n \ge 11$, i.e.\ $D^2\le -2$,  and there exist nef and big divisors $H$ and $H'$ on $X$ and $X'$, respectively, such that, for every irreducible curve $C$ on $X$, $H\cdot C =0$ $\iff$ $C=D$, and similarly for $H'$, and such that $f([H]) = [H']$. 
  \end{enumerate}
  Then $f(\overline{\mathcal{A}}_{\text{\rm{gen}}}(X) ) = \overline{\mathcal{A}}_{\text{\rm{gen}}}(X')$.
\end{lemma} 
\begin{proof} A standard argument along the lines of the proof of \cite[Lemma 5.2]{Friedanticanon} and \cite[Lemma 5.9(iii)]{Friedanticanon} shows that, under either hypothesis, $\alpha$ is an effective numerical exceptional curve for $X$ $\iff$ $f(\alpha)$ is an effective numerical exceptional curve for $X'$. Thus either (i) or (ii)  implies that $f(\overline{\mathcal{A}}_{\text{\rm{gen}}}(X) ) = \overline{\mathcal{A}}_{\text{\rm{gen}}}(X')$.
\end{proof} 

\begin{remark}\label{moregenample}   Similar methods show the following, which leads to a more general version of Lemma~\ref{genample} assuming that $f$ satisfies (i), (iii), and (iv) of Theorem~\ref{anticanonTorelli}: Suppose that $D^2\le -2$ and  that $y$ is a nef $\Ar$-divisor on $X$ such that $y\cdot [D] =0$ and, if $C$ is an irreducible curve on $X$ such that $y\cdot [C] =0$, then $C^2=-2$. Then a numerical exceptional curve $\alpha$ is effective $\iff$ $\alpha \cdot y \ge 0$.
\end{remark} 

\begin{remark} Condition (iii) of Theorem~\ref{anticanonTorelli} is also easy to deal with. In fact, reflections in the classes of elements of $\Delta_X$ generate a reflection group $W(\Delta_X)$ preserving the period map. Given $f \colon H^2(X;\Zee) \to H^2(X'; \Zee)$ satisfying (i), (ii), and (iv) of Theorem~\ref{anticanonTorelli}, after pre-composing $f$ with an element $w\in \Delta_X$, we can assume that $f$ satisfies (iii) as well.
\end{remark}

\section{$d$-semistable models and their deformation theory}

\subsection{Construction and  deformation theory}\label{ssect21}  Here and for future reference, we fix some terminology regarding del Pezzo surfaces:

\begin{definition}\label{defdP} An \textsl{almost del Pezzo surface} $Z$  is a smooth projective surface such that  $-K_Z$ is nef and big. A  \textsl{generalized  del Pezzo surface} $Z$  is a  projective surface with at worst rational double points such that  $\omega_Z^{-1}$ is ample. Thus the minimal resolution of a generalized  del Pezzo surface is an almost del Pezzo surface, and conversely the anticanonical model of an almost del Pezzo surface, i.e.\ the normal surface obtained by contracting all of the curves $C$ such that $K_Z\cdot C =0$, is a generalized  del Pezzo surface. 
\end{definition}

Next we describe a class of surfaces with simple elliptic singularities whose $d$-semistable models are well-behaved:

\begin{assumption}\label{Yassump}    $Y$ is a normal Gorenstein surface whose only singularities are $k$ simple elliptic singularities $p_1, \dots, p_k$. Let $\pi\colon\hY \to Y$ be a minimal resolution, so that the exceptional set of $\pi$ consists of $k$ disjoint smooth elliptic curves $D_1, \dots, D_k$ and let $D =D$. Let $m_i = -D_i^2$. Suppose that the singularities are locally smoothable. Then $1\le m_i \le 9$. We will further assume that $p_1, \dots, p_k$ are not base points for the line bundle $\omega_Y$.
\end{assumption} 

\begin{lemma}\label{IsurfYassump} If $Y$ is a normal Gorenstein I-surface with at worst simple elliptic singularities and no RDP singularities, then $Y$ satisfies Assumption~\ref{Yassump}.
\end{lemma}
\begin{proof} This is proved in \cite[Proposition 3.6]{FPR2}. It also follows directly from the exact sequence
$$0 \to K_{\hY} \to K_{\hY}\otimes \scrO_{\hY}(D)  \to \bigoplus_i\scrO_{D_i} \to 0,$$
with some care in the elliptic ruled case. 
\end{proof} 

\begin{remark} A similar result holds in case $Y$ is the minimal resolution of the RDP singularities of a normal Gorenstein I-surface with at worst simple elliptic singularities. 
\end{remark} 

\begin{definition}\label{defdssmodel}  Given a surface $Y$ satisfying Assumption~\ref{Yassump},  make a $d$-semistable normal crossing model for $Y$ as follows: For each $i$, choose a pair $(Z_i, D_i')$, where $Z_i$ is an almost del Pezzo surface   of degree $(-K_{Z_i})^2 = m_i$  and $D_i'\in |-K_{Z_i}|$ is an anticanonical divisor isomorphic to $D_i \subseteq \hY$. 
%To streamline the notation we will sometimes also use $Z_0$ to denote $\hY$. 
Choose an isomorphism $\varphi_i \colon D_i \to D_i'$ such that $N_{D_i/\hY}\otimes \varphi_i^*N_{D_i'/Z_i} \cong \scrO_{D_i}$ (the $d$-semistable condition) and use $\varphi_i $ to identify  $D_i'$ with  $D_i$. Note that, for each $i$, there are exactly $m_i^2$ choices for $\varphi_i$. Then  glue $Z_i$ to $\hY$ along $D_i$ via the isomorphism $\varphi_i$. Let $X_0= \hY\amalg_{D_i}\left(\coprod_iZ_i\right)$ be the resulting surface with normal crossings, let $a\colon \widetilde{X}_0 \to X_0$ be the normalization map, and let $j\colon \coprod _i D_i \to X_0$ be the inclusion. Of course, we can make the same construction when $Z_i$ is assumed instead to be a generalized del Pezzo surface.

If $\omega_{X_0}$ is the dualizing sheaf of $X_0$, then $\omega_{X_0}|Z_i\cong \scrO_{Z_i}$ and $\omega_{X_0}|\hY \cong  K_{\hY}\otimes \scrO_{\hY}(D) =L$, in the notation of \cite{FG24}. In particular, $\omega_{X_0}$ is trivial in a neighborhood of $Z_i$, and the map $H^0(X_0; \omega_{X_0}) \to H^0(\hY; L)$ is an isomorphism. Thus by Assumption~\ref{Yassump} there exist sections of $\omega_{X_0}$ which vanish only along a curve in $\hY -\bigcup_iD_i$. 
\end{definition}

Next   recall the basic setup of \cite{FriedmanSmoothings}. By construction, $T^1_{X_0} = \mathit{Ext}^1(\Omega^1_{X_0}, \scrO_{X_0}) \cong \bigoplus _i\scrO_{D_i}$. The set of first order deformations of $X_0$ is classified by
$$\mathbb{T}^1_{X_0} =\Ext^1(\Omega^1_{X_0}, \scrO_{X_0}) $$ and the obstruction space to deforming $X_0$ is given by 
$$\mathbb{T}^2_{X_0} =\Ext^2(\Omega^1_{X_0}, \scrO_{X_0}).$$
Because the singularities of $X_0$ are local complete intersections, $T^2_{X_0} = \mathit{Ext}^2(\Omega^1_{X_0}, \scrO_{X_0})=0$ and there is an exact sequence
$$0 \to H^1(X_0; T^0_{X_0}) \to \mathbb{T}^1_{X_0} \to H^0(X_0; T^1_{X_0}) \to H^2(X_0; T^0_{X_0}) \to \mathbb{T}^2_{X_0} \to H^1(X_0; T^1_{X_0}) \to 0,$$
where $T^0_{X_0}$ is the sheaf of derivations of $X_0$.  The sheaf $\bigoplus_iT^i_{X_0} = T^0_{X_0}\oplus T^1_{X_0}$ is a sheaf of graded Lie algebras. Likewise, the vector space 
$$\bigoplus_i\mathbb{T}^i_{X_0} = \mathbb{T}^0_{X_0}\oplus \mathbb{T}^1_{X_0}\oplus \mathbb{T}^2_{X_0}$$
is a graded Lie algebra and there is a natural compatibility between this  Lie bracket and the  pairing 
$$H^p(X_0; T^i_{X_0}) \otimes H^q(X_0; T^j_{X_0}) \to H^{p+q}(X_0; T^{i+j}_{X_0})$$ induced by cup product using the sheaf Lie bracket.

The group $H^1(X_0; T^0_{X_0}) $ is the Zariski tangent space to the locally trivial deformations of $X_0$, and $H^2(X_0; T^0_{X_0})$ is the obstruction space for these deformations. We then have the following:

\begin{lemma}\label{locallytriv} There is an exact sequence
$$0 \to T^0_{X_0} \to a_*\Big(T_{\hY}(-\log D)  \oplus \bigoplus_i T_{Z_i}(-\log  D_i')\Big)\to j_*\Big(\bigoplus_iT_{D_i} \Big)\to 0.$$
 Thus there is an exact sequence
\begin{gather*}
\bigoplus_iH^0(D_i;T_{D_i}) \to  H^1(X_0; T^0_{X_0}) \to H^1(\hY; T_{\hY}(-\log D)) \oplus \bigoplus_i H^1(Z_i; T_{Z_i}(-\log  D_i'))\to  \\
\to \bigoplus_i H^1(D_i; T_{D_i})  \to H^2(X_0; T^0_{X_0}) \to H^2(\hY; T_{\hY}(-\log D)) \oplus \bigoplus_i H^2(Z_i; T_{Z_i}(-\log  D_i')) \to 0.
\end{gather*}
\end{lemma}
\begin{proof} The  exact sequence of sheaves is \cite[Lemma 4.6(iv)]{FL23}. The second exact sequence is then the corresponding long exact cohomology sequence.
\end{proof}

\begin{remark}\label{loctrivremark} The image of the coboundary map $\bigoplus_iH^0(D_i; T_{D_i}) \to H^1(X_0; T^0_{X_0})$ is the tangent space to the deformations of $X_0$ obtained by deforming the gluings $\varphi_i \colon D_i \subseteq \hY \to D_i'\subseteq Z_i$. The geometric content of the above exact sequence is the following: A first order deformation of $X_0$ induces first order deformations of the pairs $(Z_i, D_i)$ and $(\hY,D)$, which preserve to first order the fact  that $\varphi_i \colon D_i \subseteq \hY$ is isomorphic to  $D_i'\subseteq Z_i$. Conversely, given   a collection of  first order deformations of the pairs $(Z_i, D_i)$  and $(\hY,D)$ satisfying the above condition, there is a first order deformation of $X_0$ and it is unique up to a choice of gluings. 
\end{remark}

We turn now to describing the obstruction space $\mathbb{T}^2_{X_0}$. Unfortunately, because 
$$H^1(X_0; T^1_{X_0}) = \bigoplus_iH^1(X_0; \scrO_{D_i}) \cong \Cee^k,$$
the deformations of $X_0$ are always obstructed. The obstructions correspond as in \cite[4.5, 5.10]{FriedmanSmoothings} to deforming the sheaf $T^1_{X_0}$ away from being trivial by e.g.\ deforming the gluings $\varphi_i$.  While one can adapt the arguments of \cite[5.10]{FriedmanSmoothings} to this situation, it is somewhat simpler to use instead the logarithmic deformation theory of Kawamata-Namikawa \cite{KawamataNamikawa}. In particular, they construct a functor $\mathbf{LD}_{X_0}$ (for logarithmic deformations) corresponding to smoothings of $X_0$ or locally trivial deformations which remain $d$-semistable. However, the price we have to pay in this approach is that we have to carry along the additional data of a log structure.    To describe this picture in more detail, using the notation of  \cite{FriedmanSmoothings}, let $\Lambda^1_{X_0}$ denote the ``abstract relative log complex" of \cite[\S3]{FriedmanSmoothings} and let $S_{X_0}$ denote its dual. (In \cite{KawamataNamikawa}, these are denoted by $\Omega^1_{X_0/\Cee}(\log)$ and $T_{X_0/\Cee}(\log)$ respectively.) Fix an everywhere generating section $\xi$ of $H^0(X_0; T^1_{X_0})$. Then there is an exact sequence
$$0 \to S_{X_0} \to T^0_{X_0} \xrightarrow{[\cdot, \xi]}   T^1_{X_0} \cong \bigoplus_i\scrO_{D_i} \to 0.$$
It is easy to see that $S_{X_0}$ is independent of the choice of $\xi$. There is then an exact sequence
$$H^0(X_0; T^0_{X_0}) \to   H^0(X_0; T^1_{X_0}) \to H^1(X_0; S_{X_0}) \to H^1(X_0; T^0_{X_0}).$$
The image of $H^0(X_0; T^1_{X_0})\cong \bigoplus_iH^0(D_i; \scrO_{D_i}) $ in $H^1(X_0; S_{X_0})$ records the deformation of the given logarithmic structure on $X_0$, keeping $X_0$ itself fixed. The map from  $H^1(X_0; S_{X_0})$ to $H^1(X_0; T^0_{X_0})$ corresponds to ``forgetting the logarithmic structure," and its image is the tangent space $T_{\Xi,x_0}$ to the space (or functor) of locally trivial deformations of $X_0$ preserving the $d$-semistability condition, by \cite[4.5]{FriedmanSmoothings}. 
More precisely,  the argument of \cite[Corollary 2.4]{KawamataNamikawa} essentially shows the following:

\begin{theorem}\label{versal}  Suppose that $H^2(X_0; S_{X_0}) =0$. Then there is a germ of a smooth manifold $(\Xi, x_0)$ and  a ``logarthmic semi-universal deformation" $\mathcal{X} \to \Xi\times \Delta^k$, where $\Delta$ is the unit disk, such that 
 \begin{enumerate} 
 \item[\rm(i)] The total space $\mathcal{X}$ is smooth and the fiber over $(x_0, 0)$ is $X_0$.
 \item[\rm(ii)] The Kodaira-Spencer map induces an isomorphism from the tangent space  $T_{\Xi, x_0}$ of $\Xi$ at $x_0$ to the tangent space of locally trivial deformations of $X_0$ for which the $d$-semistability condition holds, i.e.\ for which $T^1_{X_0}$ remains trivial. Moreover, $T_{\Xi, x_0}$ is the image of $H^1(X_0; S_{X_0})$ in $H^1(X_0; T^0_{X_0})$, i.e.\ $T_{\Xi, x_0} \cong H^1(X_0; S_{X_0})/ \im H^0(X_0; T^1_{X_0})$.
 \item[\rm(iii)] The restriction of  $\mathcal{X}$ to  $\Xi\times (\Delta^*)^k$ is a smooth morphism. 
  \item[\rm(iv)] Let  $(z_1, z_2, \dots, z_k)$ be the product  coordinates on $\Delta^k$ and let
 $$\theta\colon T_{\Xi, x_0} \oplus \bigoplus_i \Cee\frac{\partial}{\partial z_i} \to \mathbb{T}^1_{X_0}$$ be the Kodaira-Spencer map at $(x_0,0)$.  Then the image of $\theta(\partial/\partial z_i)$ in $$H^0(X_0; T^1_{X_0}) \cong \bigoplus_iH^0(D_i; \scrO_{D_i})$$ is an $i^{\text{\rm{th}}}$ basis vector.  \qed
 \end{enumerate}
\end{theorem}

Fixing an everywhere generating section of $T^1_{X_0}$ as  above,  the Lie bracket $[\cdot, \xi]\colon \mathbb{T}^1_{X_0} \to \mathbb{T}^2_{X_0} \cong H^1(X_0; T^1_{X_0})$   is compatible with the bracket $[\cdot, \xi]\colon H^1(X_0; T^0_{X_0}) \to   H^1(X_0; T^1_{X_0})$. Define $ (\mathbb{T}^1_{X_0})_{\mathbf{s}}$ to be $\Ker \{[\cdot, \xi]\colon \mathbb{T}^1_{X_0} \to \mathbb{T}^2_{X_0}\}$. As before, $ (\mathbb{T}^1_{X_0})_{\mathbf{s}}$ is independent of the choice of $\xi$. 
 
 \begin{corollary}\label{gentangsp}  Suppose that $H^2(X_0; S_{X_0}) = 0$. Then $ (\mathbb{T}^1_{X_0})_{\mathbf{s}}$ is the Zariski tangent space to the smoothing component of the functor of deformations of $X_0$ and  there is an exact sequence
$$
0 \to T_{\Xi,x_0} \to (\mathbb{T}^1_{X_0})_{\mathbf{s}} 
\to H^0(X_0; T^1_{X_0}) \cong \Cee^k \to 0.  \qed
$$
 \end{corollary} 

\begin{theorem}\label{versal2}  Suppose that $Y$ satisfies Assumption~\ref{Yassump} and that  $ H^2(\hY; T_{\hY}(-\log D)) =0$.   Then $H^2(X_0; T^0_{X_0}) = H^2(X_0; S_{X_0}) =0$ and in particular the conclusions of Theorem~\ref{versal}  hold. 
\end{theorem} 
\begin{proof}
By assumption, $H^2(\hY; T_{\hY}(-\log D)) =0$. Also, $H^2(Z_i; T_{Z_i}(-  D_i')) =0$ for every $i$ by the argument of \cite[Example 1.14]{FG24}. By \cite[Lemma 1.12]{FG24}, for every $i$,  $H^2(Z_i; T_{Z_i}(-\log  D_i')) =0$  and $H^1(Z_i; T_{Z_i}(-\log  D_i'))\to H^1(D_i; T_{D_i}) $ is surjective. Thus, by Lemma~\ref{locallytriv},  $H^2(X_0; T^0_{X_0}) =0$.

To prove that $H^2(X_0; S_{X_0}) =0$, it suffices to prove that $H^0(X_0; \Lambda^1_{X_0}\otimes \omega_{X_0}) =0$. By \cite[\S3]{FriedmanSmoothings}, there is a subsheaf $V_0 \Lambda^1_{X_0} \cong \Omega^1_{X_0}/\tau^1_{X_0}$, and the quotient  $\Lambda^1_{X_0}/V_0 \Lambda^1_{X_0}$ is isomorphic to  $\bigoplus_i\scrO_{D_i}$. By construction, 
$$\Lambda^1_{X_0} \subseteq a_*\Big(\Omega^1_{\hY}(\log D) \oplus \bigoplus_i\Omega^1_{Z_i}(\log  D_i) \Big), $$
the isomorphism $\Lambda^1_{X_0}/V_0 \Lambda^1_{X_0} \cong \bigoplus_i\scrO_{D_i}$ is induced by Poincar\'e residue, and there is a commutative diagram
$$\begin{CD}
\Omega^1_{X_0}/\tau^1_{X_0} @>>> \Lambda^1_{X_0}\\
@VVV @VVV \\
a_*\Big(\Omega^1_{\hY}  \oplus \bigoplus_i\Omega^1_{Z_i}  \Big) @>>> a_*\Big(\Omega^1_{\hY}(\log D) \oplus \bigoplus_i\Omega^1_{Z_i}(\log  D_i) \Big).
\end{CD}$$

The dualizing sheaf  $\omega_{X_0}$ is trivial on $Z_i$ and in particular has a trivial restriction to $D_i$. There is an exact sequence
$$0 \to H^0(X_0; (\Omega^1_{X_0}/\tau^1_{X_0})\otimes  \omega_{X_0}) \to H^0(X_0; \Lambda^1_{X_0}\otimes \omega_{X_0}) \to \bigoplus_i H^0(\scrO_{D_i}) 
\to H^1(X_0; (\Omega^1_{X_0}/\tau^1_{X_0})\otimes  \omega_{X_0}).$$
By \cite[Lemma 2.9]{FriedmanSmoothings},  $(T^0_{X_0})\spcheck  \cong  \Omega^1_{X_0}/\tau^1_{X_0} $, and thus $H^0(X_0; (\Omega^1_{X_0}/\tau^1_{X_0})\otimes  \omega_{X_0}) = \Hom (T^0_{X_0}, \omega_{X_0})$. By Serre duality, $\Hom (T^0_{X_0}, \omega_{X_0})$ is dual to $H^2(X_0; T^0_{X_0})$ and hence is $0$. By Assumption~\ref{Yassump}, there exists  a nonzero section $\sigma$ of $\omega_{X_0}$ which only vanishes in $\hY-\bigcup_iD_i$, and which thus defines an inclusion  $\Lambda^1_{X_0} \to \Lambda^1_{X_0}\otimes \omega_{X_0}$. There is a commutative diagram
$$\begin{CD}
@.  H^0(X_0; \Lambda^1_{X_0}) @>>> \bigoplus_i H^0(\scrO_{D_i}) @>>> H^1(X_0;  \Omega^1_{X_0}/\tau^1_{X_0}) \\
@. @VV{\times \sigma}V @VV{=}V  @VV{\times \sigma}V \\
0 @>>> H^0(X_0; \Lambda^1_{X_0}\otimes \omega_{X_0}) @>>> \bigoplus_i H^0(\scrO_{D_i}) @>>> H^1(X_0;  (\Omega^1_{X_0}/\tau^1_{X_0})\otimes \omega_{X_0}) 
\end{CD}$$
By the above remarks, the map obtained by post-composing the map $\bigoplus_i H^0(\scrO_{D_i}) \to  H^1(X_0;  \Omega^1_{X_0}/\tau^1_{X_0})$ with the natural map
$$H^1(X_0;  \Omega^1_{X_0}/\tau^1_{X_0}) \to \bigoplus_i H^1(Z_i; \Omega^1_{Z_i})$$ is the fundamental class map and is therefore injective. Clearly, this map agrees with the corresponding composition 
$$\bigoplus_i H^0(\scrO_{D_i}) \to  H^1(X_0;  (\Omega^1_{X_0}/\tau^1_{X_0})\otimes \omega_{X_0})\to \bigoplus_i H^1(Z_i; \Omega^1_{Z_i})$$
up to multiplication by $\sigma$, which is an isomorphism on $H^0(\scrO_{D_i})$ and on $H^1(Z_i; \Omega^1_{Z_i})$. 
Hence $\bigoplus_i H^0(\scrO_{D_i}) \to  H^1(X_0;  (\Omega^1_{X_0}/\tau^1_{X_0})\otimes \omega_{X_0})$ is also injective.
Putting this together, it follows that $H^0(X_0; \Lambda^1_{X_0}\otimes \omega_{X_0}) =0$ and hence that $H^2(X_0; S_{X_0}) =0$. 
\end{proof} 

The following shows that the hypotheses of Theorem~\ref{versal2} hold for I-surfaces under a mild general position assumption: 

\begin{proposition} Suppose that, as above,   $X_0 = \hY\amalg_{D_i}\left(\coprod_iZ_i\right)$, where $\hY$ is the minimal resolution of an I-surface  and the $Z_i$ are almost del Pezzo surfaces. Finally assume either that $\hY$ is elliptic ruled or that it is generic in an appropriate sense. Then $H^2(X_0; T^0_{X_0}) = H^2(X_0; S_{X_0}) =0$
\end{proposition}
\begin{proof}
  If  $\hY$ is generic or it is elliptic ruled, then $H^2(\hY; T_{\hY}(-\log(D)) =0$  by various results scattered throughout \cite{FG24}. For example, the elliptic ruled case is proved in \cite[Theorem 7.3(i)]{FG24}. Then we can conclude by Lemma~\ref{IsurfYassump} and Theorem~\ref{versal2}. 
\end{proof}

\begin{corollary} If  $X_0 = \hY\amalg_{D_i}\left(\coprod_iZ_i\right)$, where $\hY$ is the minimal resolution of an I-surface which  is either elliptic ruled or   generic in an appropriate sense and the $Z_i$ are almost del Pezzo surfaces, then the conclusions of Theorem~\ref{versal}  hold for $X_0$. \qed
\end{corollary}

\begin{remark} Suppose that one or more of $D_i$ correspond instead to a smoothable cusp singularity. Then one can complete $\hY$ to a $d$-semistable surface $X_0$ with normal crossings, in many different ways, by the methods of \cite{FriedmanMiranda}, \cite{Engel}, \cite{EngelFriedman}. Similar but slightly more complicated arguments show that the analogues of the above theorems hold for $X_0$, under some mild assumptions on the cusp and the appropriate cohomological conditions. In particular, they hold for I-surfaces $Y$ under  certain general position assumptions on $\hY$. 
\end{remark} 

Next we analyze the first order deformations of $X_0$ in more detail. First we consider the tangent space  $H^1(X_0; T^0_{X_0})$ to locally trivial deformations.  By Lemma~\ref{locallytriv}, there is an exact sequence
 $$0 \to T^0_{X_0} \to a_*\Big(T_{\hY}(-\log D)  \oplus \bigoplus_i T_{Z_i}(-\log  D_i)\Big)\to j_*\Big(\bigoplus_iT_{D_i} \Big)\to 0.$$
 The image of $\bigoplus_iH^0(D_i; T_{D_i})$ in $H^1(X_0; T^0_{X_0})$ corresponds to deforming the gluings of $D_i\subseteq \hY$ to $D_i' \subseteq Z_i$ by an infinitesimal automorphism of $D_i$. 
 To deal with the tangent space $T_{\Xi,x_0}$ to the set of locally trivial deformations of $X_0$ preserving the $d$-semistability condition, we use the following:
 
 \begin{lemma}\label{loccomp}  The space $T_{\Xi,x_0}$ is a complement in $H^1(X_0; T^0_{X_0})$ to the image of $\bigoplus_iH^0(D_i; T_{D_i})$. 
 \end{lemma}
 \begin{proof} Let $\xi \in H^0(X_0; T^1_{X_0})$ be an everywhere generating section. By \cite[4.5]{FriedmanSmoothings}, 
 $$T_{\Xi, x_0} = \im \Big\{ H^1(X_0; S_{X_0}) \to  H^1(X_0; T^0_{X_0}) \Big\} = \Ker\Big\{ [\cdot, \xi]\colon H^1(X_0; T^0_{X_0}) \to H^1(X_0; T^1_{X_0})\Big\}.$$
  (Compare also \cite{KawamataNamikawa}.)
  The local calculations in the proof there show that, if $\partial \colon  \bigoplus_iH^0(D_i; T_{D_i}) \to H^1(X_0; T^0_{X_0})$ is the coboundary map from Lemma~\ref{locallytriv}, then the corresponding homomorphism 
 $$[\partial( \cdot), \xi]\colon \bigoplus_iH^0(D_i; T_{D_i}) \to H^1(X_0; T^1_{X_0}) = \bigoplus_iH^1(D_i; \scrO_{D_i})$$
 is, up to a rescaling on the various factors, the natural action of $H^0(D_i; T_{D_i})$ on $H^1(D_i; \scrO_{D_i})$ and is therefore an isomorphism. Thus $T_{\Xi, x_0}$ is a complement to the image of $\bigoplus_iH^0(D_i; T_{D_i})$.
 \end{proof}

\subsection{Simple elliptic versus $d$-semistable models}\label{ssectsevsdss}  
Let $\mathcal{M}$ be the coarse moduli space of I-surfaces, possibly with rational double points, and let $\overline{\mathcal{M}}$ be the open subscheme  of the KSBA compactification of $\mathcal{M}$ where the I-surfaces are allowed to have simple elliptic singularities as well as possibly RDP singularities. Since we have chosen not to work with stacks, we must work instead with Kuranishi models. Let $Y$ be an I-surface with at worst $k$ simple elliptic singularities $p_1, \dots, p_k$. For simplicity assume that $Y$ does not have any RDP singularities. Of course, this is automatic in the elliptic ruled case. Let  $T$ be the base of the miniversal deformation space of $Y$ 
and let $T_{\textrm{es}} \subseteq T$ be the equisingular locus, i.e.\ the subspace of $T$ where all of the corresponding fibers have $k$ simple elliptic singularities. 
We begin by analyzing the tangent space to  $T_{\textrm{es}}$:

\begin{lemma}\label{tangdaggerlemma}  {\rm(i)} If $T_{T_{\rm{es}}, s_0}$ denotes the tangent space the tangent space to $T_{\rm{es}}$ at a point $s_0$, then  $T_{T_{\rm{es}}, s_0} \cong  H^1(\hY; T_{\hY}(-\log D))$.

\smallskip
\noindent  {\rm(ii)} For $k \le 2$ and $\hY$ suitably generic, the map $H^1(\hY; T_{\hY}(-\log D)) \to \bigoplus _iH^1(D_i; T_{D_i})$ is surjective. For $k=3$, the map  $H^1(\hY; T_{\hY}(-\log D)) \to H^1(D_i; T_{D_i}) \cong H^1(B; T_B)$ is surjective  for every $i$, and is an isomorphism in case $(m_1,m_2, m_3) = (1,1,1)$. In particular, in this case, all deformations of the pair $(\hY, D)$ come from deforming the base elliptic curve $B$.  
 \end{lemma}
\begin{proof} (i) This is a consequence of Wahl's theory \cite[Proposition 2.5, Proposition 2.7]{WahlI} (cf.\ also \cite[Remark 1.14]{FG24}). 

\smallskip
\noindent (ii) We shall just write   out the proof in the case $(m_1,m_2, m_3) = (1,1,1)$. First we recall some standard facts about deformations of ruled surfaces and blowups of surfaces. If   $Y_0 =\Pee(W)$ is as in Definition~\ref{defY0}, then since $H^i(B; \operatorname{ad}W) =0$, $i=0,1$,   $H^i(Y_0; T_{Y_0}) \cong H^i(B; T_B)$, $i=0,1$. Let $\rho\colon \hY \to Y_0$ be the blowup map at the points $p_1$ and $p_2$. Then $R^1\rho_*T_{\hY} = 0$ and there is an exact sequence
$$0 \to R^0\rho_*T_{\hY}\to T_{Y_0} \to \Cee^2_{p_1} \oplus \Cee^2_{p_2} \to 0.$$
  Also, $H^i(\hY; T_{\hY})\cong H^i(Y_0; R^0\rho_*T_{\hY})$ by the Leray spectral sequence. Thus there is an   exact sequence
$$0 \to H^0(\hY; T_{\hY})\to H^0(Y_0; T_{Y_0}) \to \Cee^2  \oplus \Cee^2  \to  H^1(\hY; T_{\hY})\to H^1(Y_0; T_{Y_0}) \to 0.$$
The map $H^0(Y_0; T_{Y_0}) \cong H^0(B; T_B)\to \Cee^2  \oplus \Cee^2$  is injective, for example because $\Aut Y_0 \cong \Aut B$ acts freely on the blowup points $p_1$ and $p_2$. Hence there is an exact sequence
$$0\to \Cee^3 \to  H^1(\hY; T_{\hY})\to H^1(Y_0; T_{Y_0}) \to 0.$$ 
In other words, $\dim H^1(\hY; T_{\hY}) = 4$, i.e.\ $\hY$ has $4$ moduli: one from the moduli of $B$ and $3$ from the moduli of the two blowup points modulo the action of $\Aut Y_0$. In terms of  $H^1(\hY; T_{\hY}(-\log D))$, 
there is an exact sequence
$$0 \to T_{\hY}(-\log D)  \to T_{\hY} \to \bigoplus_iN_{D_i/\hY} \to 0.$$
Since $\deg N_{D_i/\hY} =-1$ and  $H^2(\hY; T_{\hY}(-\log D)) =0$ by \cite[Theorem 7.3(i)]{FG24},   the image of $H^1(\hY; T_{\hY}(-\log D))$ in $H^1(\hY; T_{\hY})$ has codimension $3$. Thus $\dim H^1(\hY; T_{\hY}(-\log D)) = 1$. By \cite[Theorem 7.3(iii)]{FG24},  for every $i$, the map $H^1(\hY; T_{\hY}(-\log D))\to H^1(D_i; T_{D_i})$ is surjective, and hence it is an isomorphism since both groups have dimension one. Since the maps $D_i \to B$ are \'etale, we can identify $H^1(D_i; T_{D_i})$ with $H^1(B; T_B)$.  Thus $H^1(\hY; T_{\hY}(-\log D)) \cong H^1(D_i; T_{D_i}) \cong H^1(B; T_B)$ as claimed.
 \end{proof}

 For each singular point $p_i$ of $Y$, let $S_i$ be   the miniversal deformation space for the simple elliptic singularity $p_i$. For each $i$,  by Theorem~\ref{mainappthm}, there is  a weighted blowup $\widetilde{S}_i \to S_i$ and a finite cover $\widehat{S}_i  \to \widetilde{S}_i $, with covering group the corresponding Weyl group $W_i$ (see also Remark~\ref{logresremark}).  For each $i$,   let $S_{i,\textrm{es}} \subseteq S_i$ be the equisingular locus. Then     the fiber $\widehat{S}_{i,\textrm{es}}$ of $\widehat{S}_i  \to S_i$ over a point of $S_{i,\textrm{es}}$ corresponding to the elliptic  curve $E_i$ is isomorphic  to  $E_i\otimes Q_i$, i.e.\ to a  moduli space  of marked almost del Pezzo surfaces in the terminology of Definition~\ref{defmarked}.  More generally, the fiber $\widehat{S}_{i,\textrm{es}}$ of $\widehat{S}_i  \to S_i$ over   $S_{i,\textrm{es}}$ is $\mathcal{E}_i\otimes Q_i$, where  $\mathcal{E}_i$ is the germ of the universal elliptic curve over $S_{i,\textrm{es}}$. By   Theorem~\ref{mainappthm}, there is a family of surfaces $\widehat{\mathcal{Z}}$ over $\widehat{S}_1\times \cdots \times \widehat{S}_k$, whose fibers  are disjoint unions of $d$-semistable surfaces over the exceptional locus $\widehat{S}_{1,\textrm{es}}\times \cdots \times \widehat{S}_{k,\textrm{es}}$ and are smooth   elsewhere, and the total space of the family is smooth.  
 
 With $T$  the base of the miniversal deformation of $Y$ as above,  there is an induced  morphism $T \to S_1 \times \cdots \times S_k$. Strictly speaking, this morphism depends on the choice of an isomorphism $\varphi_i$ from $D_i$ to a fixed elliptic curve $E_i$, but we shall by and large gloss over this point.  Note that, by \cite{FG24}, the morphism   $T \to S_1 \times \cdots \times S_k$ is smooth if $k\le 2$ under mild general position assumptions, and its image can be explicitly described for $k=3$: If $S \subseteq S_1\times S_2\times S_3$ is the codimension $2$ submanifold such that, along the equisingular locus $S_{1,\textrm{es}}  \times S_{2,\textrm{es}}  \times S_{3,\textrm{es}}$, the elliptic curves remain isogeneous, then the image of $T\to S_1\times S_2\times S_3$ is $S$ and the morphism $T\to S$ is smooth (cf.\ \cite[Theorem 7.3]{FG24}).  We set $S = S_1 \times \cdots \times S_k$ if $k\le 2$ and let $S \subseteq S_1\times S_2\times S_3$ be the codimension $2$ submanifold defined above if $k=3$. Let $S_{\textrm{es}} \subseteq S$ be the equisingular locus in $S$. Then the corresponding morphism $T_{\textrm{es}} \to S_{\textrm{es}}$ is smooth as well. 
 
 Then  we can take 
$$\widehat{T} = T \times_{(S_1 \times \cdots \times S_k)}(\widehat{S}_1\times \cdots \times \widehat{S}_k).$$
Let  $\widehat{T}_{\textrm{es}}$ be the preimage of $T_{\textrm{es}}$ in $\widehat{T}$. Thus there is a smooth morphism $\widehat{T}_{\textrm{es}} \to T_{\textrm{es}}$ whose fiber over a point $s_0$ is a product of $k$ spaces of the form $E_i\otimes Q_i$.  By construction, $\widehat{T}_{\textrm{es}} \subseteq \widehat{T}$ is a smooth submanifold. Moreover, gluing in the construction described in Theorem~\ref{mainappthm} to the pulled back universal family over $\widehat{T}$, there is a proper morphism $\mathcal{X} \to \widehat{T}$ of relative dimension $2$, where $\mathcal{X}$ is smooth, the morphism is smooth over $\widehat{T} - \widehat{T}_{\textrm{es}}$, and the fibers over $\widehat{T}_{\textrm{es}}$ are $d$-semistable. The construction depends on a choice of an isomorphism $\varphi_i\colon D_i \to E_i$ to a fixed elliptic curve $E_i$, and  as noted in Definition~\ref{defdssmodel}  there are $m_i^2$ such choices. However, for the deformation theory arguments below, the choice of $\varphi_i$ will not matter. 

 At a point $t_0\in \widehat{T}_{\textrm{es}}$,  the family $\mathcal{X}$ induces a Kodaira-Spencer map $T_{\widehat{T}, t_0} \to \mathbb{T}^1_{X_0}$. By construction, the image of the Kodaira-Spencer map  lies in $(\mathbb{T}^1_{X_0})_{\mathbf{s}}$,  the composition 
 $$T_{\widehat{T}, t_0} \to (\mathbb{T}^1_{X_0})_{\mathbf{s}} \to H^0(X_0; T^1_{X_0}) \cong \Cee^k$$ is surjective,  and 
$$T_{\widehat{T}_{\textrm{es}}, t_0} = \Ker \Big\{ T_{\widehat{T}, t_0} \to H^0(X_0; T^1_{X_0}) \cong \Cee^k\Big\}.$$

\begin{theorem}\label{lochatT} Suppose that $k=3$ or that $\hY$ is general. Let $t_0\in \widehat{T}$ correspond to the singular $d$-semistable surface $X_0$. Then the   Kodaira-Spencer homomorphism induces an isomorphism from the tangent space $T_{\widehat{T}, t_0}$ of $\widehat{T}$ at $t_0$ to $(\mathbb{T}^1_{X_0})_{\mathbf{s}}$.
\end{theorem} 
\begin{proof} For simplicity, we  will only check this in case $k=3$. 
 We begin with the following notation:

\begin{definition}\label{defdagger}  Let $\hY$ be the minimal resolution of an I-surface with three simple elliptic singularities. Denote by  $\iota(\theta_i)$  the image of $\theta_i$ in $H^1(D_i; T_{D_i}) \cong H^1(B; T_B)$. Then define 
$$\Big(\bigoplus_iH^1(Z_i;  T_{Z_i}(-\log  D_i))\Big)^{\dagger}  = 
\Big\{(\theta_1, \theta_2, \theta_3) \in \bigoplus_iH^1(Z_i;  T_{Z_i}(-\log  D_i)) : \iota(\theta_1) = \iota(\theta_2) = \iota(\theta_3)\Big\}.$$ 
This subspace corresponds to deforming the $Z_i$ but keeping the curves $D_i$ pairwise isogenous. 
Define the subspace $\Big(\bigoplus_iH^1(Z_i;  T_{Z_i}(-\log  D_i))\Big)^{\ddagger} \subseteq \Big(\bigoplus_iH^1(Z_i;  T_{Z_i}(-\log  D_i))\Big)^{\dagger} $ by
 $$\Big(\bigoplus_iH^1(Z_i;  T_{Z_i}(-\log  D_i))\Big)^{\ddagger}  = \Big\{(\theta_1, \theta_2, \theta_3) \in \bigoplus_iH^1(Z_i;  T_{Z_i}(-\log  D_i)) : \iota(\theta_1) = \iota(\theta_2) = \iota(\theta_3)=0\Big\}.$$
 This second subspace corresponds to deforming the $Z_i$ but keeping all of the curves $D_i$  isogenous to a fixed curve $B$.  Let $H^1(Z_i;  T_{Z_i}(-\log  D_i)) _0 =\Ker \{\iota\colon H^1(Z_i;  T_{Z_i}(-\log  D_i) \to H^1(D_i; T_{D_i})\}$. By the discussion preceding Theorem~\ref{KSmap1}, 
\begin{align*}
\Big(\bigoplus_iH^1(Z_i;  T_{Z_i}(-\log  D_i))\Big)^{\ddagger}  &= \bigoplus_iH^1(Z_i;  T_{Z_i}(-\log  D_i)) _0\\
&\cong \bigoplus_iH^1(Z_i;  T_{Z_i}(-\log  D_i))/\im H^0(D_i; T_{D_i}).
\end{align*}
\end{definition}

 \begin{proposition}\label{tangdagger} Under the   assumption that $k=3$, there is an exact sequence
$$0 \to \bigoplus_iH^0(D_i; T_{D_i}) \to H^1(X_0; T^0_{X_0}) \to \Big(\bigoplus_iH^1(Z_i;  T_{Z_i}(-\log  D_i))\Big)^{\dagger} \to 0.$$
Hence the induced homomorphism
 $$T_{\Xi,x_0} \to \Big(\bigoplus_iH^1(Z_i;  T_{Z_i}(-\log  D_i))\Big)^{\dagger} $$
 is an isomorphism.   
 \end{proposition}
 
 \begin{proof}  
  The exact sequence at the beginning of the proof of Lemma~\ref{locallytriv} and Lemma~\ref{loccomp} imply that $T_{\Xi,x_0}$ is isomorphic to the kernel of the map 
  $$H^1(\hY; T_{\hY}(-\log D))  \oplus \Big(\bigoplus_iH^1(Z_i;  T_{Z_i}(-\log  D_i))\Big) \to \bigoplus_iH^1(D_i;T_{D_i}) \cong H^1(B; T_B)^3.$$
  By Lemma~\ref{tangdaggerlemma}, this kernel is isomorphic to  $\Big(\bigoplus_iH^1(Z_i;  T_{Z_i}(-\log  D_i))\Big)^{\dagger}$.
   \end{proof}

 In the context of Remark~\ref{loctrivremark},  we can interpret Proposition~\ref{tangdagger} as follows: A first order deformation of $X_0$ induces first order deformations of the pairs $(Z_i, D_i)$, which preserve to first order the isogenies $D_i \to B$. Conversely,    a collection of  first order deformations of the pairs $(Z_i, D_i)$ satisfying this condition  determines  a first order deformation of $B$ and hence of $(\hY,D)$, and thus  a first order deformation of $X_0$ which  is unique up to a choice of gluings. Keeping the $d$-semistability condition to first order then determines the gluings up to first order. 
 
 Returning to the proof of Theorem~\ref{lochatT}, and referring to Corollary~\ref{gentangsp}, we must show that the map $T_{\widehat{T}, t_0} \to (\mathbb{T}^1_{X_0})_{\mathbf{s}}$ is an isomorphism. Since the composition $T_{\widehat{T}, t_0} \to (\mathbb{T}^1_{X_0})_{\mathbf{s}} \to H^0(X_0; T^1_{X_0}) \cong \Cee^3$ is surjective,   it is a question of showing that the induced map from the kernel of the above map to $T_{\Xi,x_0} \cong \Big(\bigoplus_iH^1(Z_i;  T_{Z_i}(-\log  D_i))\Big)^{\dagger} $ is an isomorphism. 
 
 Using the  isomorphism 
  $H^1(\hY; T_{\hY}(-\log D)) \to \Big\{(\vartheta_1, \vartheta_2, \vartheta_3) \in (H^1(B; T_B))^3: \vartheta_1 = \vartheta_2 = \vartheta_3\Big\}$ of Lemma~\ref{tangdaggerlemma}, 
  there is a commutative diagram
 $$\begin{CD}
 T_{\widehat{T}_{\textrm{es}}, t_0} = \Ker\Big\{T_{\widehat{T}, t_0} \to  H^0(X_0; T^1_{X_0}) \Big\} @>>> T_{\Xi,x_0} \cong \Big(\bigoplus_iH^1(Z_i;  T_{Z_i}(-\log  D_i))\Big)^{\dagger}\\
 @VVV @VVV \\
 H^1(\hY; T_{\hY}(-\log D)) @>{\cong}>> \Big\{(\vartheta_1, \vartheta_2, \vartheta_3) \in (H^1(B; T_B))^3: \vartheta_1 = \vartheta_2 = \vartheta_3\Big\}.
 \end{CD}$$
 Both vertical arrows are surjective, and by definition and by Theorem~\ref{KSmap1}, the kernel of the right hand vertical arrow is $\Big(\bigoplus_iH^1(Z_i;  T_{Z_i}(-\log  D_i))\Big)^{\ddagger}$. By the discussion before the statement of  Theorem~\ref{lochatT} and Theorem~\ref{KSmap1},   there is an exact sequence 
 $$ 0 \to \Big(\bigoplus_iH^1(Z_i;  T_{Z_i}(-\log  D_i))\Big)^{\ddagger}\to T_{\widehat{T}_{\textrm{es}}, t_0} \to T_{T_{\textrm{es}}, s_0} \to 0,$$
 where $s_0 \in T_{\textrm{es}}$ is the image of $t_0$.  Combining the above shows that $T_{\widehat{T}, t_0} \to (\mathbb{T}^1_{X_0})_{\mathbf{s}}$ is an isomorphism as claimed. 
 \end{proof}

%Let $\widetilde{\mathcal{M}}$ be  the stack obtained by replacing the simple elliptic singularities by their $d$-semistable model (but possibly allowing the $Z_i$ to have rational double points). Explicitly, 

\section{The mixed Hodge structure of the $d$-semistable model} 

\subsection{The mixed Hodge structure on $X_0$} Throughout the remainder of this paper, we assume that   $k=2$ or $3$, i.e.\ that $\hY$ is rational, Enriques, or elliptic ruled. We keep the   previous notation: $X_0 = \hY\amalg_{D_i}\left(\coprod_iZ_i\right)$, $a\colon \widetilde{X}_0 \to X_0$ is the normalization, and $j\colon \coprod _iD_i \to X_0$ is the inclusion. First we consider the Hodge filtration on $H^i(X_0;\Cee)$: Recall that the spectral sequence with $E_1$ term $E_1^{p,q} = H^q(X_0; \Omega^p_{X_0}/\tau^p_{X_0}) \implies \mathbb{H}^{p+q}(X_0; \Omega^\bullet_{X_0}/\tau^\bullet_{X_0})\cong H^{p+q}(X_0; \Cee)$ degenerates at $E_1$ and the corresponding filtration is  the Hodge filtration on $H^{p+q}(X_0;\Cee)$. In the rational, Enriques, or elliptic ruled cases, we can describe $H^q(X_0; \Omega^p_{X_0}/\tau^p_{X_0})$ as follows:
 
 \begin{lemma}\label{Hodgefilt}    \begin{enumerate} \item[\rm(i)]    $H^1(X_0; \scrO_{X_0}) =0$.  
\item[\rm(ii)]   There is an exact sequence
$$0 \to H^1(\hY; \scrO_{\hY}) \to \bigoplus_i H^1(D_i; \scrO_{D_i}) \to H^2(X_0; \scrO_{X_0}) \to 0.$$
Hence the induced map $H^1(D_\ell; \scrO_{D_\ell}) \to H^2(X_0; \scrO_{X_0})$ is injective  for every $\ell$. 
 \item[\rm(iii)] $F^2H^2(X_0) = H^0(X_0; \Omega^2_{X_0}/\tau^2_{X_0})=0$.
 \item[\rm(iv)] There is an exact sequence
\begin{gather*} 0 \to H^0(\hY; \Omega^1_{\hY}) \to \bigoplus_iH^0(D_i; \Omega^1_{D_i}) \to H^1(X_0; \Omega^1_{X_0}/\tau^1_{X_0}) \to \\
\to  H^1(\hY;\Omega^1_{\hY} ) \oplus \bigoplus_iH^1(Z_i; \Omega^1_{Z_i}) \to \bigoplus_iH^1(D_i; \Omega^1_{D_i})  \to 0.
\end{gather*}
 \end{enumerate}
 \end{lemma} 
 \begin{proof}  The   Mayer-Vietoris sequence  for $\scrO_{X_0}$ reads as follows:
$$0 \to  \scrO_{X_0} \to a_*\Big( \scrO_{\hY} \oplus \bigoplus_i \scrO_{Z_i}\Big) \to j_*\Big(\bigoplus_i\scrO_{D_i}\Big) \to 0.$$
Then (i) and and the exact sequence in (ii) follow from the fact that the dual complex of $X_0$ is contractible, $H^1(\hY; \scrO_{\hY}) \to H^1(D_i; \scrO_{D_i})$ is injective,  and $H^2(\hY; \scrO_{\hY}) = H^2(Z_i; \scrO_{Z_i}) =0$. The last statement in (ii) follows since either $H^1(\hY; \scrO_{\hY}) =0$ (the rational or Enriques cases) or the composite map $H^1(\hY; \scrO_{\hY})\to H^1(D_i; \scrO_{D_i}) $ is injective for every $i$ (the elliptic ruled case).  

Since $H^0(\hY; \Omega^2_{\hY}) = H^0(Z_i; \Omega^2_{Z_i}) =0$, (iii) follows from the isomorphism 
$$\Omega^2_{X_0}/\tau^2_{X_0} \cong a_*\Big( \Omega^2_{\hY} \oplus \bigoplus_i\Omega^2_{Z_i}\Big).$$ 

Similarly, there is an exact sequence
$$0 \to \Omega^1_{X_0}/\tau^1_{X_0} \to a_*\Big( \Omega^1_{\hY} \oplus \bigoplus_i\Omega^1_{Z_i}\Big) \to j_*\Big( \bigoplus_i\Omega^1_{D_i}\Big)\to 0,$$
and (iv) is a consequence of the associated long exact cohomology sequence.
 \end{proof}

As for the weight filtration, it is determined by the Mayer-Vietoris spectral sequence for $H^2(X_0;\Zee)$, which has $E_1$ page

\medskip
\begin{center}
\begin{tabular}{|c|c|c}
$H^4(\hY)\oplus \bigoplus_iH^4(Z_i)$ &{} & {}   \\ \hline
$H^3(\hY)$ & {} & {}   \\ \hline
$H^2(\hY)\oplus\bigoplus_iH^2(Z_i)$ & $\bigoplus_iH^2(D_i)$  & {}  \\ \hline
$H^1(\hY)$ & $\bigoplus_iH^1(D_i)$ &{}  \\ \hline
$H^0(\hY)\oplus \bigoplus_iH^0(Z_i)$ & $\bigoplus_iH^0(D_i)$ & {}  \\ \hline

\end{tabular}

\end{center}

\medskip
\noindent (all coefficients $\Zee$). Here we use the fact that the $Z_i$ are del Pezzo surfaces and hence $H^1(Z_i;\Zee) = H^3(Z_i;\Zee)= 0$. This spectral sequence degenerates at $E_2$. In fact,  the spectral sequence simplifies to the long exact sequence on cohomology coming from
$$0 \to \Zee_{X_0} \to a_*\Big(\Zee_{\hY} \oplus \bigoplus_i\Zee_{Z_i} \Big) \to j_*\Big( \bigoplus_i\Zee_{D_i}\Big) \to 0.$$
Clearly the map $H^0(\hY)\oplus \bigoplus_iH^0(Z_i) \to\bigoplus_iH^0(D_i)$ is surjective.
 The maps $H^2(Z_i;\Zee) \to H^2(D_i;\Zee)$ are surjective, since there exists a divisor in $Z_i$ (for example an exceptional curve) which has intersection number one with $D_i$.  Thus we obtain:
 
 \begin{lemma}\label{genMV} There is an exact sequence
 $$0  \to H^1(\hY;\Zee) \to \bigoplus_iH^1(D_i;\Zee) \to H^2(X_0; \Zee) \to  H^2(\hY;\Zee)\oplus\bigoplus_iH^2(Z_i;\Zee)\to \bigoplus_iH^2(D_i;\Zee) \to 0,$$
 which is an exact sequence of mixed Hodge structures, suitably interpreted if $H^2(\hY;\Zee)$ has torsion. In particular, if $\hY$ is not an Enriques surface, 
 \begin{align*}
 W_1H^2(X_0; \Zee) &\cong \bigoplus_iH^1(D_i;\Zee)/\im H^1(\hY;\Zee);\\
 W_2H^2(X_0; \Zee)/W_1H^2(X_0; \Zee) &\cong \Ker\{H^2(\hY;\Zee)\oplus\bigoplus_iH^2(Z_i;\Zee)\to \bigoplus_iH^2(D_i;\Zee)\}, 
 \end{align*} 
 where  $H^1(\hY;\Zee) $ and $\bigoplus_iH^1(D_i;\Zee) $ have their usual weight one (pure)   Hodge structures and the Hodge structure $W_2H^2(X_0; \Zee)/W_1H^2(X_0; \Zee)$ is pure of weight two and  type $(1,1)$.   \qed
 \end{lemma}
 
 If $\hY$ is either a rational or an Enriques surface, $H^1(\hY;\Zee) =0$. However, if  $\hY$ is an elliptic ruled surface over the elliptic curve $B$, then $H^1(\hY;\Zee)\cong H^1(B;\Zee)$.

 We turn now to the  definition of the Jacobian  $JW_1H^2(X_0; \Zee)$:
 
 \begin{definition}\label{defJac}  Let $H$ be an effective weight one Hodge structure. Then the \textsl{Jacobian} $JH$ is the complex torus $H^{0,1}/H_{\Zee}$. The functor $J$ defines a covariant functor on the category of effective weight one Hodge structures. In particular,  if $H_1\subseteq H_2$ is an inclusion of  effective weight one Hodge structures of the same rank,   then there is an isogeny of complex tori $JH_1 \to JH_2$ giving an exact sequence
 $$0 \to (H_2)_\Zee/(H_1)_\Zee \to JH_1 \to JH_2 \to 0.$$
 On the other hand, if $C$ is a curve,  and writing $JC$ for $JH^1(C)$,  $JC\cong \Pic^0C$ and hence the functor $J$ is  can also be viewed as contravariant  with respect to morphisms of  curves. 
  \end{definition}
 
In the rational case,   the following is immediate from the Mayer-Vietoris  exact sequence:

\begin{lemma}\label{rationalW1} Suppose that $\hY$ is a rational surface. Then $JW_1H^2(X_0; \Zee)\cong  JD_1\oplus JD_2$. \qed
\end{lemma}

The   Enriques case is more subtle. 

\begin{lemma}\label{EnriquesW1}  Suppose that $\hY$ is an Enriques surface. Then $H^2(X_0; \Zee)$ is torsion free. Hence the image of $H^1(D_1;\Zee) \oplus H^1(D_2;\Zee)$  is contained in a saturated  overlattice $W_1H^2(X_0; \Zee)=W_1$ and has index $2$ in $W_1$, and there is an exact sequence
$$0 \to W_1 \to H^2(X_0; \Zee) \to \Ker\{\overline{H}^2(\hY;\Zee)\oplus\bigoplus_iH^2(Z_i;\Zee)\to \bigoplus_iH^2(D_i;\Zee)\}\to 0.$$
where $\overline{H}^2(\hY;\Zee)$ is the quotient of $H^2(\hY;\Zee)$ by the torsion subgroup. Finally, let $\eta\in \Pic^0\hY$ be the $2$-torsion line bundle, and identify $\eta$ with its image in $JD_1\oplus JD_2$. Then the map $JDi \to JW_1$ is injective for $i=1,2$ and 
$$JW_1\cong (JD_1\oplus JD_2)/\langle \eta\rangle.$$
\end{lemma}
\begin{proof}  By the Mayer-Vietoris exact sequence, the torsion subgroup of  $H^2(X_0; \Zee)$ is either trivial or isomorphic to $\Zee/2\Zee$. If it is isomorphic to $\Zee/2\Zee$, then the natural map $H^2(X_0; \Zee) \to  H^2(\hY; \Zee)$ is an isomorphism on torsion subgroups. Thus, it suffices to prove that there is no connected \'etale cover $\widetilde{X}_0$ of $X_0$ which induces the (blown up) $K3$ cover $\widetilde{Z}$ of $\hY$. Using \cite[VIII.17]{BHPV} as a general reference on Enriques surfaces, let $Y_0$ be the minimal model of $\hY$ and let $\overline{D}_1$, $\overline{D}_2$ be the images of the elliptic curves $D_1, D_2$ in $Y_0$. Then $\overline{D}_1\cdot  \overline{D}_2 =1$, so the cohomology classes of $\overline{D}_1$ and $\overline{D}_2$ are primitive. Then the $\overline{D}_i$ are multiple fibers in two different elliptic fibrations and the inverse image of $D_i$ in $\widetilde{Z}$ are connected. But the $Z_i$ are simply connected, so the cover $\widetilde{X}_0$ induces disconnected covers of the $Z_i$ and hence of $D_i$. This is a contradiction. 

The torsion subgroup $\langle \eta\rangle$  of $H^2(\hY; \Zee)$ has order $2$ and lies in the kernel of the homomorphism $H^2(\hY; \Zee) \to H^2(D_1;\Zee)\oplus H^2(D_2;\Zee)$. Hence it is in the image of
$H^2(X_0; \Zee)$. Since $H^2(X_0; \Zee)$  is torsion free and $W_1$ is the saturation of the image of $H^1(D_1;\Zee)\oplus H^1(D_2;\Zee)$, it follows that $W_1/H^1(D_1;\Zee)\oplus H^1(D_2;\Zee)$ has order $2$ and its image in $H^2(\hY; \Zee)$ is $\langle \eta\rangle$. In particular, the kernel of $JD_1\oplus JD_2 \to JW_1$ has order $2$. To see that this kernel is $\langle \eta\rangle$,  where $\eta$ is identified with   
$$(\eta|D_1, \eta|D_2)\in \Pic ^0D_1 \oplus \Pic^0D_2 \cong JD_1\oplus JD_,$$
 the Mayer-Vietoris sequences for $\Zee$,  $\scrO_{X_0}$, and $\scrO_{X_0}^*$ give a commutative diagram
$$\begin{CD}
H^1(D_1;\Zee)\oplus H^1(D_2;\Zee) @>>> H^1(D_1;\scrO_{D_1})\oplus H^1(D_2;\scrO_{D_2}) @>>>H^1(D_1;\scrO_{D_1}^*)\oplus H^1(D_2;\scrO_{D_2}^*)\\
@VVV @VVV @VVV \\
H^2(X_0;\Zee) @>>> H^2(X_0;\scrO_{X_0}) @>>> H^2(X_0;\scrO_{X_0}^*).
\end{CD}$$
Viewing   $(\eta|D_1, \eta|D_2)$   as a class in $H^1(D_1;\scrO_{D_1}^*)\oplus H^1(D_2;\scrO_{D_2}^*)$, it maps to $0$ in $H^2(X_0;\scrO_{X_0}^*)$ since it is in the image of $H^1(\hY; \scrO_{\hY}^*)$. Thus, if $(\eta'|D_1, \eta'|D_2)$ is a lift of $(\eta|D_1, \eta|D_2)$ to
$$H^1(D_1;\Q)\oplus H^1(D_2;\Q) \subseteq  H^1(D_1;\scrO_{D_1})\oplus H^1(D_2;\scrO_{D_2}),$$  then  the  image of $(\eta'|D_1, \eta'|D_2)$ in  $ H^2(X_0;\Q) \subseteq H^2(X_0;\scrO_{X_0})$ lies in $H^2(X_0;\Zee)$, giving an index two overlattice of the image of $H^1(D_1;\Zee)\oplus H^1(D_2;\Zee)$. Unwinding the identifications as in Definition~\ref{defJac}, it follows that $(\eta|D_1, \eta|D_2)$ is in the kernel of the map  $ JD_1\oplus JD_2\to JW_1$ and hence that the kernel is exactly equal to $\langle (\eta|D_1, \eta|D_2)\rangle =  \langle \eta\rangle$.  Finally, this implies that $ JD_i\to JW_1$ is injective for $i=1,2$. 
\end{proof} 

 Next we collect some basic facts about the invariants in  the elliptic ruled case: 

\begin{lemma}\label{ellipticruledW1} Let  $\hY$ be the blowup of an elliptic ruled surface over the base $B$.  Then:
\begin{enumerate}
\item[\rm(i)]  $H^2(X_0; \Zee)$ is torsion free.    
\item[\rm(ii)]    $W_1 \cong  \left( \bigoplus_iH^1(D_i;\Zee)\right)/ H^1(\hY;\Zee)$, and, in the notation of Theorems~\ref{211case} and \ref{111case},
$$JW_1 \cong \begin{cases} J\Gamma  \oplus JB, &\text{if $(m_1, m_2, m_2) = (2,1,1)$};\\
J\Gamma_1 \oplus J\Gamma_2 &\text{if $(m_1, m_2, m_2) = (1,1,1)$}.
\end{cases}$$
\end{enumerate}
\end{lemma}
\begin{proof} 
Regardless of whether $(m_1, m_2, m_2) = (2,1,1)$ or $(m_1, m_2, m_2) = (1,1,1)$, there exists at least one $i$ such that $D_i$ is the proper transform of a section.  Hence, for such an $i$,   $H^1(\hY;\Zee) \to  H^1(D_i;\Zee)$ is an isomorphism. Looking for example at the case $(m_1, m_2, m_2) = (2,1,1)$, we have
$$\bigoplus_iH^1(D_i;\Zee) \cong H^1(\Gamma;\Zee)\oplus H^1(\sigma_2; \Zee)\oplus H^1(\sigma_3; \Zee),$$
and the natural map $ H^1(\hY;\Zee) \to  H^1(\sigma_3; \Zee)$, say, is an isomorphism. Thus the map
$$H^1(\Gamma;\Zee)\oplus H^1(\sigma_2; \Zee) \to  \operatorname{Coker} \{H^1(\hY;\Zee) \to \bigoplus_iH^1(D_i;\Zee) \}$$
is an isomorphism. Since $H^2(\hY;\Zee)$ and $H^2(Z_i;\Zee)$ are torsion free,  $H^2(X_0; \Zee)$ is torsion free,  and the image of  $\bigoplus_iH^1(D_i;\Zee)$ in $H^2(X_0;\Zee)$ is a saturated sublattice. Thus $W_1 \cong H^1(\Gamma;\Zee)\oplus H^1(\sigma_2; \Zee) \cong H^1(\Gamma;\Zee)\oplus H^1(B; \Zee)$ and therefore $JW_1 \cong J\Gamma  \oplus JB$. The case where $(m_1, m_2, m_2) = (1,1,1)$ is similar.  
\end{proof} 

 Concentrating attention on the case $(m_1, m_2, m_2) = (1,1,1)$, which will be the main case of interest, we have the following more intrinsic formulation:
 
 \begin{lemma}\label{interpretJW1}  Suppose that $(m_1, m_2, m_2) = (1,1,1)$. Then in the notation of Theorem~\ref{111case},
 $$JW_1 \cong (J\Gamma_1 \oplus J\Gamma_2\oplus J\sigma)/JB.$$
 Thus in particular:
 \begin{enumerate}
 \item[\rm(i)] The induced homomorphisms $J\sigma \to JW_1$ and $J\Gamma_i\to JW_1$ are injective.
 \item[\rm(ii)] The induced homomorphism  $J\Gamma_1 \oplus J\Gamma_2 \to JW_1$ is an isomorphism.
 \item[\rm(iii)] The induced homomorphism  $J\Gamma_i  \oplus  J\sigma \to JW_1$ is surjective, and its kernel is $\langle  \eta  \rangle $, where  $\eta = (\eta_j', \eta_j)$, $j\neq i$, for $\eta_j$ and $\eta_j'$ certain $2$-torsion points on $J\sigma$ and $J\Gamma_i$ respectively. Hence
 $$JW_1 \cong J\Gamma_1 \oplus J\Gamma_2 \cong (J\sigma \oplus J\Gamma_i)/\langle  \eta  \rangle .$$
 \end{enumerate}
 \end{lemma}
 \begin{proof} (i) Since $\Gamma_i\to B$ is an isogeny of degree $2$, $\Ker \{JB\to J\Gamma_i\} =\langle \eta_i\rangle$, where $\eta_1$ and $\eta_2$ are two distinct $2$-torsion points on $B$. If $\xi \in J\sigma$ maps to $0$, then $\xi  \in \Ker \{JB\to J\Gamma_1\}\cap \Ker \{JB\to J\Gamma_2\} =\{0\}$. Thus $J\sigma \to JW_1$ is injective. The fact $J\Gamma_i\to JW_1$ is injective follows from the stronger statement (ii), which was noted in Lemma~\ref{ellipticruledW1}. To see (iii), a point $(\beta', 0,\beta) \in  J\Gamma_1\oplus J\Gamma_2\oplus J\sigma $ is in the image of $JB$ $\iff$ $\beta \in \Ker \{JB\to J\Gamma_2\}$ and $\beta'$ is the image of $\beta$ in $J\Gamma_1$. Thus $\Ker \{J\Gamma_1\oplus J\sigma  \to JW_1\} =  \langle  (\eta_2', \eta_2)\rangle$, where $\Ker \{JB\to J\Gamma_2\} =\langle \eta_2\rangle$ and $\eta_2'$ is the (nonzero) image of $\eta_2$ in $J\Gamma_1$.
 \end{proof} 
 
 \begin{remark} (i) In case $(m_1, m_2, m_2) = (1,1,1)$, the two different descriptions of $JW_1$ as $J\Gamma_1 \oplus J\Gamma_2$ and as $(J\sigma \oplus J\Gamma_i)/\langle  \eta  \rangle $ reflect the fact that the singular I-surface $Y$ can deform either to an singular I-surface with two simple elliptic singularities whose minimal resolution is  a rational surface or to one whose minimal resolution is a blown up Enriques surface. 
 
 \smallskip
 \noindent (ii) There is a similar picture  for  $(m_1, m_2, m_2) = (2,1,1)$.  In this case, $JW_1 \cong J\Gamma  \oplus JB$ as in Lemma~\ref{ellipticruledW1}. There is also the surjective  homomorphism 
$J\sigma_2 \oplus J\sigma_3\cong JB \oplus JB \to JW_1$, with kernel $\langle (\eta_1, \eta_1)\rangle=\langle \eta \rangle$, where $\Ker \{JB\to J\Gamma\} =\langle \eta_1\rangle$. Thus
$$JW_1 \cong J\Gamma  \oplus JB \cong (JB \oplus JB)/\langle \eta \rangle.$$
 \end{remark}
 
\subsection{A local Torelli theorem} By Lemma~\ref{Hodgefilt},  the differential of the period map at $X_0$ is given by the homomorphism
 $$H^1(X_0; T^0_{X_0}) \to \Hom (F^1H^2(X_0), F^0H^2(X_0)/F^1H^2(X_0)) = \Hom (H^1(X_0; \Omega^1_{X_0}/\tau^1_{X_0}) , H^2(X_0; \scrO_{X_0}))$$
 induced by cup product. In the following, we will just deal with the elliptic surface case and multiplicities $(1,1,1)$ and prove a local Torelli theorem, but  a similar result holds for the case of multiplicities $(2,1,1)$ with a slightly more complicated argument.  Intuitively, the differential of variation of mixed Hodge structure determines the first order deformations of the base curve $B$ as well as the differential of the period map for the anticanonical pairs $(Z_i, D_i)$ and so is injective by local Torelli for anticanonical pairs.  In the next section, we will show directly that the period map has degree one onto its image in this case. Of course, this generic global Torelli result immediately implies that the  local Torelli theorem  holds generically. However, it seemed worthwhile to give a direct argument, although   some details will just be sketched. 
 
 \begin{theorem}\label{localTorelli} In the elliptic ruled case with multiplicities $(1,1,1)$, let $\Xi$ be the family of locally trivial deformations of $X_0$ keeping the $d$-semistability condition and let $x_0\in \Xi$ correspond to the surface $X_0$. Then the differential of the period map for the variation of mixed Hodge structure on $H^2(X_0)$ defined by $\Xi$  is injective.
 \end{theorem}
 \begin{proof}   Proposition~\ref{tangdagger} identifies the tangent space $T_{\Xi,x_0}$ with   $\Big(\bigoplus_iH^1(Z_i;  T_{Z_i}(-\log  D_i))\Big)^{\dagger} $  as  defined in Definition~\ref{defdagger}. 
Suppose that $\theta = (\theta_1, \theta_2, \theta_3) \in \Big(\bigoplus_iH^1(Z_i;  T_{Z_i}(-\log  D_i))\Big)^{\dagger} $ satisfies: $\theta \smile \mu = 0$ for all $\mu \in H^1(X_0; \Omega^1_{X_0}/\tau^1_{X_0})$. First, if $\partial \colon \bigoplus_iH^0(D_i; \Omega^1_{D_i}) \to H^1(X_0; \Omega^1_{X_0}/\tau^1_{X_0})$ is the coboundary, it is easy to check that 
 $$(\theta_1, \theta_2, \theta_3) \smile \partial(\psi_1, \psi_2, \psi_3) = \sum_i\partial(\iota(\theta_i)\smile \psi_i),$$
 where $\iota(\theta_i)\smile \psi_i \in H^1(D_i; T_{D_i}) \otimes H^0(D_i; \Omega^1_{D_i}) \cong H^1(D_i; \scrO_{D_i})$ and $\partial(\iota(\theta_i)\smile \psi_i)$ denotes its image in $H^2(X_0; \scrO_{X_0})$ via the coboundary from the exact sequence in Lemma~\ref{Hodgefilt}(ii). Choosing $\psi_\ell \neq 0$ for exactly one $\ell$, it follows that $\iota(\theta_\ell)=0$ since the restriction of $\partial$ to  the summand $H^1(D_\ell; \scrO_{D_\ell})$ is injective, again by Lemma~\ref{Hodgefilt}(ii).  Hence $\iota(\theta_i) =0$ for all $i$ since $\theta   \in \Big(\bigoplus_iH^1(Z_i;  T_{Z_i}(-\log  D_i))\Big)^{\dagger} $. Thus $\theta  \in \Big(\bigoplus_iH^1(Z_i;  T_{Z_i}(-\log  D_i))\Big)^{\ddagger} $, the subspace defined in Definition~\ref{defdagger}. The above formula also shows that cup product induces a well-defined pairing
 $$ \Big(\bigoplus_iH^1(Z_i;  T_{Z_i}(-\log  D_i))\Big)^{\ddagger}  \otimes \Big(H^1(X_0; \Omega^1_{X_0}/\tau^1_{X_0})/\im \bigoplus_iH^0(D_i; \Omega^1_{D_i})\Big) \to H^2(X_0; \scrO_{X_0}).$$

Next, fix $\ell$ and define
 $$H^1(Z_\ell; \Omega^1_{Z_\ell})_0 = \Ker\{H^1(Z_\ell; \Omega^1_{Z_\ell}) \to H^1(D_\ell; \Omega^1_{D_\ell})\}.$$
 Then from the exact sequence
 $$0 \to \Omega^1_{Z_\ell}(\log D_\ell)(-D_\ell) \to \Omega^1_{Z_\ell} \to \Omega^1_{D_\ell} \to 0,$$
 it follows that there is an exact sequence
 $$0 \to H^0(D_\ell; \Omega^1_{D_\ell}) \to H^1(Z_\ell; \Omega^1_{Z_\ell}(\log D_\ell)(-D_\ell)) \to H^1(Z_\ell; \Omega^1_{Z_\ell})_0 \to 0.$$
By Lemma~\ref{Hodgefilt}(iv),  there is a map $H^1(Z_\ell; \Omega^1_{Z_\ell})_0 \to H^1(X_0; \Omega^1_{X_0}/\tau^1_{X_0})/\im \bigoplus_iH^0(D_i; \Omega^1_{D_i})$ and hence a composed map
 $$ H^1(Z_\ell; \Omega^1_{Z_\ell}(\log D_\ell)(-D_\ell)) \to H^1(Z_\ell; \Omega^1_{Z_\ell})_0 \to H^1(X_0; \Omega^1_{X_0}/\tau^1_{X_0})/\im \bigoplus_iH^0(D_i; \Omega^1_{D_i}).$$
 
Again for a fixed $\ell$, define  $H^1(Z_\ell;  T_{Z_\ell}(-\log  D_\ell)) _0 =\Ker \{\iota\colon H^1(Z_\ell;  T_{Z_\ell}(-\log  D_\ell)) \to H^1(D_\ell; T_{D_\ell})\}$ as in Definition~\ref{defdagger}.  Thus there is an inclusion 
$$H^1(Z_\ell;  T_{Z_\ell}(-\log  D_\ell)) _0 \subseteq \Big(\bigoplus_iH^1(Z_i;  T_{Z_i}(-\log  D_i))\Big)^{\ddagger} $$
by setting the remaining components of the direct sum to be $0$.  
Let 
$$\scrO_{Z_\ell}(-D_\ell)\subseteq \scrO_{X_0} = \Ker \{\bigoplus_i\scrO_{Z_i} \to \bigoplus_i\scrO_{D_i}\}$$ be the inclusion obtained by extending a section of $\scrO_{Z_\ell}(-D_\ell)$ by $0$ on the remaining components and let $H^2(Z_\ell; \scrO_{Z_\ell}(-D_\ell)) \to H^2(X_0; \scrO_{X_0})$ be the induced homomorphism.  Then  the following diagram is commutative (where the horizontal maps are given by cup product):
 $$\begin{CD}
 H^1(Z_\ell;  T_{Z_\ell}(-\log  D_\ell)) _0 \otimes   H^1(Z_\ell; \Omega^1_{Z_\ell}(\log D_\ell)(-D_\ell))  @>>> H^2(Z_\ell; \scrO_{Z_\ell}(-D_\ell))\\
 @VVV  @VVV \\
 \Big(\bigoplus_iH^1(Z_i;  T_{Z_i}(-\log  D_i))\Big)^{\ddagger}  \otimes   \Big(H^1(X_0; \Omega^1_{X_0}/\tau^1_{X_0})/\im \bigoplus_iH^0(D_i; \Omega^1_{D_i})\Big) @>>> H^2(X_0; \scrO_{X_0}).
 \end{CD}$$

 \begin{lemma}\label{inclinj}   The   homomorphism $H^2(Z_\ell; \scrO_{Z_\ell}(-D_\ell)) \to H^2(X_0; \scrO_{X_0})$ is injective.
 \end{lemma} 
 \begin{proof} There is a commutative diagram
 $$\begin{CD}
 0 @>>> \scrO_{Z_\ell}(-D_\ell)  @>>> \scrO_{Z_\ell} @>>> \scrO_{D_\ell} @>>> 0\\
 @. @VVV @VVV @VVV @. \\
   0 @>>> \scrO_{X_0}   @>>> a_*\Big(\scrO_{\hY} \oplus  \bigoplus_i\scrO_{Z_i} \Big) @>>> j_*\Big(\bigoplus_i \scrO_{D_i}\Big) @>>> 0.
 \end{CD}$$
 Thus there is a commutative diagram
  $$\begin{CD}
  H^1(D_\ell;  \scrO_{D_\ell}) @>>> \bigoplus_i H^1(D_i; \scrO_{D_i}) \\
  @V{\cong}VV @VVV \\
  H^2(Z_\ell; \scrO_{Z_\ell}(-D_\ell)) @>>> H^2(X_0; \scrO_{X_0}).
   \end{CD}$$
   By Lemma~\ref{Hodgefilt}(ii), the composite map  $H^1(D_\ell;  \scrO_{D_\ell}) \to  \bigoplus_i H^1(D_i; \scrO_{D_i}) \to H^2(X_0; \scrO_{X_0})$ is injective. Hence the map $H^2(Z_\ell; \scrO_{Z_\ell}(-D_\ell)) \to H^2(X_0; \scrO_{X_0})$ is injective as well. 
 \end{proof}

 Now suppose that $\theta = (\theta_1, \theta_2, \theta_3) \in  \Big(\bigoplus_iH^1(Z_i;  T_{Z_i}(-\log  D_i))\Big)^{\ddagger}$ is such that  $\theta  \smile \mu =0$ for all $\mu$.  Fixing $\ell$, choose $\mu$ to be the image of $(\mu_1, \mu_2,\mu_3) \in 
  \bigoplus_iH^1(Z_i; \Omega^1_{Z_i}(\log D_i)(-D_i))$ where  $\mu_j =0$ for $j\neq \ell$. Then $\theta \smile \mu \in H^2(X_0; \scrO_{X_0})$ is the image of  $\theta_\ell \smile \mu_\ell \in H^2(Z_\ell; \scrO_{Z_\ell}(-D_\ell))$. Since $H^2(Z_\ell; \scrO_{Z_\ell}(-D_\ell)) \to  H^2(X_0; \scrO_{X_0})$ is injective by  Lemma~\ref{inclinj}, it follows that $\theta_\ell \smile \mu_\ell = 0$  for all $\mu_\ell \in H^1(Z_\ell; \Omega^1_{Z_\ell}(\log D_\ell)(-D_\ell))$. However,  Serre duality gives a perfect pairing 
 $$H^1(Z_\ell;  T_{Z_\ell}(-\log  D_\ell)) \otimes H^1(Z_\ell; \Omega^1_{Z_\ell}(\log D_\ell)(-D_\ell)) \to H^2(Z_\ell;\scrO_{Z_\ell}(-D_\ell))= H^2(Z_\ell;K_{Z_\ell}).$$  Hence   $\theta_\ell =0$  for all $\ell$. Thus  $\theta =0$,  so that  the differential of the period map is injective.
  \end{proof} 

\section{The limiting mixed Hodge structure and the extended period map}

\subsection{The Clemens-Schmid exact sequence for $X_0$: the rational or Enriques cases} 

First assume that $\hY$ is a rational or Enriques surface, so that $k=2$. In particular, there are two monodromy matrices $N_1$, $N_2$ arising from the deformation theory of $X_0$. Referring to Theorem~\ref{versal}, taking the diagonal embedding of $\Delta$ in $\Delta^2$ gives a smoothing $\mathcal{X}\to \Delta$ of $X_0$ for which the monodromy is $N = N_1 + N_2$. Because of the somewhat special nature of the $d$-semistable surface $X_0$, the total space  $\mathcal{X}$ is smooth. 

The main point is then the following:

\begin{theorem}\label{describeMHS1}  In the above notation, let $X_t$ be a general fiber of $\mathcal{X}\to \Delta$ and let $H^2_{\text{\rm{lim}}}= H^2_{\text{\rm{lim}}}(X_t;\Zee)$ denote the limiting mixed Hodge structure (with $\Zee$-coefficients).
\begin{enumerate}
\item[\rm(i)]  $\im N$ has rank $4$.  
\item[\rm(ii)] The following  is an exact sequence of mixed Hodge structures over $\Zee$:
$$H_4(X_0;\Zee) \to  H^2(X_0; \Zee) \to H^2_{\text{\rm{lim}}} \xrightarrow{N} H^2_{\text{\rm{lim}}}.$$
\item[\rm(iii)] The image of $H_4(X_0;\Zee)$ is spanned by classes $\xi_1$, $\xi_2$, which restrict to the classes $$([D_i], -[D_i])\in \overline{H}^2(\hY;\Zee) \oplus H^2(Z_i; \Zee)$$
and which span a primitive isotropic subspace of $\overline{H}^2(\hY;\Zee) \oplus \bigoplus_iH^2(Z_i;\Zee)$. 
\item[\rm(iv)] The map $H^2(X_0; \Zee) \to H^2_{\text{\rm{lim}}}$ induces   an isomorphism $W_1H^2(X_0; \Zee) \to W_1H^2_{\text{\rm{lim}}}$. Thus   $W_1H^2_{\text{\rm{lim}}}\cong W_1H^2(X_0; \Zee)$ is a primitive integral subspace of $H^2_{\text{\rm{lim}}}$ of rank $4$. 
\item[\rm(v)] $W_2H^2_{\text{\rm{lim}}}/ W_1H^2_{\text{\rm{lim}}} \cong \{\xi_1, \xi_2\}^\perp/ \Zee\xi_1 + \Zee\xi_2$, a subquotient of $\overline{H}^2(\hY;\Zee) \oplus \bigoplus_iH^2(Z_i;\Zee)$. 
\end{enumerate}
\end{theorem} 
\begin{proof}  The exactness of the Clemens-Schmid sequence over $\Q$ shows that the weight filtration on  $H^2_{\text{\rm{lim}}}(X_t;\Q)$  is of the form $(W_1)_\Q \subseteq(W_2)_\Q \subseteq(W_3)_\Q$,  with $(W_1)_\Q\cong W_1H^2(X_0; \Q)$.  Thus, over $\Q$,  $\im N = (W_1)_\Q$ has dimension $4$.   This implies (i).  To see (ii), let  $\mathcal{X}^* = \mathcal{X}-X_0$. Arguing as in \cite[3.5]{FriedmanTorelli}, the Wang sequence and the exact sequence of the pair $(\mathcal{X}, \mathcal{X}^*)$ give two exact sequences
\begin{gather*}
H^1_{\text{\rm{lim}}} = 0 \to H^2(\mathcal{X}^*; \Zee) \to H^2_{\text{\rm{lim}}}  \xrightarrow{N}  H^2_{\text{\rm{lim}}}; \\
H_4(X_0; \Zee) \to H^2(X_0; \Zee) \to H^2(\mathcal{X}^*; \Zee) \to H_3(X_0; \Zee) .
\end{gather*}  
It will therefore suffice to show that $H_3(X_0; \Zee)$ is torsion free (in fact it is $0$ in this case).  This follows from the universal coefficient theorem and the fact that $H^4(X_0; \Zee)$ is torsion free, by the Mayer-Vietoris spectral sequence.
%$$\bigoplus_iH_3(D_i;\Zee) \to  H_3(\hY;\Zee) \oplus \bigoplus_iH_3(Z_i; \Zee) \to H_3(X_0; \Zee) \to  \bigoplus_iH_3(D_i;\Zee).$$
%Then $H_3(D_i;\Zee) =0$, $H_3(Z_i; \Zee) =0$ since $Z_i$ is rational, and $ H_3(\hY;\Zee)  =0$ either because $\hY$ is rational or it is an Enriques surface and then $ H_3(\hY;\Zee) $ is Poincar\'e dual to  $H^1(\hY;\Zee) =0$. 

Then (iii) follows by direct calculation, since $[D_i]$ is primitive in both $H^2(Z_i;\Zee)$ and in $\overline{H}^2(\hY; \Zee)$.  By looking at Hodge type, $W_1H^2(X_0; \Zee) \cap \im H_4(X_0; \Zee) = 0$ and hence the map $W_1H^2(X_0; \Zee) \to W_1H^2_{\text{\rm{lim}}}$ is injective. By (ii), the image is a primitive integral subspace of $H^2_{\text{\rm{lim}}}$ and hence is equal to $W_1H^2_{\text{\rm{lim}}}$ since both have rank $4$.  Finally, to see  (v), note that 
$$W_2H^2(X_0;\Zee)/W_1H^2(X_0;\Zee) = \Ker\{\overline{H}^2(\hY;\Zee)\oplus\bigoplus_iH^2(Z_i;\Zee)\to \bigoplus_iH^2(D_i;\Zee)\}= \{\xi_1, \xi_2\}^\perp,$$
using intersection pairing on the factors and the sum map from $H^4(\hY;\Zee) \oplus \bigoplus_iH^4(Z_i;\Zee)$ to $\Zee$. Also, $W_2H^2_{\text{\rm{lim}}} = \Ker N = \im H^2(X_0; \Zee)$. Then the exact sequence in (ii) identifies $W_2H^2_{\text{\rm{lim}}}/ W_1H^2_{\text{\rm{lim}}}$ with   $ \{\xi_1, \xi_2\}^\perp/ \Zee\xi_1 + \Zee\xi_2$. 
\end{proof}

The upshot is the following: Use the shorthand $W_i$ to denote $W_iH^2_{\text{\rm{lim}}}$, where the $W_i$ are primitive subspaces of $H^2_{\text{\rm{lim}}}$. The  limiting mixed Hodge structure on $H^2_{\text{\rm{lim}}}$ looks like
$$W_1 \subseteq W_2 \subseteq W_3,$$
where $W_1$ is a primitive integral isotropic subspace, $W_3/W_2 \cong W_1(-1)$, and $W_1$, $W_2$, and $W_2/W_1$ are computed from $W_2H^2(X_0)$.  In particular, $W_2$ is invariantly defined and has a well-defined integral structure.  Our main interest is in fact the mixed Hodge structure on \textit{primitive} cohomology, and we will use the notation $(W_2)_0$ and $(W_2/W_1)_0$ when we work with primitive cohomology.  Note that $(W_1)_0 = W_1$. Since $L =K_{X_t}$ on a general fiber, with limit line bundle $\omega_{X_0}$, we have:

\begin{corollary} $(W_2/W_1)_0 \cong \{\xi_1, \xi_2, [L]\}^\perp/ \Zee\xi_1 + \Zee\xi_2$. \qed
\end{corollary}

In terms of integral lattices, using \S\ref{smoothcase}, we have:

\begin{proposition}\label{latticeLambda}  $(W_2/W_1)_0$ is an even negative definite unimodular lattice of rank $24$. \qed
\end{proposition}

We will determine this lattice in the various cases in the next   section. Here, we recall Carlson's theory of extensions of mixed Hodge structures \cite{Carlson1}, \cite{Carlson2}: 

\begin{proposition} The mixed Hodge structure on $(W_2)_0$ is classified by a homomorphism 
$$\psi\colon (W_2/W_1)_0 \to JW_1,$$
 where $JW_1$ is the Jacobian of the weight one Hodge structure $W_1$. \qed
\end{proposition}

\begin{remark} Similarly, the mixed Hodge structure on $H^2(X_0)$ is classified by a homomorphism (also denoted by the same letter) $\psi\colon W_2H^2(X_0;\Zee)/W_1H^2(X_0;\Zee) \to JW_1$. However, the polarization condition implies that $\psi([L]) =0$ and the $d$-semistability condition implies that $\psi(\xi_1) = \psi(\xi_2) =0$. Thus, the extension of mixed Hodge structures coming from  $(W_2)_0$ determines and is determined by the corresponding extension coming from $H^2(X_0)$ subject to the above conditions.
\end{remark} 

The homomorphism $\psi$ corresponding to $H^2(X_0)$ can be described explicitly (\cite{Carlson0}, \cite[Theorem 2.4]{Carlson3}, and,  for a recent exposition which  covers the case where $\hY$ is not necessarily simply connected, \cite[Proposition 6.1]{EGW}). We begin with the rational case:

\begin{proposition}\label{extdata1} Suppose that $\hY$ is a rational surface. Given 
$$\lambda \in W_2H^2(X_0;\Zee)/W_1H^2(X_0;\Zee) \cong \Ker\{H^2(\hY;\Zee)\oplus\bigoplus_iH^2(Z_i;\Zee)\to \bigoplus_iH^2(D_i;\Zee)\},$$
represent $\lambda$ by holomorphic line bundles 
$$(L_\lambda, M^{(1)}_\lambda, M^{(2)}_\lambda) \in \Pic \hY \oplus \Pic Z_1\oplus \Pic Z_2$$
such that $\deg (L_\lambda |D_i) = \deg(M^{(i)}_\lambda|D_i)$. Then $\psi(\lambda)$ is identified with $(\lambda_1, \lambda_2) \in JD_1\oplus JD_2 \cong JW_1$, where
$\lambda_i \in JD_i \cong \Pic^0D_i$ is the line bundle $(L_\lambda^{-1} \otimes M^{(i)}_\lambda)|D_i$. \qed
\end{proposition}

\begin{remark} There is a similar description in case $\hY$ is an Enriques surface. In this case, a class 
$$\lambda \in \Ker\{\overline{H}^2(\hY;\Zee)\oplus\bigoplus_iH^2(Z_i;\Zee)\to \bigoplus_iH^2(D_i;\Zee)\} $$ lifts to a triple $(L_\lambda, M^{(1)}_\lambda, M^{(2)}_\lambda)$ as above, where however $L_\lambda$ is only unique modulo the $2$-torsion line bundle $\eta$. If $\lambda_i$ is defined as in Proposition~\ref{extdata1}, then $\psi(\lambda)$ is identified with  the image of $(\lambda_1, \lambda_2)$ in $(JD_1\oplus JD_2)/\langle\eta\rangle \cong JW_1$, which is independent of the choice of a lift. 
\end{remark} 

\subsection{The Clemens-Schmid exact sequence for $X_0$: the elliptic surface case} The analysis here is very similar to the other cases.  There are three monodromy matrices $N_1$, $N_2$, $N_3$  arising from the deformation theory of $X_0$. Referring to Theorem~\ref{versal}, taking the diagonal embedding of $\Delta$ in $\Delta^3$ gives a smoothing  $\mathcal{X}\to \Delta$ of $X_0$ for which $\mathcal{X}$ is smooth and the monodromy is $N = N_1 + N_2+N_3$. 

Arguments along the lines of the proof of Theorem~\ref{describeMHS1} then show:

\begin{theorem}\label{describeMHS2} In the above notation, let $X_t$ be a general fiber and let $H^2_{\text{\rm{lim}}}$ denote the limiting mixed Hodge structure (with $\Zee$-coefficients).
\begin{enumerate}
\item[\rm(i)]  $\im N$   has rank $4$.  
\item[\rm(ii)] The following  is an exact sequence of mixed Hodge structures over $\Zee$:
$$H_4(X_0;\Zee) \to  H^2(X_0; \Zee) \to H^2_{\text{\rm{lim}}} \xrightarrow{N} H^2_{\text{\rm{lim}}}.$$
\item[\rm(iii)] The image of $H_4(X_0;\Zee)$ is spanned by classes $\xi_1$, $\xi_2$, $\xi_3$, which restrict to the classes $$([D_i], -[D_i])\in H^2(\hY;\Zee) \oplus H^2(Z_i; \Zee)$$
and which span a primitive isotropic subspace of $H^2(\hY;\Zee) \oplus \bigoplus_iH^2(Z_i;\Zee)$. 
\item[\rm(iv)] The map $H^2(X_0; \Zee) \to H^2_{\text{\rm{lim}}}$ induces   an isomorphism $W_1H^2(X_0; \Zee) \to W_1H^2_{\text{\rm{lim}}}$, and the subspace  $W_1H^2_{\text{\rm{lim}}}$ is a primitive integral subspace of $H^2_{\text{\rm{lim}}}$. 
\item[\rm(v)] $W_2H^2_{\text{\rm{lim}}}/ W_1H^2_{\text{\rm{lim}}} \cong \{\xi_1, \xi_2, \xi_3\}^\perp/ \Zee\xi_1 + \Zee\xi_2$, a subquotient of $H^2(\hY;\Zee) \oplus \bigoplus_iH^2(Z_i;\Zee)$.  \qed
\end{enumerate}
\end{theorem}

\begin{corollary}\label{elllatticeLambda} $(W_2/W_1)_0 \cong \{\xi_1, \xi_2, \xi_3, [L]\}^\perp/ \Zee\xi_1 + \Zee\xi_2+ \Zee\xi_3$. Hence $(W_2/W_1)_0$ is a negative definite even unimodular lattice of rank $24$. \qed
\end{corollary}

In this case, the corresponding extension of mixed Hodge structures is described by a homomorphism $\psi$ as before. The recipe for $\psi$ is as follows \cite[Proposition 6.1]{EGW}:

\begin{proposition} Suppose that $\hY$ is an elliptic  surface. Given 
$$\lambda \in W_2H^2(X_0;\Zee)/W_1H^2(X_0;\Zee) \cong \Ker\{H^2(\hY;\Zee)\oplus\bigoplus_iH^2(Z_i;\Zee)\to \bigoplus_iH^2(D_i;\Zee)\},$$
choose   holomorphic line bundles 
$$(L_\lambda, M^{(1)}_\lambda, M^{(2)}_\lambda, M^{(3)}_\lambda) \in \Pic \hY \oplus \Pic Z_1\oplus \Pic Z_2 \oplus \Pic Z_3$$
which restrict to the class $\lambda$ on each component. In particular $\deg (L_\lambda|D_i) = \deg (M^{(i)}_\lambda)|D_i$, and $L_\lambda$ is well-determined up to the action of $\Pic^0\hY$. Then $\psi(\lambda)$ is the image of  
$$(\lambda_1, \lambda_2,\lambda_3) \in \left(JD_1\oplus JD_2 \oplus JD_3\right)/\Pic^0\hY\cong JW_1,$$ where
$\lambda_i \in JD_i \cong \Pic^0D_i$ is the line bundle $(L_\lambda^{-1} \otimes M^{(i)}_\lambda)|D_i$. \qed
\end{proposition}

\subsection{Picard-Lefschetz formulas}\label{PLfmlas}
We continue to use the notation $X_0 = \hY\amalg_{D_i}\left(\coprod_iZ_i\right)$.   Suppose that $X_t$ is a general deformation of $X_0$, in the sense that the total space $\mathcal{X}$ of the deformation is smooth. Recall the following formula due to Clemens \cite{Clemens}, \cite[3.8]{FriedmanTorelli}: Let $c_t\colon X_t \to X_0$ be the Clemens collapsing map, and set $U_i = c_t^{-1}(Z_i)$, $i=1,2,3$, $U_0 =  c_t^{-1}(\hY)$, and $\widetilde{D}_i = U_0 \cap U_i$. Then $\widetilde{D}_i $ is an $S^1$-bundle over $D_i$. For a pair of closed curves  $\{\alpha_i, \beta_i\}$ in $D_i$ representing a standard symplectic basis, let $\tilde{\alpha}_i$ and $\tilde{\beta}_i$ be the corresponding $S^1$-bundles (i.e.\ the tubes over the cycles $\alpha_i, \beta_i$). Thus $\tilde{\alpha}_i$ and $\tilde{\beta}_i$ define homology classes in $H_2(\widetilde{D}_i )$, $H_2(U_i)$, or $H_2(X_t)$. Viewing them as cycles in $H_2(X_t)$, define   $N_i\colon H_2(X_t)\to H_2(X_t)$ by the following formula:
$$N_i(\xi) =  \langle \xi, \tilde{\beta}_i\rangle \tilde{\alpha}_i - \langle \xi, \tilde{\alpha}_i\rangle \tilde{\beta}_i,$$
where we use $\langle \cdot , \cdot \rangle$ to denote intersection pairing.
Of course, there is an analogous formula for the action of $N_i \colon H^2(X_t)\to H^2(X_t)$ via Poincar\'e duality. The local arguments of \cite{Clemens} show the following:

\begin{proposition}\label{describemono}  If $X_t$ is a general deformation of $X_0$, the monodromy $N$ is $\sum_i N_i$. \qed
\end{proposition} 

\begin{remark} (i) More generally, suppose that $\mathcal{X} \to \Delta$ is a deformation of $X_0$ such that, locally around each point of $D_i$, the morphism has the form $xy = t^{k_i}$. Then the monodromy $N$ is $\sum_i k_iN_i$.

\smallskip
\noindent (ii) There are obvious compatibilities in the situation where we partially smooth $X_0$ to a $d$-semistable variety by smoothing some of the components of the singular locus, but not all.
\end{remark}

\begin{proposition}\label{monodromy}  The $\Zee$-span of the $N_i$ is a  primitive subgroup of the abelian group of all integer matrices.
\end{proposition} 
\begin{proof}
We shall just write down the proof in the case $k=3$ and $(m_1, m_2, m_3) = (1,1,1)$. Then we can reinterpret Lemma~\ref{interpretJW1} as follows:
The  cycles $\tilde{\alpha}_i, \tilde{\beta}_i\in H^2(X_t;\Zee)=H^2_{\text{\rm{lim}}}$, $1\le i \le 3$,  span a primitive $4$-dimensional isotropic subspace $W_1$ of $H^2_{\text{\rm{lim}}}$ which is a pure weight one sub-Hodge structure for the limiting MHS on $H^2_{\text{\rm{lim}}}$. From the explicit description of $N_i$,  $\im N_i$ is the $\Zee$-span of $\tilde{\alpha}_i, \tilde{\beta}_i$. By Lemma~\ref{interpretJW1}(i),  the $\im N_i$ are primitive integral subspaces of $W_1$ and correspond to sub-Hodge structures, via the composition
$$H^1(D_i;\Zee) \to W_1H^2(X_0;\Zee) \xrightarrow{\cong} W_1.$$
 For $i\neq j$,  the map  $\im N_i \oplus  \im N_j  \to W_1$ is an isomorphism over $\Q$.  More precisely,   $\tilde{\alpha}_1, \tilde{\beta}_1 , \tilde{\alpha}_2, \tilde{\beta}_2$ are a $\Zee$-basis for the lattice $W_1$, i.e.\  $\im N_1 \oplus  \im N_2  \to W_1$ is an isomorphism of weight one Hodge structures over $\Zee$. However, for $i = 1,2$, the image in $W_1$ of $\im N_i \oplus  \im N_3$ is a sublattice of index $2$ (and the sublattice  for $i=1$ is different from the sublattice  for $i=2$.)  The following lemma is a more explicit version of this statement:

\begin{lemma} There is a choice of the integral bases $\tilde{\alpha}_1, \tilde{\beta}_1 , \tilde{\alpha}_2, \tilde{\beta}_2, \tilde{\alpha}_3, \tilde{\beta}_3$ such that 
\begin{align*}
\tilde{\alpha}_3 &= 2\tilde{\alpha}_1 + \tilde{\alpha}_2;\\
\tilde{\beta}_3 &= \tilde{\beta}_1 + 2\tilde{\beta}_2.
\end{align*}
\end{lemma}
\begin{proof} Referring to Lemma~\ref{interpretJW1}, there exists an integral basis $\{\alpha_3, \beta_3\}$ for $H^1(D_3;\Zee) \cong H^1(B;\Zee)$ such that the $2$-torsion points $\eta_1$ and $\eta_2$ correspond to $\frac12\alpha_3$ and $\frac12\beta_3$ respectively. The overlattice $\operatorname{span}\{\alpha_1, \beta_1\}$ can then be taken to have as an integral basis $\alpha_1 = \frac12\alpha_3, \beta_1= \beta_3$. For  this basis, the homomorphism from  $\operatorname{span}\{\alpha_3, \beta_3\}$ to $\operatorname{span}\{\alpha_1, \beta_1\}$ is given by $\alpha_3 = 2\alpha_1$, $\beta_3 = \beta_1$. Making a similar argument for  $\operatorname{span}\{\alpha_2, \beta_2\}$, the homomorphism is given by $\alpha_3 =  \alpha_2$, $\beta_3 = 2\beta_2$. Thus, in the integral basis for $W_1 = H^1(\Gamma_1;\Zee)\oplus H^1(\Gamma_2;\Zee)$, $\alpha_3 = 2\alpha_1 + \alpha_2$ and $\beta_3 = \beta_1+ 2\beta_2$. A similar statement holds when  $\alpha_i, \beta_i$ are replaced by $\tilde{\alpha}_i, \tilde{\beta}_i$.
\end{proof}

Given the lemma, we must show that, given integers $\lambda_i$, $1\le i \le 3$, if $\sum_{i=1}^3\lambda_iN_i$ is divisible by $k$, then $k|\lambda_i$ for all $i$. Let $\tilde{\alpha}_i^*, \tilde{\beta}_i^*$ be   dual integral classes to $\tilde{\alpha}_i, \tilde{\beta}_i$ under cup product, i.e.\ $\langle  \tilde{\alpha}_i^*, \tilde{\alpha}_j \rangle  = \langle  \tilde{\beta}_i^*, \tilde{\beta}_j \rangle =\delta_{ij}$ and $\langle  \tilde{\alpha}_i^*, \tilde{\beta}_j \rangle =0$ for all $i,j$. Then 
$$\sum_{i=1}^3\lambda_iN_i(\tilde{\alpha}_2^*) = (-\lambda_2 -2\lambda_3)\tilde{\beta}_2 + (-\lambda_3)\tilde{\beta}_1.$$
Hence, $\sum_{i=1}^3\lambda_iN_i(\tilde{\alpha}_2^*)$ is divisible by $k$ $\iff$ $k|\lambda_3$ and $k|\lambda_2$. A similar argument shows that  $\sum_{i=1}^3\lambda_iN_i(\tilde{\beta}_1^*)$ is divisible by $k$ $\iff$ $k|\lambda_3$ and $k|\lambda_1$. Thus, 
if $\sum_{i=1}^3\lambda_iN_i$ is divisible by $k$, then $k|\lambda_i$ for all $i$.
\end{proof}

\subsection{The extended period map}\label{extendedperiodmap}  Let $\mathcal{M}$ denote the moduli space of I-surfaces, possibly with rational double points, and let $\Phi\colon \mathcal{M} \to \Gamma \backslash D$ be the period map and let $Z = \Phi(\mathcal{M})$ be its image. As usual, we ignore the issues caused by finite group actions coming from automorphisms of I-surfaces or from the rational double points. These can be handled in the usual way by imposing level structure, restricting to suitable dense open subsets, or adopting the point of view of orbifolds, i.e.\ suitable (analytic) stacks. The goal in this subsection is to describe the necessary properties of a compactification of $Z$ that we will need. 

First, let $\overline{\mathcal{M}}$ be the open subset of the KSBA compactification of $\mathcal{M}$ where the I-surfaces are allowed to have simple elliptic singularities as well as possibly RDP singularities. Let $\widetilde{\mathcal{M}}$ be the natural blowup of $\mathcal{M}$ where we replace the simple elliptic singularities by their $d$-semistable models as described in \S\ref{ssectsevsdss}. Thus $\widetilde{\mathcal{M}} - \mathcal{M}$ is a divisor with normal crossings (in the orbifold sense) and at most $3$ local branches of the boundary divisor  have a nonempty intersection.  Let $U \cong \Delta^k\times S$ be a  neighborhood a of a boundary point, with $k\le 3$, so that  $U\cap \mathcal{M} \cong (\Delta^*)^k\times S$.   If $N_i$ is the logarithm of monodromy around the $i^{\text{th}}$ coordinate hyperplane, then every $N$ in the associated cone satisfies $N^2 =0$.   Let $W_1\subseteq W_2 \subseteq W_3$ be the weight filtration defined by any $N$ in the interior of the cone. Then   $\im N_i$  is a  sub-Hodge structure of $W_1$. Thus $\dim W_1 =2$ if $k=1$ and $\dim W_1 = 4$ if $k=2,3$, and $\dim \im N_i =2$ in all cases. There is a distinguished codimension $3$ stratum $\mathcal{S}_{1,1,1}$ corresponding to the elliptic ruled case with multiplicities $(1,1,1)$. (This stratum is denoted $\mathfrak{N}_{1,1,1}$ in \cite{FPR1}, \cite{FPR2}.) The local monodromy in this case is described by Proposition~\ref{monodromy}. There is also a local Torelli theorem (Theorem~\ref{localTorelli}) for the variation of mixed Hodge structure along  $\mathcal{S}_{1,1,1}$.  A similar picture holds for the stratum $\mathcal{S}_{2,1,1}$ corresponding to the elliptic ruled case with multiplicities $(2,1,1)$.

Based on  the work of Usui \cite{Usuicomp}, Kato-Nakayama-Usui \cite{KU}, \cite{KNU},  Green-Griffiths-Robles \cite{GGR},   Deng-Robles \cite{DengRobles}, and work in progress of Deng-Robles, it is natural to conjecture the following: 

\begin{conjecture}\label{extendedperiodmapthm}  There exists an analytic space  $\overline{Z}$ which  is a partial compactification of the image $Z$ of the period map $\Phi \colon \mathcal{M} \to \Gamma\backslash D$ with the following properties:
\begin{enumerate}
\item[\rm(i)]  The boundary points of $\widetilde{\Phi}$ contain the information of the nilpotent orbit of the limiting mixed Hodge structure along with the unordered collection of the subspaces $\im N_i$.  
\item[\rm(ii)]   After replacing $\widetilde{\mathcal{M}}$ by  $\widehat{\mathcal{M}}$,  a sequence of toric blowups of  $\widetilde{\mathcal{M}}$ over boundary strata, the map  $\Phi \colon \mathcal{M} \to \Gamma\backslash D$ extends to a holomorphic map  $\widehat{\Phi}\colon \widehat{\mathcal{M}}\to \overline{Z}$. 
\end{enumerate}
\end{conjecture}

For our purposes, it is not sufficient to allow a toric blowup $\widehat{\mathcal{M}} \to \widetilde{\mathcal{M}}$. On the other hand, we have a good understanding of the monodromy around the components of $\widetilde{\mathcal{M}}- \mathcal{M}$. Thus, we further conjecture: 

\begin{conjecture}\label{extendedperiodmapthm2}  In addition to Conjecture~\ref{extendedperiodmapthm}, the following hold:
\begin{enumerate}
\item[\rm(iii)]  The map  $\Phi \colon \mathcal{M} \to \Gamma\backslash D$ extends to a holomorphic map  $\widetilde{\Phi}\colon \widetilde{\mathcal{M}}\to \overline{Z}$, i.e.\ no additional blowups are necessary. 
\item[\rm(iv)]  The compactification $\overline{Z}$ is an orbifold in a neighborhood of $\widetilde{\Phi}(\mathcal{S}_{1,1,1})$, and the complement $\overline{Z} -Z$ is an orbifold divisor with normal crossings in such a neighborhood, locally consisting of a union of $3$ (orbifold) smooth divisors. 
\item[\rm(v)]  The normal derivative of $\widetilde{\Phi}$ to $\mathcal{S}_{1,1,1}$ is injective at a general (smooth) point  of $\mathcal{S}_{1,1,1}$, hence the derivative of $\widetilde{\Phi}$ is injective at a general point  of $\mathcal{S}_{1,1,1}$.
\end{enumerate}
\end{conjecture}

\begin{remark}\label{fingpactions}  We make some remarks about the various finite groups arising on the moduli and Hodge theory side of this picture for a general point of $\mathcal{S}_{1,1,1}$. On the moduli side, if $Y$ is an I-surface with three simple elliptic singularities of multiplicity one, we have constructed a cover $\hat{T}$ of the weighted blowup of the germ of the miniversal deformation of $Y$, with covering group $W(E_8)\times W(E_8) \times W(E_8)$. As far as the three elliptic singularities of $Y$ are concerned, one is special: for the exceptional divisor in  the minimal resolution, it is the proper transform of a section, whereas the other two exceptional divisors are proper transforms of bisections. These two singular points are exchanged by monodromy along the family of elliptic curves parametrizing the base $B$ of the elliptic ruled surfaces. 

On the Hodge theory side, we have the weight filtration of the limiting  mixed Hodge structure:  $W_1 \subseteq (W_2)_0 \subseteq (W_3)_0 $. If $T$ acts trivially on the associated graded, then $T$ is unipotent. The Hodge structures on $W_1$ and $(W_3/W_2)_0$ are those coming from an elliptic curve and in general only have the automorphism $\pm \Id$. As for $(W_2/W_1)_0= \Lambda$, we have seen that, in the multiplicity $(1,1,1)$ case, its automorphism group as a lattice is the semidirect product of $W(E_8)\times W(E_8) \times W(E_8)$ with the symmetric group $\mathfrak{S}_3$.  Of course, the groups $W(E_8)\times W(E_8) \times W(E_8)$ on both sides match up in the natural way. 
The group $\mathfrak{S}_3$ does not act on the boundary components or limiting mixed Hodge structures corresponding to $\mathcal{S}_{1,1,1}$ because of the asymmetry noted in Lemma~\ref{interpretJW1}. Instead, $\mathfrak{S}_2$ acts in the natural way, corresponding to interchanging the labeling of the two singular points corresponding to proper transforms of bisections. 
\end{remark}

 \section{Analysis of the various strata}

\subsection{The possibilities for  $R(\Lambda) $} For the rest of this paper, we denote by $\Lambda$ the even negative definite unimodular lattice $(W_2/W_1)_0$. Recall that, unless $\hY$ is an elliptic ruled surface, i.e.\ $k=2$, we have
\begin{align*}
(W_2/W_1)_0&= \{\xi_1, \xi_2, [L]\}^\perp/\Zee\xi_1 + \Zee\xi_2 \\
&=\Ker \{\overline{H}^2_0(\hY;\Zee) \oplus H^2(Z_1; \Zee)\oplus H^2(Z_2; \Zee)\to H^2(D_1;\Zee) \oplus H^2(D_2;\Zee)\}/\Zee\xi_1 + \Zee\xi_2,
\end{align*}
where $\overline{H}^2_0(\hY;\Zee)$ denotes the orthogonal complement of $[L]$ in $\overline{H}^2(\hY;\Zee)$. There is a similar description in the elliptic ruled case.   By a case-by-case analysis, we will show the following theorem:

\begin{theorem}\label{listroots} The lattice $\Lambda$ satisfies $R(\Lambda) = E_8+ E_8 + E_8$ in all cases except for  the rational case with multiplicities $(2,2)$.  In other words, $R(\Lambda) = E_8+ E_8 + E_8$    in the Enriques    or elliptic ruled case  or in the rational case with multiplicities $(2,1)$ or $(1,1)$.  
\end{theorem} 

\begin{remark} The more detailed analysis of each of these cases will show that the limiting mixed Hodge structures are all of different types in an appropriate sense. More precisely, we have identified six different strata of I-surfaces with $2$ or $3$ simple elliptic singularities, giving limiting mixed Hodge structures of type $\lozenge_{0,2}$, and the asymptotic behavior of the period map for each such stratum determines the stratum.
\end{remark}

Before we turn to the  individual cases, we introduce the following notation: Given a function $\psi\colon  E_8+ E_8 + E_8 \to JD_1\oplus JD_2$,  write $\psi_{ij}$, $1\le i \le 3$, $1\le j\le 2$, for the corresponding function from the $i^{\text{\rm{th}}}$ copy of $E_8$ to the $j^{\text{\rm{th}}}$ factor of $JD_1\oplus JD_2$.

\subsection{The elliptic ruled  case with multiplicities $(1,1,1)$} In this case, the   lattice  $H^2(Z_i;\Zee)_0$ is isomorphic to  $\Lambda_{E_8} \subseteq \Lambda$. Hence $R(\Lambda) = E_8+ E_8 + E_8$, with the three summands corresponding to the lattices $H^2(Z_i;\Zee)_0$.

Then we have the following Torelli theorem:

\begin{theorem} Assume Conjectures~\ref{extendedperiodmapthm} and \ref{extendedperiodmapthm2}.  With notation as in \S\ref{extendedperiodmap}, if $x$ is a general point of $\mathcal{S}_{1,1,1}$, then $\widetilde{\Phi}^{-1}(\widetilde{\Phi}(x)) = \{x\}$. Moreover, the differential of $\widetilde{\Phi}$ is injective at $x$. 
\end{theorem}
\begin{proof} First, the image of the extended period map determines $JD_1\cong J\Gamma_1$, $JD_2\cong J\Gamma_2$, and $JD_3\cong J\sigma \cong JB$ as subvarieties of $JW_1$ in the usual notation. By Lemma~\ref{interpretJW1}, the pair $\{D_1,D_2\}$ is distinguished by the property that the map $JD_1\oplus JD_2 \to JW_1$ is an isomorphism as opposed to an isogeny of degree $2$. Thus $JD_3$ is also distinguished, so the image of the extended period map determines $B$, hence $Y_0$, along with $\Gamma_1$ and $\Gamma_2$. Thus $\hY$ is determined by the image of the extended period map. The primitive cohomology of each $Z_i$ determines the root system $E_8 + E_8 + E_8$. For each factor $E_8$, the corresponding extension homomorphism $\psi$ factors through the inclusion $JD_i \subseteq JW_1$ and thus determines the pair $(Z_i, D_i)$. The gluing map from $D_i\subseteq Z_i$ to $D_i\subseteq \hY$ is uniquely determined since $(D_i)^2_{Z_i}= 1$ and $(D_i)^2_{\hY}=  -1$. Thus $\widetilde{\Phi}^{-1}(\widetilde{\Phi}(x)) = \{x\}$. The second statement has been noted in Theorem~\ref{extendedperiodmap}(iv). 
\end{proof} 

\begin{remark} Without assuming Conjectures~\ref{extendedperiodmapthm} and \ref{extendedperiodmapthm2}, the above argument still proves a global Torelli theorem for the variation of mixed Hodge structure induced by the nilpotent orbits of the limiting mixed Hodge structures for points of $\mathcal{S}_{1,1,1}$, together with the data of the inclusions $JD_i \subseteq JW_1$.
\end{remark}

\subsection{The remaining cases where $R(\Lambda) = E_8+ E_8 + E_8$}\label{22rootsystemsect}

\subsubsection{The elliptic ruled  case with multiplicities $(2,1,1)$}   We freely use the notation of Theorem~\ref{211case}. First, the image of the extended period map determines $JD_1\cong J\Gamma$, $JD_2\cong J\sigma_2 \cong JB$, and $JD_3\cong J\sigma_3 \cong JB$. Note that $JD_2 \cong JD_3$, so this case is different from the multiplicity $(1,1,1)$ case above. In any case $JB$ is distinguished by the image of the extended period map. However, $\hY$ is not completely determined by $B$; there is a one-dimensional modulus coming from the choice of the two sections $\sigma_2$, $\sigma_3$. The primitive cohomology of each $Z_i$ determines the root system $E_8 + E_8 + E_7$. For each factor, the corresponding extension homomorphism $\psi$ factors through the inclusion $JD_i \subseteq JW_1$ and thus determines the pair $(Z_i, D_i)$. For $i=2,3$, the gluing map from $D_i\subseteq Z_i$ to $D_i\subseteq \hY$ is uniquely determined since $(D_i)^2_{Z_i}= 1$ and $(D_i)^2_{\hY}=  -1$. 

There is a natural extension of the root system $E_8 + E_8 + E_7\subseteq R(\Lambda)$ to a root system  $E_8 + E_8 + E_8\subseteq R(\Lambda)$: Let $\varepsilon = -\sigma +f$ viewed as an element of $H^2(\hY;\Zee)$. Let $\varepsilon'$ be an exceptional curve in $Z_1$ such that, in the labeling of \S\ref{anticanonTorellisub} for the $E_8$ root system in $[D_1]^\perp \subseteq H^2(Z_1;\Zee)$, $\varepsilon'\cdot \alpha_1 =1$ and $\varepsilon'\cdot \alpha_i =0$ for $2\le i \le 7$. Take $\alpha =  \varepsilon +\varepsilon'$. Then $\alpha^2 = -2$,  $\alpha \cdot D_i = \alpha \cdot L = 0$, $i=2,3$. Since  $(\varepsilon\cdot D_1)_{\hY} = (\varepsilon'\cdot D_1)_{Z_1}$,  $\alpha$ defines an element of $R(\Lambda)$. Thus $\{\alpha, \alpha_1, \dots, \alpha_7\}$ are the simple roots for a root system of type $E_8$ disjoint from the two others coming from $[D_i]^\perp \subseteq H^2(Z_i;\Zee)$, $i=2,3$.

\medskip
For all of the other cases in this paper, $k=2$. Thus conjecturally the image of the compactified period map $\widetilde{\Phi}$ at such a point $x$  does not contain any points corresponding to the $k=3$ case.  Note that  $\im N_1  \oplus \im N_2 \to W_1$ is an isomorphism over $\Q$. Moreover, $\widetilde{\Phi}(x)$ should  conjecturally record  the information of the spaces $\im N_i\subseteq W_1$, $i=1,2$.

\subsubsection{The Enriques case} In this case, $JW_1 \cong (JD_1\oplus JD_2)/\langle\eta\rangle $ in the notation of Lemma~\ref{EnriquesW1}. In particular, $JW_1$ is not the direct sum of $JD_1$ and $JD_2$, or equivalently the map $\im N_1  \oplus \im N_2 \to W_1$ is not surjective and hence is not an isomorphism. The lattice $\Lambda$ contains the orthogonal complement to $[D_i]$ in $H^2(Z_i; \Zee)$ and to $\{[D_1], [D_2]\}$ in $\overline{H}^2(\hY; \Zee)$ and hence to the orthogonal complement of $\{\overline{D}_1], [\overline{D}_2]\}$ in $\overline{H}^2(Y_0; \Zee)$, where $Y_0$ is the minimal model of $\hY$. Each of these lattices is isomorphic to the  negative definite unimodular lattice $\Lambda_{E_8}$. Hence $\Lambda \cong \Lambda_{E_8}^3$ and $R(\Lambda) = E_8+ E_8 + E_8$.

\subsubsection{The rational  case with multiplicities $(1,1)$}\label{11MHS}  In this case, $JW_1 \cong JD_1\oplus JD_2$ and the map $ \im N_1  \oplus \im N_2 \to W_1$ is   an isomorphism, which distinguishes this case Hodge-theoretically from the Enriques case. The lattice $\Lambda$ contains the orthogonal complement to $[D_i]$ in $H^2(Z_i; \Zee)$ and hence $R(\Lambda)$ contains two copies of $E_8$. By Theorem~\ref{Niem}, or directly, $R(\Lambda) = E_8+ E_8 + E_8$. Labeling the three copies of $E_8$ so that the first one corresponds to the orthogonal complement to $[D_1]$ in $H^2(Z_1; \Zee)$ and the second to the orthogonal complement to $[D_2]$ in $H^2(Z_2; \Zee)$, we clearly have $\psi_{12}=0$  and $\psi_{21}=0$. 

\subsubsection{The rational  case with multiplicities $(2,1)$}\label{21MHS} As in the  rational  case with multiplicities $(1,1)$, $JW_1 \cong JD_1\oplus JD_2$ and $\Lambda$ contains the orthogonal complement to $[D_i]$ in $H^2(Z_i; \Zee)$.  Hence $R(\Lambda)$ contains $E_8 + E_7$. Thus, by Theorem~\ref{Niem}, $R(\Lambda) = E_8+ E_8 + E_8$. However, in this case there is no labeling of the copies of $E_8$ for which $\psi_{12}= \psi_{21}=0$.

To see this last statement explicitly, by  Theorem~\ref{21case}, $\hY$ is the blowup of a rational elliptic surface with a multiple fiber $F$ of multiplicity $2$ at two points $p_1$ and $p_2$ which are the intersection of a smooth nonmultiple fiber $G$ and a smooth bisection $\Gamma$, and $D_1$ and $D_2$ are the proper transforms of $G$ and $\Gamma$ respectively. Here $G^2=0$ as it is a fiber, and hence $D_1^2 = -2$. Since $\Gamma^2 = 1$ by adjunction, as $\Gamma \cdot F = 1 = -\Gamma \cdot K_X$, it follows that $D_2^2 = -1$. Let $\phi_1$ and $\phi_2$ be the exceptional curves on $\hY$ corresponding to $p_1$ and $p_2$. It is straightforward to check using  Theorem~\ref{21case} that we can assume that $X$ is the blowup of $\Pee^2$ at $9$ distinct points, with corresponding exceptional curves $\varepsilon_1, \dots, \varepsilon_9$, and that 
\begin{align*}
D_1 &= 6h - 2\sum_{i=1}^9\varepsilon_i - \phi_1-\phi_2;\\
D_2 &= 3h -  \sum_{i=1}^8\varepsilon_i - \phi_1-\phi_2; \\
L &= 6h - 2\sum_{i=1}^8\varepsilon_i - \varepsilon_9 -\phi_1-\phi_2.
\end{align*}

There is an  $E_7$ root system   contained in $[D_1]^\perp \subseteq H^2(Z_1; \Zee)$ and there is an  $E_8$ root system   contained in $[D_2]^\perp \subseteq H^2(Z_2; \Zee)$.   Setting 
 $$\alpha_1 = \varepsilon_2 - \varepsilon_1, \alpha_2 = \varepsilon_3-\varepsilon_2, \dots, \alpha_7 = \varepsilon_8 - \varepsilon_7, \alpha_8 = h - \varepsilon_1-\varepsilon_2 - \varepsilon_3, $$
the $\alpha_i$ determine a root system of type $E_8$ in $\Lambda$. Finally, 
 let $\alpha = \phi_1 - \phi_2$.  Then $\alpha \cdot D_1= \alpha \cdot D_2 = \alpha \cdot L =0$, so that $\alpha \in \Lambda$.  Thus $R(\Lambda)$ contains the root system $E_8 + E_8 + E_7 + A_1$, and so $R(\Lambda) = E_8+ E_8 + E_8$ by Theorem~\ref{Niem}. In particular, the $E_7+ A_1$ subsystem is contained in a unique $E_8$ factor.  
 
 Label the three factors of $R(\Lambda) $ so that the first factor contains the $E_7+ A_1$ subsystem, the second factor is the $E_8$ root system   contained in $H^2(Z_2; \Zee)$, and the third factor is the $E_8$ determined by the $\alpha_i$ as above. Then $\psi_{21}=0$  and  $\psi_{22}\neq 0$. Also, $\psi_{1i}(\alpha) =\scrO_{D_i}(p_2-p_1)$, and hence is nonzero on both $JD_1$ and $JD_2$. It is easy to see that  $\psi_{32}\neq 0$. Thus the limiting mixed Hodge structures in this case are different from those described in \S\ref{11MHS}.
 
 Note: setting 
 $$\varepsilon = -\Big(3h -  \sum_{i=1}^8\varepsilon_i \Big) + \varepsilon_9 + \phi_1,$$
 one checks that $\varepsilon \cdot \alpha_i = 0$, $1\le i\le 8$,  $\varepsilon \cdot D_2 =   \varepsilon \cdot L =0$, and  $\varepsilon^2 =-1$. Let $\beta_1, \dots, \beta_7$ be a set of  simple roots for the $E_7$ root system on $Z_1$, labeled as in \S\ref{anticanonTorellisub}.  There exists an  exceptional curve $\varepsilon'$ on $Z_1$ such that $\varepsilon'\cdot \beta_1 =1$ and $\varepsilon'\cdot \beta_i =0$ for $i> 1$. Setting $\beta =\varepsilon+ \varepsilon'$, $\beta \in \Lambda$ and $\beta$ is orthogonal to the two $E_8$ root systems coming from $\{\alpha_1, \dots, \alpha_8\}$ and from $[D_2]^\perp \subseteq H^2(Z_2;\Zee)$. Thus we  can explicitly complete the $E_7$ root system in $R(\Lambda)$ coming from $[D_1]^\perp \subseteq H^2(Z_1;\Zee)$ to an $E_8$ root system in $R(\Lambda)$ with simple roots $\{\beta, \beta_1, \dots, \beta_7\}$.

 \subsection{More on the geometry in the $(2,2)$ case}  In  this section, we suppose that $\hY$ is a rational surface satisfying the conclusions of Theorem~\ref{22case}. Thus, there exists an exceptional curve $C$ on $\hY$ such that, if  $\rho_1\colon \hY \to Y_0^{(1)}$ is the contraction of $C$, then $(Y_0^{(1)}, D_2)$ is an anticanonical pair. In particular, $Y_0^{(1)}$ is the blowup of $\Pee^2$ at $11$ points  lying on the image of $D_2$, and hence $[D_2] = 3h -\sum_{i=1}^{11}\varepsilon_i$ for suitable exceptional curves. One way to find such curves is as follows: We have the double cover morphism $\nu\colon Y_0^{(1)} \to \mathbb{F}_1$, which is  branched along a smooth curve $\Sigma \in |2\sigma_0 + 6f|$. Then an easy calculation shows that 
 $$2g(\Sigma) - 2= \Sigma^2 + K_{\mathbb{F}_1}\cdot \Sigma = 6.$$
 Hence $g(\Sigma) = 4$. The projection $\mathbb{F}_1 \to \Pee^1$ exhibits $\Sigma$ as a branched double cover of $\Pee^1$. By Riemann-Hurwitz, there are exactly $10$ branch points. Equivalently, there are exactly $10$ fibers $f_1, \dots, f_{10}$ which are tangent to $\Sigma$. Thus $\nu^{-1}(f_i) = \varepsilon_i + \varepsilon_i'$, where the $\varepsilon_i$ are exceptional curves on $Y_0^{(1)} $ meeting transversally at one point. For $i\neq j$, the curves $\varepsilon_i $ and $ \varepsilon_i'$ are disjoint from both $\varepsilon_j$ and  $\varepsilon_j'$. Choosing $\varepsilon_1, \dots, \varepsilon_9$ arbitrarily, $Y_0^{(1)} $ is a blowup of the two point blowup of $\Pee^2$, or equivalently the one point blowup of $\mathbb{F}_0$. Since  $\varepsilon_{10} $ and $ \varepsilon_{10}'$ meet transversally at one point, there is a unique choice for $\varepsilon_{10}$ such that the blowdown of $Y_0^{(1)} $ along $\varepsilon_1, \dots, \varepsilon_{10}$ is $\mathbb{F}_1$. 
 
 The double cover morphism $\nu$ is induced by the linear system  $|\overline{D}_1|= |\nu^*(\sigma_0 + f)|$. The general member of  $|\nu^*(\sigma_0 + f)|$ is a smooth curve of genus $2$. Thus $\overline{D}_1$ corresponds to a member with an ordinary double point, or equivalently an irreducible, hence smooth member of $|\sigma_0 + f|$ which is tangent to $\Sigma$ at one point. Given $\overline{D}_1$, the exceptional curve $C$ and hence $\hY$ are determined by the double point of $\overline{D}_1$. There is a one dimensional family of curves in $|\sigma_0 + f|$ which are tangent to $\Sigma$ at one point, and for at least a general $\Sigma$ only finitely many of them will have a given $j$-invariant. The main point will be to distinguish these finitely many possibilities. 
 
 Next we determine the class $\gamma$ of $\overline{D}_1$ in terms of the basis $\{h, \varepsilon_1, \dots, \varepsilon_{11}\}$ of $H^2(Y_0^{(1)})$. By construction, $\gamma^2 = 2$ and $\gamma\cdot [D_2] =0$. Also, for $1\le i \le 10$, $\gamma \cdot \varepsilon_i = \nu^*(\sigma_0 + f)\cdot \varepsilon_i = 1$. Thus
 $$\gamma = ah -\sum_{i=1}^{10}\varepsilon_i + b\varepsilon_{11}.$$
 Using $\gamma\cdot [D_2] =0$ gives $3a+ b =10$. Then 
 $$\gamma^2 = 2 = a^2 -10 - b^2 = a^2 -10 - (10-3a)^2 = -8a^2 +60a -110.$$
 Equivalently, $a$ is an integer  and satisfies $8a^2 -60 a +112=4(a-4)(2a-7) =0$. Thus $a=4$ and $b=-2$. It follows that $[\overline{D}_1] = 4h - \sum_{i=1}^{10}\varepsilon_i - 2\varepsilon_{11}$ and that
 $$[D_1] = 4h - \sum_{i=1}^{10}\varepsilon_i - 2\varepsilon_{11} -2C.$$
 
 \subsection{The root system in the $(2,2)$ case}\label{22rootsystem}  For the rest of this section,   $\Lambda$ denotes the even negative definite unimodular lattice corresponding to the $(2,2)$ case.  We begin by completing the proof of Theorem~\ref{listroots} by showing:
 
 \begin{lemma}\label{22roots}  $R(\Lambda)$ is of type $E_7 + E_7 + D_{10}$, and $\Lambda_R$ has index $4$ in $\Lambda$.
 \end{lemma} 
 \begin{proof}
 First consider  $[D_i]^\perp\subseteq  H^2(Z_i; \Zee)$. Since $Z_i$ is a (possibly generalized) del Pezzo surface and $D_i^2 = 2$, $[D_i]^\perp$ is the root lattice for $E_7$ and by construction $[D_i]^\perp\subseteq \Lambda$. This construction yields two copies of $E_7$ inside $R(\Lambda)$. To find the remaining $D_{10}$, we use the standard description of $D_2$ and $D_1$ above. First, let $\alpha_i = \varepsilon_{i+1} - \varepsilon_i $, $1\leq i \leq 9$. Then define
 $$\alpha_{10} =  h - \varepsilon_9 - \varepsilon_{10} - \varepsilon_{11}.$$
 A calculation shows that $\alpha_{10}^2=-2$, $\alpha_{10}\cdot \alpha_8 = 1$ and 
 $$\alpha_{10} \cdot [D_2] = \alpha_{10}\cdot [D_1] = \alpha_{10}\cdot C =\alpha_{10}\cdot [ L]= 0.$$
 In particular, $\alpha_{10}\in R(\Lambda)$ and $\{\alpha_1, \dots, \alpha_{10}\}$ are the simple roots for a root system of type $D_{10}$ disjoint from the $E_7$ roots constructed above. Thus $R(\Lambda)$ contains a root system of  type $E_7 + E_7 + D_{10}$ and it is easy to see that it is contained in no larger simply laced root system of rank $24$. Thus $R(\Lambda)$ is of type $E_7 + E_7 + D_{10}$. 
 
 The connection index of the $E_7$ root system is $2$ and that of the $D_{10}$ root system is $4$. Hence the index of $\Lambda_R$ in $\Lambda$ is  $4$. 
 \end{proof}
 
 \begin{remark}\label{outer}  The curves $D_1$ and $D_2$ play a symmetric role. In this case, there is the anticanonical pair $(Y_0^{(2)}, D_1)$. The hyperplane class $h'$ is given by $2h-\varepsilon_{11} -C - \varepsilon_{10}$ and a system of $11$ disjoint exceptional curves is given by $\varepsilon_i' = \varepsilon_i$, $i\le 9$, $\varepsilon_{10}' = h-\varepsilon_{10} - \varepsilon_{11}$, and $\varepsilon_{11}' = h - \varepsilon_{10} - C$. Using this basis to compute a set of simple roots for the $D_{10}$ lattice, we see that $\alpha_i' = \alpha_i$, $i\le 8$, $\alpha_9' = \alpha_{10}$, and $\alpha_{10}' = \alpha_9$. This corresponds to the outer automorphism of $D_{10}$. 
 \end{remark} 
 
 There is an explicit representative for $\Lambda/ \Lambda_R$ (using the notation $\varpi_i$ for the fundamental weights of a root system as in \cite{Bourbaki}):
 
 \begin{lemma}\label{constructbeta}   Given  a choice of simple roots for for the $E_7$ root system for $[D_1']^\perp \subseteq H^2(Z_1;\Zee)$, there is  a unique $\beta_{11} \in \Lambda$ such that
 \begin{enumerate}
 \item[\rm(i)] $\beta_{11} \cdot \alpha_i = 0$, $i\neq 9$, and $\beta_{11} \cdot \alpha_9 = 1$, i.e.\  $\beta_{11}=\varpi_9$ for the $D_{10}$ root system. 
 \item[\rm(ii)] The image of $\beta_{11}$ in $H^2(Z_2;\Zee)$ is $0$. 
 \item[\rm(iii)] The image of $\beta_{11}$ in $H^2(Z_1;\Zee)$ is $\varpi_7$ for the $E_7$ root system for $[D_1']^\perp \subseteq H^2(Z_1;\Zee)$ and the labeling of the roots given in \cite[Planche VI]{Bourbaki}. 
 \end{enumerate}
 Finally, $\beta_{11}^2 =-4$ and the image of $\beta_{11}$ in $\Lambda/ \Lambda_R$ is a generator for $\Lambda/ \Lambda_R \cong \Zee/4\Zee$.
 \end{lemma} 
 \begin{proof} The class $\beta = \varepsilon_{11} - \varepsilon_{10}$ satisfies: $\beta  \cdot \alpha_i = 0$, $i\neq 9$, and $\beta \cdot \alpha_9 = 1$. However, $\beta\cdot [D_1] = \beta\cdot [L] = 1$. Thus $(\beta-[L]) \cdot [l] =0$. Let $\beta_{11} = \beta-[L] - \phi$, where $\phi $ is an exceptional curve in $Z_1$ inducing the weight $\varpi_7$ for an appropriate choice of simple roots. By construction, $\beta_{11} \cdot \xi_1=\beta_{11} \cdot \xi_1 =0$, so that $\beta_{11} \in \Lambda$, and 
 $$\beta_{11}^2 = (\beta-[L])^2 -1 = -4.$$
 The uniqueness of   $\beta_{11}$ is clear since the simple roots span $\Lambda$. 
 \end{proof}
 
\begin{remark} A possible strategy for proving Torelli using this boundary stratum would proceed as follows: First, the extension data determines the pairs $(Z_1, D_1')$ and $(Z_2, D_2')$. Next,   the  roots $\alpha_1, \dots, \alpha_{10}, \beta_{11}$ span a lattice $\Lambda'$ in $\Lambda$ which is ``formally an $E_{11}$ lattice," because $\beta_{11} \cdot \alpha_i = 0$, $i\neq 9$, and $\beta_{11} \cdot \alpha_9 = 1$.   Looking at the extension homomorphism $\psi_2\colon \Lambda' \to JD_2$ is the same as looking at the extension homomorphism for the anticanonical pair $(Y_0^{(1)}, D_2)$, because $\psi_2(\beta_{11}) = \varphi_{Y_0^{(1)}}(\varepsilon_{11} - \varepsilon_{10})$. Thus, using the global Torelli theorem for anticanonical pairs (Theorem~\ref{anticanonTorelli}), the pair $(Y_0^{(1)}, D_2)$ is determined up to isomorphism.  (Here, we need to use the fact that the line bundle $\scrO_{Y_0^{(1)}}(\overline{D}_1)$ is a nef and big line bundle on $Y_0^{(1)}$ to apply Lemma~\ref{genample} in order to check that Condition (ii) of Theorem~\ref{anticanonTorelli} is satisfied. To deal with the case where there are smooth rational curves of self-intersection $-2$, we also need to use Remark~\ref{moregenample}.)  In the notation of Lemma~\ref{constructbeta}, since $\phi  \cdot D_1'=1$, it is easy to check that translations of $D_1$ operate simply transitively on the class $\psi_1(\beta_{11})$, so that the gluing isomorphism $D_1\cong D_1'$ is uniquely determined by the extension homomorphism. By symmetry, the same is true for the isomorphism $D_2\cong D_2'$. 
 
 The main point is to identify $\overline{D}_1\subseteq Y_0^{(1)}$. There is a  $D_{10}$ lattice in $H^2(\hY)$ with basis $\alpha_i'= \varepsilon_{i+1} - \varepsilon_i $, $1\leq i \leq 8$, $\alpha_9'=\varepsilon_{10} - \varepsilon_9$,  $\alpha_{10}'= h'- \varepsilon_9-\varepsilon_{10}'-\varepsilon_{11}'$. Here, as noted in Remark~\ref{outer}, $\alpha_9'=\alpha_{10}$ and $\alpha_{10}' =\alpha_9$, realizing the outer automorphism of $D_{10}$. There is a corresponding ``formally $E_{11}$ lattice"  in $\Lambda$ with basis  $\alpha_i'$, $1\le i \le 10$ and $\beta_{11}'$, the analogue of the class constructed above but reversing the roles of $D_1$ and $D_2$. Let $q_i = \varepsilon_i '\cap D_1$, $1\le i \le 11$. The extension data for the root lattice with respect to $D_1$ determine the $q_i$ up to translation by a $3$-torsion point: $q_i \mapsto q_i + \xi$ with $3\xi = 0$.  In terms of the Jacobian $JD_1$, there are $11$ marked divisors of degree $0$. Ideally, we would like to physically identify the points $q_i$ themselves, for the following reason:  Referring back to the double cover picture, with $\nu \colon Y_0^{(1)} \to \mathbb{F}_1$, the image $S$ of $\overline{D}_1$ is a section of $|\sigma_0 + f|$, and $\overline{D}_1$ is determined by this section. For $i\le 9$, the image of $\varepsilon_i = \varepsilon_i'$ is a fiber $f_i$, and   $\nu(q_i) = S\cap f_i$. The section $S$ is disjoint from the negative section of $\mathbb{F}_1$. Blowing down the exceptional section, $S$ becomes a line in $\Pee^2$ not passing through the point $x$ that was blown up and $f_1$, $f_2$ become two lines passing through $p$. Then the image of $S$ in $\Pee^2$ is determined by the two points  $\nu(q_1), \nu(q_2)$. This would then determine  $S$, hence  $\overline{D}_1$ and finally $\hY$. 
 
 Without knowing the $q_i$, for a generic $\hY$ there are only finitely many possibilities for the nodal curve $\overline{D}_1\subseteq Y_0^{(1)}$, because in general the $j$-invariant for the family of elliptic curves which are the normalizations of the branched double over of a line in $\mathbb{F}_1$ tangent to the branch divisor at one point is nonconstant. It seems very likely that the additional information of the classes $\psi_1(\alpha_i) \in JD_1$ distinguishes the possibilities, but it is not yet clear how to use this information.
 \end{remark}
 
 \subsection{Some related Torelli problems} In this final section, we list some  Torelli type problems suggested by the above analysis as well as \cite{EGW}.  In the following,   let $X$ be a smooth surface and let $D$ be a smooth curve on $X$, not necessarily connected, or more generally a nodal curve. For simplicity, we will just consider the mixed Hodge structure on $H^2(X-D;\Zee)$, not the various related problems coming from normal crossing surfaces. The question is when this mixed Hodge structure determines the pair $(X,D)$, at least generically. Here are various examples of this situation:
 
 \begin{enumerate}
 \item $X$ is a rational elliptic surface with a section and $D = f_1 + f_2$ is a union of two smooth fibers, with the  period map  $\varphi \colon H^2_0(X; \Zee) \to Jf_1 \oplus Jf_2$ defined in the usual way. In this case,   $\varphi $  can have positive dimensional fibers: Consider a rational elliptic surface with an $\widetilde{E}_8$ fiber and a cuspidal fiber. Then the  surface $X$ has constant $j$-invariant, so that all smooth fibers are isomorphic. Hence $\varphi$ is trivial,  but the set of pairs $(X, f_1+ f_2)$ has dimension one.
 \item $X$ is a rational elliptic surface with a multiple fiber $F$ of multiplicity $m >1$ (for example of multiplicity $2$) and $D$ is a smooth (non-multiple) fiber. Thus $K_X = -F$ but  $D$ is not an anticanonical divisor and hence $(X,D)$ is not an anticanonical pair. 
 \item  $X$ is an Enriques surface and $D= D_1+ D_2$, where $D_1$ and  $D_2$ are two smooth elliptic curves  on $X$ intersecting transversally at one point. In this case, we have   the  period map  $\varphi \colon \overline{H}^2_0(X; \Zee) \to (JD_1 \oplus JD_2)/\langle \eta \rangle$ as in Lemma~\ref{EnriquesW1}, where $\eta$ is the nontrivial $2$-torsion line bundle on $X$ and we identify $\eta$ with its image via restriction in  $JD_1 \oplus JD_2$. 
 \end{enumerate}
 
 \appendix 
  \section{Simultaneous log resolution for simple elliptic singularities}\label{appendix}
  
  The goal of this appendix is to give a modular interpretation to the process of replacing a surface with simple elliptic singularities by a $d$-semistable model, also called a simultaneous log resolution. This problem has a long history. For a fixed elliptic curve $E$, this study dates    back to work of Looijenga \cite{Looi1}, \cite{Looi2} and  M\'erindol \cite{Mer}.    Grojnowski  and Shepherd-Barron \cite{GSB} as well as Davis \cite{DD1}, \cite{DD2} consider the problem from  a more group-theoretic point of view.  While the natural setting is that of (algebraic or analytic) stacks, we will mostly stick to a more naive, complex analytic treatment. To motivate the discussion, consider the following situation: Fix an elliptic curve $E$, and suppose that we are given
  \begin{enumerate}
  \item  A flat   proper morphism $\pi \colon \mathcal{Z} \to \Delta$, where $\mathcal{Z}$ is a  smooth complex threefold, such that, for $t\neq 0$, $\pi^{-1}(t)=Z_t$ is an almost del Pezzo surface, and $\pi^{-1}(0) $ is a normal crossing (and hence $d$-semistable) divisor of the form $R\amalg_DZ$, 
  where $R$ is a ruled surface over the elliptic curve $E$ with invariant $e> 0$, $Z$ is an almost del Pezzo surface with $R\cap Z = D$, where $D\cong E$ is the negative section of $R$ and is an anticanonical divisor in $Z$;
  \item A Cartier divisor $\mathcal{D} \subseteq \mathcal{Z}$, such that $\mathcal{D}  \cong E \times \Delta$,  $\mathcal{D}\cap \pi^{-1}(t)$ is an anticanonical divisor in $Z_t$ for $t\neq 0$, and $\mathcal{D}\cap \pi^{-1}(0)$ is a section of $R$ disjoint from the negative section.
  \end{enumerate}
  Here, we could replace the pair $(\Delta, 0)$ by an arbitrary pair $(S, S_0)$, where $S$ is an analytic space and $S_0$ is a Cartier divisor, with the natural changes to the definitions above. Then there exist two different contractions of $\mathcal{Z}$, given as follows:
  
For the first contraction,  standard arguments show that   $R^0\pi_*\scrO_{\mathcal{Z}}(k\mathcal{D})$ is locally free for all $k\ge 0$. Using a basis of sections for appropriate powers of $k$, i.e.\ taking the relative Proj
  $$\mathbf{Proj}_{\Delta}\bigoplus_{k\ge 0}R^0\pi_*\scrO_{\mathcal{Z}}(k\mathcal{D}),$$
   defines a birational morphism $\mathcal{Z} \to \mathcal{Z}_1 \subseteq \Pee \times \Delta$, where $\Pee$ is a weighted projective space, such that, for $t\neq 0$, the fiber of the induced morphism $\mathcal{Z}_1\to \Delta$ over $t$ is the anticanonical model of $Z_t$, and the fiber over $0$ is a (weighted) cone over $E$. 
   
   For the second contraction, the Cartier divisor $R\subseteq \mathcal{Z}$ satisfies: $\scrO_{\mathcal{Z}}(R)|f = \scrO_f(-1)$, where $f$ is a fiber of the ruling $R \to E$. Hence $R$ can be smoothly contracted to obtain another birational morphism $\mathcal{Z} \to \mathcal{Z}_2$, where all fibers of the induced morphism $\mathcal{Z}_2\to \Delta$ are almost del Pezzo surfaces and the image of $\mathcal{D}$ meets the fiber over $0$ in an anticanonical divisor. 
  
  We can partially reverse the second construction  as follows: Given a family  $\mathcal{Z}_2\to \Delta$, all of whose fibers are almost del Pezzo surfaces, and a divisor $\mathcal {D}_1\subseteq \mathcal{Z}_2$ with $\mathcal{D}_1  \cong E \times \Delta$ and such that $\mathcal{D}_1 $ restricts to an anticanonical divisor $D_t$ in every fiber, blow up the curve $D_0$ in the fiber $Z$ over $0$ and let $\mathcal{Z} \to \Delta$ be the new family. The exceptional divisor $R$ is then the ruled surface $\Pee(\scrO_E\oplus N_{D_0/Z})$ over $D$, and the fiber of $\mathcal{Z} \to \Delta$ over $0$ is $R\amalg_DZ$. On the other hand, the first construction is not in general reversible: Given a morphism $\mathcal{Z}_1\to \Delta$ and a divisor whose fibers are generalized del Pezzo surfaces away from $0$, and is the weighted cone over $D$ in $\mathbb{P}$ over $0$ (with weights corresponding to the del Pezzo surface fibers), we cannot simply blow up the vertex of the cone in the fiber over $0$, since for example if the total space of $\mathcal{Z}_1$ is smooth the exceptional divisor will be $\Pee^2$, not the appropriate del Pezzo surface, and it  will have multiplicity $> 1$ in the fiber over $0$. Thus a base change is necessary. 
  
  \begin{remark} As in \S\ref{ssect21}, we can also consider the deformation theory of the $d$-semistable surface $R\amalg_DZ$ and relate its deformation functor to that of the corresponding simple elliptic singularity.
  \end{remark} 
  
  Before we describe the general setup, we introduce the following notation: $E$ will always denote a fixed elliptic curve with a fixed origin $p_0$. 
  
  \begin{definition} For $4\le r \le 8$, let  $E_r$ be the  usual root system (where by convention $E_5=D_5$ and $E_4 = A_4$),   $W=W(E_r)$  the corresponding  Weyl group,  and  $Q=Q(E_r)$   the associated  root lattice, which we have previously denoted $\Lambda_{E_r}$ in \S\ref{latticessect}.  Let $\widehat{Q}$ denote the diagonal lattice  with basis $e_0, e_1, \dots e_r$ and such that $e_0^2 =1$ and $e_i^2 = -1$ for $i > 0$.  If  $\kappa = (3, -1, \dots, -1) \in \widehat{Q}$, then  $\kappa^2 =9-r=d$, where $d=9-r$ (here $d$ is the degree of the relevant del Pezzo  surfaces or the multiplicity $m$ of the corresponding simple elliptic singularities). Note that $\kappa^\perp \cong Q$ and that the intersection form sets up a perfect pairing  $(\widehat{Q}/\Zee\cdot \kappa)\otimes     Q\to \Zee$. 
  \end{definition} 
   
  We can then make the following definition:
   
   \begin{definition}\label{defmarked}  A \textsl{marked almost del Pezzo surface  with an anticanonical divisor isomorphic to   $E$} is a  quadruple $(Z,D, \mu, \psi)$, where $Z$ is an almost del Pezzo surface, $D$ is an anticanonical divisor, $\mu\colon D \to E$ is an isomorphism of lattices such that $\mu^*\scrO_E(dp_0) \cong N_{D/Z}$, and $\psi\colon H^2(Z; \Zee) \to \widehat{Q}$ is an isomorphism such that $\psi([D]) =\kappa$. A \textsl{marked generalized del Pezzo surface  with an anticanonical divisor isomorphic to   $E$} is defined similarly as a  quadruple $(Z,D, \mu, \psi)$, but where instead   $\psi\colon H^2(\hat{Z}; \Zee) \to \widehat{Q}$ is an isomorphism such that $\psi([D]) =\kappa$ for  the minimal resolution $\hat{Z} \to Z$. Of course, there is a bijection from the set of marked almost del Pezzo surfaces to the set of marked generalized del Pezzo surfaces. 
   \end{definition}

The set of  all marked almost del Pezzo surface  with an anticanonical divisor isomorphic to   $E$  is isomorphic to $E\otimes _{\Zee}Q \cong \Hom(\widehat{Q}/\Zee\cdot \kappa, E)$. In fact, the isomorphism is via the \textsl{extended period map} 
  $$\hat{\varphi}_Z \colon \widehat{Q}/\Zee\cdot \kappa \to E$$
  defined as follows: Given $\alpha \in \widehat{Q}$, $\psi^{-1}(\alpha) \in H^2(Z; \Zee)$ is represented by a unique line bundle $L_\alpha$ with, say, $\deg(L_\alpha|D) = a$. Then $(L_\alpha|D)\otimes \mu^*\scrO_E(-ap_0) $ is a line bundle of degree $0$ on $D$, and thus corresponds to an element of $JD \cong E$. Note that $L_\kappa = \scrO_Z(D)$ so that $\kappa$ is sent to $0\in E$. Thus there is a well-defined map $\hat{\varphi}_Z \colon \widehat{Q}/\Zee\cdot \kappa \to E$, which induces the usual period map up to the identification of $JD$ with $E$ via the inclusion $Q \to \widehat{Q}/\Zee\cdot \kappa$. The global Torelli theorem (Theorem~\ref{anticanonTorelli}) then easily implies that the quadruple $(Z,D, \mu, \psi)$ is specified by its extended period map. Note that the index of $Q$ in $\widehat{Q}/\Zee\cdot \kappa$ is $\kappa^2 = 9-r=d$, so that there are $d^2$ extensions of a homomorphism $Q \to E$ to a homomorphism $\widehat{Q}/\Zee\cdot \kappa \to E$; this is the same as the number of choices of the isomorphism $\mu\colon D \to E$ or equivalently the number of $d^{\textrm{th}}$ roots of  $N_{D/Z}$. Thus the abelian variety $E\otimes _{\Zee}Q$ is the (coarse) moduli space of  marked  almost del Pezzo surfaces. It is not a fine moduli space for marked almost del Pezzo surfaces owing to the existence of flops (elementary transformations) coming from almost del Pezzo surfaces with $-2$-curves. Nevertheless, \cite[2.3.3]{Mer} constructs  a  ``universal" family $\mathscr{V}\to E\otimes _{\Zee}Q$, together with a divisor $\mathcal{D}\subseteq \mathscr{V}$ isomorphic to $E\times (E\otimes _{\Zee}Q)$ such that, for each $x\in E\otimes _{\Zee}Q$, the pair $(V_x, D)$ is a marked almost del Pezzo surface in a natural sense, where $V_x$ is the fiber of $\mathscr{V}$ over $x$ and $D \cong E\times \{x\}$ is the corresponding divisor.  In fact,   such a family exists for each set of simple roots in $Q$. Finally, fixing a choice of  simple roots, the extended period map $E\otimes _{\Zee}Q \to E\otimes _{\Zee}Q$ is the identity. 
  
  \begin{remark} The Weyl group $W$ action on $E\otimes _{\Zee}Q$  does not extend to a holomorphic action on $\mathscr{V}$. Instead, as a consequence of the global Torelli theorem (Theorem~\ref{anticanonTorelli}) for families (see e.g\ \cite[\S{II}.3]{Looi3}, \cite[\S6]{GHK}), the Weyl group action extends to an action of $W$ on $\mathscr{V}$ by \textbf{birational} maps (not morphisms) on $\mathscr{V}$ covering the given action on $E\otimes _{\Zee}Q$. More precisely, there is a universal family $\overline{\mathscr{V}} \to E\otimes _{\Zee}Q$ of generalized del Pezzo surfaces together with a marking on $H^2$ of the minimal resolution (which we describe in more detail below) and the action of $W$ on $E\otimes _{\Zee}Q$  lifts to an action  on $\overline{\mathscr{V}}$.  Then $\mathscr{V}$ is a simultaneous resolution of the family  $\overline{\mathscr{V}} \to E\otimes _{\Zee}Q$ and $W$ acts simply transitively on the set of simultaneous resolutions constructed in this way.
\end{remark}

  \begin{remark} The choice of a set of simple roots $\alpha_1, \dots, \alpha_r$ leads to a choice of fundamental weights $\varpi_1, \dots, \varpi_r$. In particular, using the labeling of the simple roots of \S\ref{anticanonTorellisub}, $\varpi_1$ can be taken to be the class of an exceptional curve in $H^2(Z;\Zee)$ via the isomorphism $Q\spcheck  \cong  \widehat{Q}/\Zee\cdot \kappa$. However, despite the ordering of simple roots in \S\ref{anticanonTorellisub}, which was chosen for the purposes of \S\ref{22rootsystem}, for the realization of an almost del Pezzo surface as a blowup of $\Pee^2$, the exceptional curve corresponding to $\varpi_1$  is more naturally viewed as    the ``last" blowup. 
  \end{remark}

The link with the deformation theory of the pair $(Z,D)$   is given by the following: The Zariski tangent space to deformations of the pair $(Z,D)$ is given by $H^1(Z; T_Z(-\log D))$. If we want to fix the isomorphism class of $D$, the appropriate Zariski tangent space is given by 
$$\Ker\{H^1(Z; T_Z(-\log D))\to H^1(D; T_D)\} \cong H^1(Z; T_Z(-  D))/\im H^0(D; T_D).$$
 Since $T_Z(-D) = T_Z\otimes K_Z \cong \Omega^1_Z$, there is a further identification
 $$H^1(Z; T_Z(-  D))/\im H^0(D; T_D)\cong H^1(Z; \Omega^1_Z)/\Cee[D].$$
 In particular, $\dim H^1(Z; T_Z(-  D))/\im H^0(D; T_D) =r$. 

\begin{theorem}\label{KSmap1}  Let $x\in E\otimes _{\Zee}Q$ correspond to the pair $(Z,D)$. The Kodaira-Spencer map corresponding to the family $\mathscr{V}$ induces an isomorphism on tangent spaces 
\begin{align*}
T_{E\otimes _{\Zee}Q, x} &\xrightarrow{\cong} H^1(Z; T_Z(-  D))/\im H^0(D; T_D)\\
&\cong \Ker\{H^1(Z; T_Z(-\log D))\to H^1(D; T_D)\} .
\end{align*}
\end{theorem} 
\begin{proof} By construction, the extended period map is the identity map from $E\otimes _{\Zee}Q$ to itself.  The tangent space $T_{E\otimes _{\Zee}Q, x}$ is isomorphic to 
$$H^1(E; T_E) \otimes _{\Zee}Q \cong \Cee \otimes _{\Zee}Q \cong H^1(Z; \Omega^1_Z)/\Cee[D].$$
On the other hand, the differential of the (extended) period map  is a homomorphism
\begin{align*}
H^1(Z; T_Z(-  D))/\im H^0(D; T_D) &\to \Hom(H^0(Z; \Omega^2_Z(\log D)) ,H^1_0(Z; \Omega^1_Z(\log D))\\
&\cong H^1_0(Z; \Omega^1_Z(\log D)),
\end{align*}
where $H^1_0(Z; \Omega^1_Z(\log D)) = \Ker \{H^1(Z; \Omega^1_Z(\log D))\to H^1(D;\scrO_D)\}\cong H^1(Z; \Omega^1_Z)/\Cee[D]$. By a standard argument (cf.\ \cite[Theorem 3.16]{Friedanticanon}),  the differential of the period map factors through the Kodaira-Spencer homomorphism. Hence the Kodaira-Spencer homomorphism must be injective. Since 
$$\dim T_{E\otimes _{\Zee}Q, x} = \dim  H^1(Z; T_Z(-  D))/\im H^0(D; T_D) =r,$$
the Kodaira-Spencer homomorphism is an isomorphism.
\end{proof}

There is a coarse moduli space $M$ for generalized del Pezzo surfaces  with an anticanonical divisor isomorphic to $E$, or more precisely triples $(Z,D, \mu)$, where $Z$ is a generalized del Pezzo surface, $D$ is a smooth divisor in $|\omega_Z^{-1}|$, and $ \mu\colon D \to E$ is an isomorphism such that $\mu^*\scrO_E(dp_0) \cong N_{D/Z}$. The moduli space $M$  is a weighted projective space $\Pee = \Pee(1, h_1, \dots, h_r)$, where the highest root $\tilde \alpha$ for the root system $E_r$ is given in terms of a set of simple roots $\alpha_1, \dots, \alpha_r$ by: 
  $$\tilde\alpha = h_1\alpha_1 + \cdots + h_r\alpha_r.$$
  (We use the notation $h_i$ instead of the more usual notation $g_i$ to avoid confusion with the coefficient $g_2, g_3$ in the Weierstrass form of the equation for $E$.) 
  We shall describe $M$  explicitly below in case $d=1$, where the sequence $h_1, \dots, h_r$ is, up to permutation, $2,2,3,3,4,4,5,6$.  The family of pairs $(\mathscr{V}, \mathcal{D}) \to E\otimes _{\Zee}Q$ defines a $W$-invariant morphism $E\otimes _{\Zee} Q \to M$ by associating to each pair $(V_x, D)$ the corresponding pair $(\overline{V}_x, D)$, where $V_x$ is an almost del Pezzo surface and $\overline{V}_x$ is the corresponding generalized del Pezzo surface. This morphism in turn  induces an isomorphism $(E\otimes _{\Zee}Q)/W \cong  M$.  This last isomorphism then gives a concrete form to Looijenga's theorem that $(E\otimes _{\Zee}Q)/W \cong \Pee(1, h_1, \dots, h_r)$. However, to deal with the existence  of simultaneous log resolutions, we will need a more precise construction.

Turning to the theory of deformations of simple elliptic singularities, or equivalently weighted projective cones over a fixed elliptic curve $E$, let $(U,p)$ be the germ of a simple elliptic singularity, with minimal resolution $\pi\colon \widetilde{U}\to U$ and exceptional divisor $E = \pi^{-1}(p_0)$. Define $d = -E^2$ to be the \textsl{multiplicity} of $U$, so that $d=m$ is the usual multiplicity unless $d=1$, and let $r=9-d$.  We assume in what follows that $d \le 4$, so that in particular $U$ is a local complete intersection.  The  tangent space  $H^0(U; T^1_U)$ to the  deformation functor of $U$ has a $\Cee^*$ action with all weights nonpositive. The negative weight space $H^0(U; T^1_U)^-$ has dimension $10-d = r+1$ and the quotient $\mathbb{P} = \Big(H^0(U; T^1_U)^- -\{0\}\Big)/\Cee^*$ is a weighted projective space of dimension $r$. The  weight zero space $H^0(U; T^1_U)^0$ has dimension one and corresponds to deforming $E$.  For our purposes, it is better to consider the globalized version of the above:  We  replace $U$ by the weighted projective cone $\overline{R}$. Equivalently,  $R$ is the ruled surface $\Pee(\scrO_E\oplus \lambda)$,  where $\lambda$ is a fixed line bundle over $E$ of degree $d$ which we may as well assume is $\scrO_E(dp_0)$, and $\overline{R}$ is the contraction of $R$ along the negative section. Let $D$ be a fixed section of $R$ of square $d$ disjoint from the negative section   and consider deformations of the pair $(\overline{R}, D)$ fixing the isomorphism $D\to E$, or more precisely triples $(Z, t, \varphi)$, where $Z$ is a generalized del Pezzo surface or the cone over $E$, $t\in H^0(Z; \omega_Z^{-1})$ is a nonzero section defining a Cartier divisor $D$, and $\varphi\colon D \to E$ is an isomorphism with $\varphi^*\lambda \cong N_{D/Z}$.  The $\Cee^*$-action corresponds to multiplying the section $t$ by an element of $\Cee^*$. 

If  $S$ is  the  base space of the miniversal deformation of $U$, then the $\Cee^*$-action on $U$ extends to the germ of a $\Cee^*$-action on $S$. The equisingular locus  $S_{\text{\rm{es}}}\subseteq S$  corresponds to deforming the elliptic curve $E$ but otherwise keeping the singularity a simple elliptic singularity. Then  $S_{\text{\rm{es}}}=S^0$ is identified with the base space of the Kuranishi family of deformations of the elliptic curve $E$ and  the tangent space $T_{S_{\text{\rm{es}}}, s}$ to   $S_{\text{\rm{es}}}$ at a point $s$ is the weight zero subspace $H^0(U; T^1_U)^0$ and is identified with $H^1(E; T_E)$. There is a noncanonical isomorphism $S \cong S^{-}\times S_{\text{\rm{es}}}$, corresponding to the splitting $H^0(U; T^1_U) \cong H^0(U; T^1_U)^-\oplus H^0(U; T^1_U)^0$. More canonically, there is a morphism $S\to S_{\text{\rm{es}}}=S^0$ whose fiber is $S^-$, corresponding to the inclusion of the subring of $\Cee^*$-invariant functions  on $S$.

By \cite[Th\'eor\`eme 6.1]{Mer} there exists an affine  cone $\mathscr{C}$ over  the abelian variety $E\otimes _{\Zee}Q$ and a $\Cee^*$-equivariant isomorphism
$$\mathscr{C}/W \cong H^0(U; T^1_U)^-.$$
(In \cite{Mer}, $E\otimes _{\Zee}Q$ is denoted by $H$ and $\mathscr{C}$ by $C_H$.) Here $\mathscr{C}$ is the contraction of the zero section in the total space of a negative $W$-linearized line bundle $\mathscr{L}$ over $E\otimes _{\Zee}Q$. 
Letting $0$ denote the vertex of the cone $\mathscr{C}/W$,  the induced morphism 
$$(\mathscr{C}/W -\{0\})/\Cee^* \to \Big(H^0(U; T^1_U)^- -\{0\}\Big)\Big/\Cee^*$$
identifies $\Big(H^0(U; T^1_U)^- -\{0\}\Big)\Big/\Cee^*$ with $ (E\otimes _{\Zee}Q)/W$ where both sides are the  coarse moduli space  of  (unmarked)  almost del Pezzo surfaces. 
More precisely, there is a ``universal"  family $\mathscr{F}$ over $\mathscr{C}$ whose fiber at a point $x\neq v$  is an almost del Pezzo surface corresponding to the image of $x$ in $ E\otimes _{\Zee}Q$ via the projection $\mathscr{C} -\{0\} \to  E\otimes _{\Zee}Q$ and whose fiber over $0$   is $\overline{R}$, the  weighted projective  cone over $E$. 

The relation with weighted blowups is as follows. Let $p\colon \widehat{\mathscr{C}}\to  \mathscr{C}$ be the blowup of $\mathscr{C}$ at the vertex $0$, i.e.\ the total space of the negative line bundle $\mathscr{L}$, with exceptional divisor (the zero section) $\mathcal{A}\cong E\otimes _{\Zee}Q$. Then there is a morphism $q\colon \widehat{\mathscr{C}} \to E\otimes _{\Zee}Q$. Let $ q^*\mathscr{V}$ be the pullback of the   family $\mathscr{V}\to E\otimes _{\Zee}Q$ and let $\pi\colon  q^*\mathscr{V} \to  \widehat{\mathscr{C}}$ be the corresponding morphism. The divisor $\mathcal{D}\subseteq \mathscr{V}$ pulls back to a divisor $q^*\mathcal{D} = \widetilde{\mathcal{D}}$  on $ q^*\mathscr{V}$. Let $\widetilde{\mathcal{D}}_0$ be its restriction to the preimage $\pi^{-1}(\mathcal{A}) \subseteq q^*\mathscr{V}$. Here there is a commutative diagram
$$\begin{CD}
\pi^{-1}(\mathcal{A}) @>{\cong}>> \mathscr{V} \\
@VVV @VVV \\
\mathcal{A} @>{\cong}>>   E\otimes _{\Zee}Q.
\end{CD}$$
Blow up the smooth codimension $2$ subvariety $\widetilde{\mathcal{D}}_0$ of $q^*\mathscr{V}$  and let $\widehat{\mathscr{Z}} \to \widehat{\mathscr{C}}$ be the resulting family. Thus $\widehat{\mathscr{Z}} \to \widehat{\mathscr{C}}$ is a family of surfaces over $\widehat{\mathscr{C}}$ whose fiber over a point $x$ in the smooth divisor $\mathcal{A}$ is reduced and is isomorphic to the normal crossing surface $V_x\amalg _ER$, where as above $R$ is the ruled surface over $E$ containing $E$ as a section with $(E^2)_R = -d$.  Away from $\mathcal{A}$,  the fibers  of $\widehat{\mathscr{Z}} \to \widehat{\mathscr{C}}$ are almost del Pezzo surfaces. As in the discussion at the beginning of the appendix, the   proper transform of  $\pi^{-1}(\mathcal{A})$ becomes exceptional in the sense that  it can be contracted to a subvariety isomorphic to $\mathcal{A}$. Doing so gives a family of surfaces over $\widehat{\mathscr{C}}$ which is the pullback to $\widehat{\mathscr{C}}$ of the family $\mathscr{F}\to \mathscr{C}$. In particular, the fibers over $\mathcal{A}$ are now all isomorphic to the cone $\overline{R}$. The Weyl group $W$ acts on $\mathscr{L}$ and hence on  $\widehat{\mathscr{C}}$ and the morphism $p\colon \widehat{\mathscr{C}}\to  \mathscr{C}$ is $W$-equivariant.   Thus there is a commutative diagram
$$\begin{CD}
\widehat{\mathscr{C}} @>>> \mathscr{C} \\
@VVV @VVV \\
\widehat{\mathscr{C}} /W @>>>   \mathscr{C}/W.
\end{CD}$$
Here, $\mathscr{C}/W$ is the affine space $H^0(U; T^1_U)^-$ and the fiber of $\widehat{\mathscr{C}} /W \to   \mathscr{C}/W$ over $0$ is $(E\otimes _{\Zee}Q)/W$, which by Looijenga's theorem is the weighted projective space $\Pee(1, h_1, \dots, h_r)$ corresponding to the $\Cee^*$-quotient of $H^0(U; T^1_U)^-$ by the natural $\Cee^*$-action. Thus $\widehat{\mathscr{C}} /W$ can be viewed as the weighted blowup $\widetilde{\Cee}^{r+1}$ of the affine space $\Cee^{r+1}$ at the origin, with exceptional divisor the weighted projective space $\Pee(1, h_1, \dots, h_r)$. Note that both $\widetilde{\Cee}^{r+1}$ and $\Pee(1, h_1, \dots, h_r)$ are orbifolds and so are themselves stacks in a natural way, but these stacks are not the same as the stacks $[\widehat{\mathscr{C}} /W ]$ or    $[(E\otimes _{\Zee}Q)/W]$.

The deformation theory of $U$  behaves well in families: If $\mathcal{E}\to T$ is a family of elliptic curves,  there  is a  bundle $\mathbf{A}$ of affine spaces  of dimension $r+1$ over $T$  together with a zero section and a $\Cee^*$-action which restricts to the vector space $H^0(U; T^1_U)^-$ over every fiber. 

\begin{remark} For more discussion on the  construction of $\mathbf{A}$  see for example \cite{Wirth} or \cite[Corollary 4.1.7]{FM2001}. Note however that the case of $E_8$, i.e.\ the case $d=1$, is somewhat exceptional, in the sense that the bundle of affine spaces over $T$ need not be the total space of a vector bundle (and is therefore not treated in \cite{Wirth}). Roughly speaking, the issue is that there are two vector bundles over the open sets corresponding to $g_2\neq 0$ and $g_3\neq 0$, where $g_2$ and $g_3$ are the coefficients for the Weierstrass equation  for $\mathcal{E}$, but they are glued together by an isomorphism of affine bundles which is   not linear  on the fibers. We will return to this point shortly. 
\end{remark}

 M\'erindol's  construction can then be done in families: given a family $\mathcal{E}\to T$ of elliptic curves (with a section), we can form the families $\mathcal{E}\otimes _{\Zee}Q$ and $(\mathcal{E}\otimes _{\Zee}Q)/W$, as well as the relative cone $\mathcal{C}$ and its blowup $\widetilde{\mathcal{C}}$. The construction goes through as before (possibly after shrinking $T$).
In particular, the   construction applies to case where $\mathcal{E}\to S_{\text{\rm{es}}}$ is the miniversal deformation  of the elliptic curve $E$.   This picture globalizes to a bundle $\mathbf{A}$ of affine spaces over $S_{\text{\rm{es}}}$ together with the  zero section, which we can identify with  $S_{\text{\rm{es}}}$.  In this case,  $\mathbf{A}$ is the base space of a miniversal family for the corresponding simple elliptic singularities, and the section $S_{\text{\rm{es}}}$ is the base space for the family of equisingular deformations.  In particular,   there is a family $\mathcal{V} \to \mathcal{E}\otimes _{\Zee}Q$ and the analogue of Theorem~\ref{KSmap1} holds:

\begin{theorem}\label{KSmap2} With $\mathcal{E}\to S_{\text{\rm{es}}}$ the germ of  a miniversal deformation of the elliptic curve $E$, let $x\in \mathcal{E}\otimes _{\Zee}Q$ correspond to the pair $(Z,D)$. Then the  Kodaira-Spencer map for the family $\mathcal{V}$ induces an isomorphism on tangent spaces $T_{\mathcal{E}\otimes _{\Zee}Q, x} \to H^1(Z; T_Z(-  \log D))$.  \qed
\end{theorem}

There is   a bundle  $\mathcal{C}$ of affine  cones  over  $S_{\text{\rm{es}}}$, and a $\Cee^*$-equivariant isomorphism (of spaces over $S_{\text{\rm{es}}}$)
$$\mathcal{C}/W \cong  \mathbf{A}.$$
If $\widetilde{\mathbf{A}} \to \mathbf{A}$ is the weighted blowup along the curve $S_{\text{\rm{es}}}$, then there is a cover $\widehat{\mathbf{A}} \to \widetilde{\mathbf{A}}$ with covering group the Weyl group $W$ and a family $\widehat{\mathcal{Z}}$ of $d$-semistable surfaces over $\widehat{\mathbf{A}}$ which replicates the  above construction on every fiber. More precisely, we have the following theorem:

\begin{theorem}\label{mainappthm}  There exists a smooth variety $\widehat{\mathbf{A}}$ of dimension $r+2$,   a finite morphism $\widehat{\mathbf{A}} \to \widetilde{\mathbf{A}}$ with covering group  $W$, and a family $\widehat{\mathcal{Z}}\to\widehat{\mathbf{A}}$, such that the following hold:
\begin{enumerate}
\item[\rm(i)] The inverse image of the section $S_{\text{\rm{es}}}\subseteq \mathbf{A}$ in $\widetilde{\mathbf{A}}$ is the  weighted projective space bundle $(\mathcal{E}\otimes _{\Zee}Q)/W$ and its inverse image in $\widehat{\mathbf{A}}$ is the smooth divisor $\mathcal{E}\otimes _{\Zee}Q$. 
\item[\rm(ii)] The space $\widehat{\mathcal{Z}}$ is smooth,  and the fibers of $\widehat{\mathcal{Z}}\to\widehat{\mathbf{A}}$ are $d$-semistable over $\mathcal{E}\otimes _{\Zee}Q $  and are almost del Pezzo surfaces over $\widehat{\mathbf{A}} - \mathcal{E}\otimes _{\Zee}Q$. 
\qed
\end{enumerate}
\end{theorem} 

\begin{remark}\label{logresremark}  In the application in \S\ref{ssectsevsdss}, we will work with the germ $S$ of the space $\mathbf{A}$ at a point on the zero section, i.e.\ corresponding to a simple elliptic singularity, as well as the corresponding weighted blowup $\widetilde{S}$ and the Weyl cover $\widehat{S} \to \widetilde{S}$. Note that, by construction, $S$ is a germ but that $\widetilde{S}$ and $\widehat{S}$ are not: $\widetilde{S}$  contains the  weighted projective space bundle $(\mathcal{E}\otimes _{\Zee}Q)/W$ over $T$ and $\widehat{S}$ contains a smooth hypersurface isomorphic to $\mathcal{E}\otimes _{\Zee}Q$.
\end{remark}

There is a very concrete description of this process (following unpublished notes of the first author and  John Morgan \cite{FM2005}). For simplicity, we just discuss the case $d=1$, i.e.\ $r=8$. Consider the set of all equations of generalized del Pezzo surfaces $Z$ in the weighted projective space $\Pee(1,1,2,3)$ (with homogeneous coordinates $(x,y,z,t)$  of weights $2,3,1,1$ respectively)  of weighted degree $6$ and with a fixed hyperplane section isomorphic to $E$ and defined by a fixed $t\in H^0(Z; \omega_Z^{-1})$. These equations are of the form 
$$y^2 = x^3 + g_2xz^4 + g_3z^6 + tP_5(x,y,z) + t^2P_4(x,y,z)
+t^3P_3(x,y,z) +t^4P_2(x,z) + et^5z + ft^6,$$
where the $P_i$ are weighted homogeneous of degree $i$ in the appropriate
variables. 
%(By convention, we work here with positive weights.) 
Here $x,y,z$ are determined up to weighted homogeneous
changes of coordinates which are the identity modulo $t$,
i.e.\  up to the following transformations:
\begin{eqnarray*}
x &\mapsto & x + \alpha _1tz + \alpha _2t^2;\\
y &\mapsto & y + \beta _1tx + \beta _2tz^2 +\beta _3t^2z + \beta
_4t^3;\\
z &\mapsto & z + \gamma t.
\end{eqnarray*}
Using the first two transformations, there are unique choices of $x$
and
$y$ so that the equation involves no terms of the form $t^3y,t^2yz, txy,
tyz^2,  tx^2z, t^2x^2$. Equivalently, we can eliminate
all terms which have a factor of the form  $ty$ or $tx^2$.  To use the third transformation to  eliminate one more term depends in this case on whether $g_2\neq 0$ or $g_3\neq 0$ (this issue only appears for  $E_8$, i.e. $d=1$). For example, if $g_2\neq 0$, we can eliminate the term $txz^3$ and are left with the following equation in standard form: 
$$y^2 = x^3 + g_2xz^4 + g_3z^6 + atz^5 + b_1t^2xz^2 + b_2t^2z^4 +
c_1t^3xz + c_2t^3z^3 + d_1t^4x + d_2t^4z^2 + et^5 z + ft^6.$$
A parallel discussion handles the case where $g_3\neq 0$. 

Let $\Cee^9$ have coordinates $a, b_1, b_2, c_1, c_2, d_1,d_2, e, f$ together with the $\Cee^*$-action  where the weights are $(1,2,2,3,3,4,4,5,6)$. (It is more convenient here to use positive weights for the $\Cee^*$-action on $\Cee^9$.) 
The above equation defines a family $\mathcal{Z}_0 \subseteq \Cee^9\times \Pee(1,1,2,3)$, viewed as a family of surfaces over $\Cee^9$, whose fiber over a nonzero point is a generalized del Pezzo surface and whose fiber over $0$ is the weighted projective cone over $E$. The family $\mathcal{Z}_0 \to \Cee^9$ is then the globalized version of the universal family of negative weight deformations  of the simple elliptic singularity of multiplicity one. There are two different $\Cee^*$-actions on $\mathcal{Z}_0$: the trivial action, which covers the trivial action on $\Cee^9$, and the $\Cee^*$-action which is the given action on $\Cee^9$ and where $s\in \Cee^*$ acts on the homogeneous coordinates $(x,y,z,t)$ for $\Pee(1,1,2,3)$  by: $s\cdot (x,y,z,t) = (x,y,z,s^{-1}t)$. 

Next we consider the weighted blowup of $\Cee^9$ at the origin: consider the space
$\Cee \times (\Cee^9-\{0\})$ together with the $\Cee^*$-action defined by 
$$s\cdot (\tau, v) = (s^{-1}\tau, s\cdot v).$$
Then $\pi_2\colon \Cee \times (\Cee^9-\{0\}) \to \Cee^9-\{0\}$ is $\Cee^*$-equivariant for the given actions. Define 
$$F\colon \Cee \times (\Cee^9-\{0\}) \to \Cee^9$$ by
 $F(\tau, v)  = \tau\cdot v$. 
By construction, $F(s\cdot (\tau, v)) = F(\tau, v)$, i.e.\ $F$ is equivariant for the given action on $\Cee \times (\Cee^9-\{0\})$ and the trivial action on $\Cee^9$. The quotient stacks satisfy: there exist morphisms $[(\Cee \times (\Cee^9-\{0\})/\Cee^*] \to [(\Cee^9-\{0\})/\Cee^*]$ and $[(\Cee \times (\Cee^9-\{0\})/\Cee^*] \to \Cee^9$. The second morphism realizes the coarse moduli space of $[(\Cee \times (\Cee^9-\{0\})/\Cee^*]$ as the weighted blowup  $\widetilde{\Cee}^9$ of $\Cee^9$ and the first corresponds to  the associated morphism on coarse moduli spaces   $\widetilde{\Cee}^9 \to \Pee(1, 2,2,3,3,4,4,5, 6)$. 

Note that $F^{-1}(0) = \{0\}\times (\Cee^9-\{0\})$ is a smooth divisor and that $\Cee^*$ acts freely on 
$$\Big(\Cee \times (\Cee^9-\{0\})\Big) - F^{-1}(0) = \Cee^* \times (\Cee^9-\{0\}).$$ There are two families over $\Cee \times (\Cee^9-\{0\})$: $\mathcal{Z}_1 = F^*\mathcal{Z}_0$, defined by
$$y^2  = x^3 + g_2xz^4 + g_3z^6 + a(\tau t) z^5 + b_1(\tau t) ^2xz^2 + \cdots + f(\tau t) ^6,$$
and $\mathcal{Z}_2  = \pi_2^*(\mathcal{Z}_0|\Cee^9-\{0\})$, defined by
$$y^2  = x^3 + g_2xz^4 + g_3z^6 + atz^5 + b_1t^2xz^2 + \cdots  + ft^6.$$ 
Thus both $\mathcal{Z}_1$ and $\mathcal{Z}_2$ are hypersurfaces in $\Cee \times (\Cee^9-\{0\}) \times \Pee(1,1,2,3)$. The two different $\Cee^*$-actions on $\mathcal{Z}_0$ then lift to actions on $\mathcal{Z}_1$ and $\mathcal{Z}_2$ respectively: the trivial action on $\mathcal{Z}_0$  lifts to the action defined on points $(\tau, v, x,y,z,t)\in \mathcal{Z}_1$   by
$$s\cdot (\tau, v, x,y,z,t) =  (s^{-1}\tau, s\cdot v, x,y,z,t),$$
and the second action on $\mathcal{Z}_0$ lifts to 
the action defined on points $(\tau, v, x,y,z,t)\in \mathcal{Z}_2$  by
$$s\cdot (\tau, v, x,y,z,t) =  (s^{-1}\tau, s\cdot v, x,y,z,s^{-1}t).$$
All fibers of $\mathcal{Z}_2$ are generalized del Pezzo surfaces. For $\mathcal{Z}_1$, the fibers over $\Cee^* \times (\Cee^9-\{0\})$ are generalized del Pezzo surfaces but the fibers over the divisor 
$F^{-1}(0)$ are all isomorphic to the weighted cone over $E$. More precisely, $\mathcal{Z}_1$ is the pullback to $\Cee \times (\Cee^9-\{0\})$ via $F^*$ of the negative weight miniversal deformation  of the cone (or more precisely of the pair $(\overline{R},D)$). To compare the two families, we have the following:

\begin{lemma} Over $\Cee^* \times (\Cee^9-\{0\})$, there is a $\Cee^*$-equivariant isomorphism from $\mathcal{Z}_1|\Cee^* \times (\Cee^9-\{0\})$ to  $\mathcal{Z}_2|\Cee^* \times (\Cee^9-\{0\}$. More precisely, the family $\mathcal{Z}_1$ is obtained from the family $\mathcal{Z}_2$ as follows: Blow  up the smooth codimension $2$ subvariety  defined by $\tau = t =0$, which is isomorphic to $F^{-1}(0) \times E$, and then blow down the proper transform of $\mathcal{Z}_2|F^{-1}(0)$ along $F^{-1}(0)$. 
\end{lemma} 
\begin{proof} Over $\Cee^* \times (\Cee^9-\{0\})\times \Pee(1,1,2,3)$, there is a $\Cee^*$-equivariant isomorphism   $$B\colon \mathcal{Z}_2|\Cee^* \times (\Cee^9-\{0\})\to  \mathcal{Z}_1|\Cee^* \times (\Cee^9-\{0\}$$  defined by
$$B(\tau, v,x,y,z,t) =  (\tau, v,x,y,z,\tau^{-1}t).$$
 The rational map $\mathcal{Z}_2 \dasharrow  \mathcal{Z}_1$ defined by $B$ is resolved by a single blowup along the smooth codimension $2$ subvariety   $\tau = t =0$ and the corresponding morphism is as described. (This is just a statement about the rational map $\Cee^2 \dasharrow \Cee^2$ defined by $(\tau, t) \mapsto (\tau, t/\tau)$.)
\end{proof}

As in the lemma, let $\overline{\mathcal{Z}}$ be the blow  up of $\mathcal{Z}_2$ along the smooth codimension $2$ subvariety    $\tau = t =0$. Then $\overline{\mathcal{Z}}$ dominates both $\mathcal{Z}_1$ and $\mathcal{Z}_2$, and the fibers of the morphism $\overline{\mathcal{Z}} \to \Cee \times (\Cee^9-\{0\}) $ are generalized del Pezzo surfaces away from $F^{-1}(0)$ and are $d$-semistable surfaces $Z\amalg_DR$ over $F^{-1}(0)$, except that $Z$ might have RDP singularities. Also, the $\Cee^*$-action on $\Cee \times (\Cee^9-\{0\}) $ lifts to a $\Cee^*$-action on $\overline{\mathcal{Z}}$.  However, there exist points for this $\Cee^*$-action with nontrivial finite isotropy. Thus we cannot simply take the quotient by $\Cee^*$. To remedy this situation, we take the  finite cover by  imposing  level structure, i.e.\ by considering marked generalized del Pezzo surfaces  $(Z,D,\mu, \psi)$ as defined in Definition~\ref{defmarked}. The moduli space of such quadruples is then $E\otimes_{\Zee}Q$ and there is a morphism  $E\otimes_{\Zee}Q \to \Pee(1, 2, \dots, 6)$.  Let   $\widetilde{\Cee}^9$ denote  as before the weighted blowup  of the affine space $\Cee^9$ at the origin. Then there is a morphism $\widetilde{\Cee}^9 \to \Pee(1, 2, \dots, 6)$ of analytic spaces. We can thus form the normalization $\widetilde{\mathscr{C}}$ of the fiber product $(E\otimes_{\Zee}Q)\times_{\Pee(1, 2, \dots, 6)}\widetilde{\Cee}^9$. Equivalently, the  $W$-linearized ample line bundle $\mathscr{L}$ on $E\otimes_\Zee Q$ constructed in \cite{Mer}   leads to a $\Cee^*$-equivariant morphism $\mathscr{C} -\{0\} \to \Cee^9-\{0\}$. There is also a $\Cee^*$-equivariant morphism
$$\Cee\times (\mathscr{C} -\{0\} ) \to \Cee \times (\Cee^9-\{0\}),$$
where $\Cee^*$ acts freely on $\Cee\times (\mathscr{C} -\{0\} ) =\Cee \times (\mathbb{V}(\mathscr{L})- \textrm{the zero section})$ and the quotient is the space $\widetilde{\mathscr{C}}=\mathbb{V}(\mathscr{L})$. The family $\overline{\mathcal{Z}}$ on $\Cee \times (\Cee^9-\{0\})$ pulls back to a family on $\Cee\times (\mathscr{C} -\{0\} )$ on which $\Cee^*$ acts freely. Taking the quotient gives a family, which we  denote  by $\overline{\mathscr{Z}}$, over $\widetilde{\mathscr{C}}$. Here, the family $\overline{\mathscr{Z}}$ is a simultaneous log resolution as before, except that the reducible fibers (i.e.\ those over the exceptional divisor) may have RDP singularities disjoint from the double locus, and likewise the irreducible fibers may have RDP singularities. Then the family $\widetilde{\mathscr{Z}}\to \widetilde{\mathscr{C}}$ constructed above is a simultaneous resolution of the RDP singularities appearing in the family $\overline{\mathscr{Z}} \to \widetilde{\mathscr{C}}$, so that every fiber is smooth or $d$-semistable. In particular $\widetilde{\mathscr{Z}}\to \widetilde{\mathscr{C}}$ is a simultaneous log resolution of the universal family of negative weight deformations $\mathcal{Z}_0 \to \Cee^9$. The picture is described by the following commutative diagram:

$$\xymatrix{
&\widetilde{\mathscr{Z}}  \ar[d]  \ar@/^/[drr] &{} &{}
  \\
&\overline{\mathscr{Z}} \ar[d]  \ar[rr]&{}    &\mathcal{Z}_0\ar[d]\\
& \widetilde{\mathscr{C}} \ar[r]  &{} \widetilde{\Cee}^9  \ar[r] &\Cee^9 }$$
%$$\begin{CD}
%\widetilde{\mathscr{Z}}  @. {} @. {}\\
%@VVV @. @. \\
%\overline{\mathscr{Z}} @. {} @.  \mathcal{Z}_0\\
%@VVV @. @VVV\\
%\widetilde{\mathscr{C}} @>>> \widetilde{\Cee}^9 @>>> \Cee^9,
%\end{CD}$$
where $\overline{\mathscr{Z}}$ dominates the pullback of $\mathcal{Z}_0$ to $\widetilde{\mathscr{C}}$ and $\widetilde{\mathscr{Z}}$ is a simultaneous resolution of $\overline{\mathscr{Z}}$. 

A similar discussion handles the case of a family $\mathcal{E}\to T$ (with some care because of the issue with the vanishing of $g_2$ or $g_3$). In particular, taking $\mathcal{E}\to S_{\text{\rm{es}}}$ to be the Kuranishi family gives a simultaneous log resolution of the universal family of all deformations of the simple elliptic singularity.

\begin{remark} In terms of stacks, there is a sequence of morphisms
$$E\otimes_\Zee Q \to [(E\otimes_\Zee Q)/W] \to [(\Cee^9-\{0\})/\Cee^*]$$
such that $[(E\otimes_\Zee Q)/W] \to [(\Cee^9-\{0\})/\Cee^*]$ induces  an isomorphism on coarse moduli spaces. Similar results hold for the weighted blowups and for the families we have constructed over them.  
(Compare \cite[\href{https://stacks.math.columbia.edu/tag/044U}{Tag 044U}]{stacks-project}.)  
\end{remark}

 \bibliography{Isurfaces}

\end{document}